\documentclass[11pt]{amsart} \textwidth=14.5cm \oddsidemargin=1cm
\usepackage{amsmath}
\usepackage{amsxtra}
\usepackage{amscd}
\usepackage{amsthm}
\usepackage{amsfonts}
\usepackage{amssymb}
\usepackage{eucal}
\usepackage{pmgraph}
\usepackage{stmaryrd}
\usepackage[usenames,dvipsnames]{color}
\usepackage{color}
\input prepictex
\input pictex
\input postpictex

\newtheorem{thm}{Theorem} [section]
\newtheorem{cor}[thm]{Corollary}
\newtheorem{lem}[thm]{Lemma}
\newtheorem{prop}[thm]{Proposition}
\newtheorem{conjecture}[thm]{Conjecture}

\theoremstyle{definition}
\newtheorem{definition}[thm]{Definition}
\newtheorem{example}[thm]{Example}

\theoremstyle{remark}
\newtheorem{rem}[thm]{Remark}

\numberwithin{equation}{section}

\begin{document}

%Referring commands:
\newcommand{\thmref}[1]{Theorem~\ref{#1}}
\newcommand{\secref}[1]{Section~\ref{#1}}
\newcommand{\lemref}[1]{Lemma~\ref{#1}}
\newcommand{\propref}[1]{Proposition~\ref{#1}}
\newcommand{\corref}[1]{Corollary~\ref{#1}}
\newcommand{\remref}[1]{Remark~\ref{#1}}
\newcommand{\eqnref}[1]{(\ref{#1})}
\newcommand{\exref}[1]{Example~\ref{#1}}

\newcommand{\blue}[1]{{\color{blue}#1}}
\newcommand{\red}[1]{{\color{red}#1}}
\newcommand{\green}[1]{{\color{green}#1}}

%Simplified symbols:
\newcommand{\nc}{\newcommand}
 \nc{\Z}{{\mathbb Z}}
 \nc{\C}{{\mathbb C}}
 \nc{\N}{{\mathbb N}}
 \nc{\F}{{\mf F}}
 \nc{\Q}{\mathbb{Q}}
 \nc{\la}{\lambda}
 \nc{\ep}{\epsilon}
 \nc{\h}{\mathfrak h}
 \nc{\ka}{\kappa}
 \nc{\n}{\mf n}
 \nc{\G}{{\mathfrak g}}
 \nc{\DG}{\widetilde{\mathfrak g}}
 \nc{\SG}{\breve{\mathfrak g}}
 \nc{\La}{\Lambda}
 \nc{\is}{{\mathbf i}}
 \nc{\V}{\mf V}
 \nc{\bi}{\bibitem}
 \nc{\NS}{\mf N}
 \nc{\E}{\mc E}
 \nc{\ba}{\tilde{\pa}}
 \nc{\half}{\frac{1}{2}}
 \nc{\mc}{\mathcal}
 \nc{\mf}{\mathfrak} \nc{\hf}{\frac{1}{2}}
 \nc{\hgl}{\widehat{\mathfrak{gl}}}
 \nc{\gl}{{\mathfrak{gl}}}
 \nc{\hz}{\hf+\Z}

\nc{\U}{\mathfrak u}
\nc{\SU}{\overline{\mathfrak u}}
\nc{\ov}{\overline}
\nc{\ul}{\underline}
\nc{\wt}{\widetilde}
\nc{\I}{\mathbb{I}}
\nc{\X}{\mathbb{X}}
\nc{\Y}{\mathbb{Y}}
\nc{\hh}{\widehat{\mf{h}}}
\nc{\aaa}{{\mf A}}
\nc{\xx}{{\mf x}}
\nc{\wty}{\widetilde{\mathbb Y}}
\nc{\ovy}{\overline{\mathbb Y}}
\nc{\vep}{\bar{\epsilon}}
\nc{\wotimes}{\widehat{\otimes}}
\nc{\Omn}{\mc{O}^{m|n}}
\nc{\OO}{\mc{O}}

\newcommand{\wgv}{\Lambda^{\infty} \mathbb V}
\newcommand{\wgw}{\Lambda^{\infty} \mathbb V^*}

\nc{\ch}{\text{ch}}
\nc{\glmn}{\mf{gl}(m|n)}
\nc{\Uq}{{\mathcal U}_q}
\nc{\VV}{\mathbb V}
\nc{\WW}{\mathbb W}
 \nc{\TL}{{\mathbb T}_{\mathbb L}}
 \nc{\TU}{{\mathbb T}_{\mathbb U}}
 \nc{\TLhat}{\widehat{\mathbb T}_{\mathbb L}}
 \nc{\TUhat}{\widehat{\mathbb T}_{\mathbb U}}
 \nc{\TLwt}{\widetilde{\mathbb T}_{\mathbb L}}
 \nc{\TUwt}{\widetilde{\mathbb T}_{\mathbb U}}
 \nc{\TLC}{\ddot{\mathbb T}_{\mathbb L}}
 \nc{\TUC}{\ddot{\mathbb T}_{\mathbb U}}
 \nc{\wtl}{X(m|n)}

\nc{\sm}{{\rm sm}}
%\nc{\lg}{\rm{lg}}

\advance\headheight by 2pt

\title[Brundan-Kazhdan-Lusztig conjecture]
{Brundan-Kazhdan-Lusztig conjecture for general linear Lie superalgebras}

\author[Cheng]{Shun-Jen Cheng}
\address{Institute of Mathematics, Academia Sinica, Taipei,
Taiwan 10617} \email{chengsj@math.sinica.edu.tw}

\author[Lam]{Ngau Lam}
\address{Department of Mathematics, National Cheng-Kung University, Tainan, Taiwan 70101}
\email{nlam@mail.ncku.edu.tw}

\author[Wang]{Weiqiang Wang}
\address{Department of Mathematics, University of Virginia, Charlottesville, VA 22904}
\email{ww9c@virginia.edu}

\begin{abstract}
In the framework of canonical and dual canonical bases of Fock
spaces, Brundan in 2003 formulated a Kazhdan-Lusztig-type conjecture
for the characters of the irreducible and tilting modules in the BGG
category for the general linear Lie superalgebra for the first time.
In this paper, we prove Brundan's conjecture and its variants
associated to all Borel subalgebras in full generality.
\end{abstract}

%\noindent{\bf Key words:} Lie superalgebra, differential
%operators, free field realization, Howe duality.

%\vspace{.3cm}
\maketitle

\let\thefootnote\relax\footnotetext{{\em 2010 Mathematics Subject Classification.} Primary 17B10.}

\tableofcontents

\section{Introduction}

\subsection{Background}

In the classical papers \cite{Kac1, Kac2}, Kac in 1970's initiated
the study of representations of Lie superalgebras, including the
general linear Lie superalgebra $\glmn$ and other basic Lie
superalgebras. Realizing the difficulty of generalizing the Weyl
character formula for finite-dimensional irreducible modules to the Lie
superalgebra setting, Kac found a Weyl type character formula for a
class of so-called {\em typical} finite-dimensional simple
modules. Since then, there have been numerous attempts to obtain
further results toward the irreducible character problem for
Lie superalgebras (see the bibliography of the book
\cite{CWbook} for a partial list).

Around 1996, Serganova \cite{Se} developed a mixed algebraic and
geometric approach to provide an algorithm for obtaining the
finite-dimensional irreducible characters for $\glmn$. For lack of a
general conceptual framework, there was virtually no serious attempt
in addressing the irreducible character problem in a
Bernstein-Gelfand-Gelfand (BGG) category $\OO$ for Lie superalgebras
such as $\glmn$, until the work of Brundan.

\subsection{BKL conjecture}

In his 2003 seminal paper \cite{Br} Brundan formulated an elegant
conjecture on the characters for the irreducible modules and tilting
modules in the full BGG category $\Omn$ of $\glmn$-modules for the
first time (also see \cite{Br3} for a similar formulation for the
Lie superalgebra $\mathfrak q(n)$). Brundan's formulation was in
terms of canonical and dual canonical bases of Lusztig and Kashiwara
\cite{Lu1, Ka} in a Fock space $\mathbb T^{{\bf b}_{\text{st}}}
:=\VV^{\otimes m} \otimes \WW^{\otimes n}$ (or rather in some
suitable completion $\widehat{\mathbb T}^{{\bf b}_{\text{st}}}$ of
$\mathbb T^{{\bf b}_{\text{st}}}$), where $\VV$ is the natural
module of the quantum group $U_q({\mathfrak g\mathfrak l}_\infty)$
and $\WW$ is the restricted dual of $\VV$. In this formulation, the
Verma modules correspond to the standard monomial basis, the
irreducible modules to the dual canonical basis, and the
tilting modules to the canonical basis in
$\widehat{\mathbb T}^{{\bf b}_{\text{st}}}$.

In other words, the entries of the transition matrices between the
standard monomial and (dual) canonical bases define the
Brundan-Kazhdan-Lusztig (BKL) polynomials, whose values at $q=1$
solve the multiplicity problem of irreducible and tilting characters
when expressed in terms of Verma characters. In a nutshell, the
category $\Omn$ categorifies the Fock space $\mathbb T^{{\bf
b}_{\text{st}}}$ and its (dual) canonical bases. (In the
Introduction below, we will ignore completely the issues of
completions of various Fock spaces, though it will take some
considerable portion of this paper to take care of such issues
properly.)

Brundan's conjecture can be adapted for various parabolic categories
of $\glmn$-modules, where an even Levi subalgebra of $\glmn$ in the
block matrix form of size $(m_1,m_2,\ldots,m_\ell | n_1, \ldots,
n_r)$ gives rise to a Fock space in terms of $q$-wedge subspaces as
follows:
$$
\wedge^{m_1}\VV\otimes\ldots\otimes\wedge^{m_\ell}\VV
\otimes\wedge^{n_1}\WW\otimes\ldots\otimes\wedge^{n_r}\WW.
$$
The validity of Brundan's conjecture on the full BGG category
implies the validity of all such parabolic versions. In the same
paper Brundan \cite{Br} proved a maximal parabolic version of his
general conjecture, which can be phrased that the category $\mc
F^{m|n}$ of finite-dimensional $\glmn$-modules categorifies
$\wedge^m \VV \otimes \wedge^n \WW$ and its (dual) canonical bases.

In the case when either $m$ or $n$ is zero, Brundan's formulation
reduces to the by now well-known reformulation of the classical
Kazhdan-Lusztig conjecture \cite{KL, KL2} (and a tilting module
version \cite{So, So2}) for category $\OO$ of general linear Lie
algebras, which takes advantage of the Schur-Jimbo duality
\cite{Jim}. The formulation of Brundan's conjecture in terms of
quantum groups and canonical basis is a conceptual way of getting
around the well-known difficulty that the Weyl group $W=\mf S_m
\times \mf S_n$ and its associated Hecke algebra are insufficient to
control the linkage principle for $\glmn$.

As is well known, there exists simple systems for $\glmn$ which are
not conjugate under the action of the Weyl group $W$, for $m,n\ge
1$. The $W$-conjugacy classes of simple systems are in bijection
with what we call $0^m1^n$-sequences $\bf b$ (there are ${m+n\choose
n}$ of them in total). To each such $\bf b$ is associated a Borel
subalgebra $\mf b$ of $\glmn$, and these Borel subalgebras are not
conjugate to each other. It has been expected (cf. Kujawa's thesis
\cite{Ku}) that Brundan's conjecture affords variants in terms of
Fock spaces $\mathbb T^{\bf b}$, which is a $q$-tensor space with
$m$ tensor factors isomorphic to $\VV$ and $n$ factors isomorphic to
$\WW$, determined by the sequence $\bf b$ (see \eqref{eq:Fock} for a
precise definition), for each $\bf b$. We will refer to all these
$\bf b$-variants as {\em Brundan-Kazhdan-Lusztig (BKL) conjecture}.
Kujawa's work \cite{Ku} provided a first supporting evidence on the
crystal basis level for the BKL conjecture. In this paper, we shall
need and hence formulate such $\bf b$-variants of Brundan's
conjecture precisely, which requires some suitable completions of
the corresponding Fock space $\mathbb T^{\bf b}$ in order to
construct the (dual) canonical bases. The original Brundan's
conjecture is associated to the standard sequence ${\bf
b}_{\text{st}}=(0,\ldots,0,1,\ldots,1)$.

It is clear from the beginning that Brundan's conjecture is a
central problem in representation theory of Lie superalgebras. As
the proof of the classical KL conjecture \cite{KL, KL2} was
completed in \cite{BB, BK} independently using deep geometric
machinery and the BKL conjecture includes the type $A$ KL conjecture
as a special case, there seemed to be little hope of proving
Brundan's conjecture directly due to the inadequate development of
the geometric approach in super representation theory. However,
Brundan's work convinced the authors that a general and conceptual
approach, though likely completely novel, to the representation
theory of Lie superalgebras might still be possible.

\subsection{Goal}

The goal of this paper is to prove the BKL conjecture for the BGG
category $\Omn$ in full generality. The proof consists of two major
steps. First, via an extension of the super duality approach
developed earlier by the authors, we establish an inductive
procedure on $n$ of proving the BKL conjecture for $\gl(m|n+1)$
based on the validity of {\em one} $\bf b$-variant of the BKL
conjecture for $\glmn$, for every $m$. Our second main step is to
show that the $\bf b$-variants of BKL conjecture for all
$0^m1^n$-sequences $\bf b$ are equivalent to each other.

\subsection{Super duality}

Let us explain the super duality in some detail,  which was already
used to solve some distinguished parabolic versions of the BKL
conjecture.

A precise connection, which was christened as Super Duality, between
representation theory of $\glmn$ and that of $\gl(m+n)$ was
formulated in \cite{CWZ} and in full generality in \cite{CW} for the
first time. Super duality is a (conjectural at that time) category
equivalence between suitable parabolic module categories of $\glmn$
and $\gl(m+n)$ at the $n\mapsto \infty$ limit. Such a connection in
the most special case \cite{CWZ} was supported by, and in turn
implies Brundan's solution for the irreducible and tilting
characters in the category $\mc F^{m|n}$. The conjectured general
super duality was shown in \cite{CW} to imply the distinguished
parabolic versions of BKL conjecture whose corresponding Fock spaces
are of the form
$\wedge^{m_1}\VV\otimes\ldots\otimes\wedge^{m_\ell}\VV
\otimes\wedge^{n}\WW$. One bonus consequence of the super duality
approach is that the BKL polynomials in these distinguished
parabolic categories are shown to be exactly the classical parabolic
Kazhdan-Lusztig polynomials.

In \cite{CL}, a powerful yet elementary approach was developed to
prove the super duality conjecture of \cite{CW}. One notable feature
of \cite{CL} is that it does not rely on the results of \cite{Br} a
priori, and hence the parabolic versions of the BKL conjecture as
formulated in \cite{CW} followed. There is yet another independent
and  complete solution by Brundan and Stroppel \cite{BS} toward the
irreducible and tilting character problem for the maximal parabolic
category $\mc F^{m|n}$. As far as the full BGG category $\Omn$ is
concerned, almost nothing is known so far. The super duality
approach has been further developed in \cite{CLW1} for the
ortho-symplectic Lie superalgebras, where the irreducible and
tilting characters were shown to be expressible in terms of Verma
characters via classical Kazhdan-Lusztig polynomials.

\subsection{Strategy of proof}\label{stra:of:proof}

Now we explain in more detail the outline of our proof of the BKL
conjecture. The BKL conjecture for the BGG category $\Omn$ with
respect to the standard Borel subalgebra will be abbreviated as
$\texttt{BKL}(m|n)$. Let $\h_{m|n}$ be the Cartan subalgebra of
diagonal matrices in $\glmn$. We shall denote by
$\texttt{BKL}(m|n+\underline{k})$ the (parabolic) BKL conjecture
for the parabolic category of $\mf{gl}(m|n+k)$-modules with Levi
subalgebra $\mf h_{m|n} \oplus \mf{gl}(k)$, for $0\le k \leq \infty$. Furthermore, we let $\texttt{BKL}(m|n|k)$ denote the BKL conjecture
for the BGG category of $\mf{gl}(m+k|n)$-modules associated
to the $0^{m+k}1^n$-sequence $(0^m,1^n,0^k)$, for $0\le k < \infty$. Finally, we let $\texttt{BKL}(m|n|\underline{k})$ denote the BKL conjecture
for the parabolic BGG category of $\mf{gl}(m+k|n)$-modules associated
to the same $0^{m+k}1^n$-sequence $(0^m,1^n,0^k)$ and with Levi
subalgebra $\mf h_{m|n} \oplus \mf{gl}(k)$, for $0\le k \leq \infty$.

The overall strategy for establishing the BKL conjecture is by
induction on $n$. The inductive procedure, denoted by
 $%{\bf Procedure_n}:
\texttt{BKL}(m|n), \, \forall m \Longrightarrow \texttt{BKL}(m|n+1), \, \forall m$,
 is divided into the following steps:
{\allowdisplaybreaks
\begin{align}
 \texttt{BKL}(m+k|n) \;\; \forall k,m
 & \Longrightarrow  \texttt{BKL}(m|n|k) \;\; \forall k,m, \text{ by changing Borels}
  \label{ind:oddref} \\
 &\Longrightarrow \texttt{BKL}(m|n|\underline{k})
  \;\; \forall k,m,  \text{ by passing to parabolic}
  \label{ind:para} \\
 &\Longrightarrow \texttt{BKL}(m|n|\underline{\infty}),  \forall m, \text{ by taking $k\mapsto \infty$}
  \label{ind:infty} \\
 & \Longrightarrow \texttt{BKL}(m|n+\underline{\infty}),\forall m, \text{ by super duality}
  \label{ind:SD}  \\
 & \Longrightarrow \texttt{BKL}(m|n+1)\;\; \forall m, \text{ by truncation}.
  \label{ind:trunc}
\end{align}
 }

It is instructive to write down the Fock spaces corresponding to the
steps above:
{\allowdisplaybreaks
\begin{align*}
 \VV^{\otimes (m+k)} \otimes \WW^{\otimes n}\;\;  \forall k,m
 &\Longrightarrow  \VV^{\otimes m} \otimes \WW^{\otimes n} \otimes
 \VV^{\otimes k}\;\;  \forall k,m
  \\
 &\Longrightarrow \VV^{\otimes m} \otimes \WW^{\otimes n}\otimes
 \wedge^k\VV\;\; \forall k,m
   \\
 &\Longrightarrow \VV^{\otimes m} \otimes \WW^{\otimes n}\otimes
 \wedge^\infty \VV\;\;\forall m
   \\
 & \Longrightarrow \VV^{\otimes m} \otimes \WW^{\otimes n}\otimes
 \wedge^\infty\WW\;\;\forall m
   \\
 & \Longrightarrow
 %\VV^{\otimes m} \otimes \WW^{\otimes n}\otimes \WW =
 \VV^{\otimes m} \otimes \WW^{\otimes (n+1)}\;\;\; \forall m.
\end{align*}
}

While our proof is purely algebraic, it is ultimately based on the
geometric proof of the original KL conjecture. The base case for the
induction, $ \texttt{BKL}(m|0)$, is equivalent to the original
Kazhdan-Lusztig conjecture \cite{KL} for $\mf{gl}(m)$, which is a
theorem of \cite{BB} and \cite{BK} (also see \cite{BGS, V}; the tilting module characters were due to \cite{So, So2}).
Step~\eqref{ind:para} can be regarded as a generalization of Deodhar
\cite{Deo}, Soergel \cite{So} and Brundan \cite{Br}.
Steps~\eqref{ind:infty}--\eqref{ind:trunc} are generalizations of our
earlier work in \cite{CWZ, CW, CL}, in which we establish the
compatibilities between various constructions on the categorical
level and their counterparts on the Fock space level.

Step~\eqref{ind:oddref} is a special case of a key new result of
this paper, which states that all $\bf b$-variants of the BKL
conjecture (for fixed $m,n$) are equivalent. We are reduced to
compare the (dual) canonical bases of the Fock spaces $\mathbb
T^{\bf b}$ and $\mathbb T^{\bf b'}$ associated to adjacent
$0^m1^n$-sequences $\bf b$ and $\bf b'$. Here being adjacent
corresponds to differing by an odd reflection on the corresponding
Borel subalgebras. To that end, {\em parabolic monomial bases} for
$\mathbb T^{\bf b}$ and $\mathbb T^{\bf b'}$ are introduced, and
they are used to match the (dual) canonical bases in $\mathbb T^{\bf
b}$ and $\mathbb T^{\bf b'}$. Note that parabolic monomial bases
admit counterparts in the category $\Omn$. We also show how
tilting modules transform under odd reflections, and establish a
remarkable property that every tilting module in $\Omn$ has $\bf
b$-Verma flags for all $0^m1^n$-sequences $\bf b$. Now (1.1) follows
from such constructions and results.
\subsection{Applications}\label{sec:intro:app}

We show in Proposition~\ref{prop:sameO} that the categories
$\Omn_{\bf b}$ are identical, for all $0^m1^n$-sequences ${\bf b}$
(note that the even subalgebras of the Borels corresponding to
various $\bf b$ are fixed to be the same). On the other hand, we
show in Section~\ref{sec:adjCB} that the canonical and dual
canonical bases in $\mathbb T^{\bf b}$ for different $\bf b$ are
compatible and hence they are unique in a suitable sense; but the
standard bases in $\mathbb T^{\bf b}$ do depend on $\bf b$. These
results strongly suggest the existence of a graded
lift in the sense of Soergel and others on the BGG category
$\Omn_{\bf b}$ (cf. \cite{BGS}). For the maximal parabolic category
$\mc F^{m|n}$, the $\Z$-grading was first established in \cite{BS}.

In another direction, a new approach to the full BGG category  for
the ortho-symplectic Lie superalgebras is being developed by
Huanchen Bao and one of the authors (see \cite{BW}).
%in his ongoing dissertation at University of Virginia.
While fresh new ideas are needed in the new setting, the approach
and the results established in the current paper have served as an
inspiration and played a fundamental role.

\subsection{Organization}

The paper is organized as follows. It is divided into two parts.
Part~1, which consists of
Sections~\ref{sec:Fock}--\ref{sec:adjacent}, deals with the
combinatorics of Fock spaces, canonical bases, and BKL polynomials.
Part~2, which consists of
Sections~\ref{sec:repn:prep}--\ref{sec:BKL}, concerns the
representation theory of $\glmn$. The results of Part~1 are used in
the final Section~\ref{sec:BKL}.

In Section~\ref{sec:Fock}, we review the quantum group, Hecke
algebra of type $A$, and Jimbo duality. The Bruhat orderings on the
Fock spaces $\mathbb T^{\bf b}$ are introduced.

In Section~\ref{sec:BKLpol}, we introduce the $A$- and
$B$-completions of Fock spaces. Adapting from Lusztig \cite{Lu} and
generalizing Brundan \cite{Br}, we formulate a bar involution on the
$B$-completions of Fock spaces, and establish the existence of
canonical and dual canonical bases. The transition matrices between
(dual) canonical bases and the standard monomial basis allows us to
define the BKL polynomials.

In Section~\ref{sec:compareCB}, via the notion of truncation maps,
we compare the (dual) canonical bases on Fock spaces involving
$\wedge^k\VV$ or $\wedge^k\WW$ for varying $k$. We formulate a
combinatorial version of super duality, which is an isomorphism of
Fock spaces preserving the (dual) canonical bases, when replacing
$\wedge^\infty \VV$ by $\wedge^\infty \WW$. The presentation here
generalizes and improves the special cases treated in \cite{CWZ,
CW}. We then formulate a precise relationship between (dual)
canonical bases of a Fock space and those of its various $q$-wedge
subspaces, and thus a relationship between their BKL polynomials.

In Section~\ref{sec:adjacent}, we develop a new approach to compare
precisely the canonical as well as dual canonical bases in two Fock
spaces associated to adjacent $0^m1^n$-sequences. For that purpose,
we introduce two kinds of parabolic monomial bases denoted by $N$'s
and $U$'s, the $N$'s being adapted to dual canonical basis while the
$U$'s to the canonical basis. All these are built on computations on
$\VV\otimes \WW$ and $\WW\otimes \VV$.

In Section~\ref{sec:repn:prep}, we show that the BGG category
$\Omn_{\bf b}$ of $\glmn$-modules are independent of $\bf b$, and so
denoted by $\Omn$. The $\bf b$-Verma modules and $\bf b$-tilting
modules are introduced. We show that a $\bf b$-tilting module is
always a $\bf b'$-tilting module for any other sequence $\bf b'$,
and provide a precise identification when $\bf b$ and $\bf b'$ are
adjacent. We also introduce some auxiliary modules denoted by $N$'s
and $U$'s in $\Omn$, which correspond to the parabolic monomial
bases with same the letters of Section~\ref{sec:adjacent}. The results
of this section are valid for any basic Lie superalgebra.

In Section~\ref{sec:SD}, we develop the super duality machinery in
the generality we need (compare \cite{CL, CLW1}). We establish an
equivalence of module categories for two infinite-rank Lie
superalgebras $\G$ and $\SG$, which match the corresponding
Kazhdan-Lusztig-Vogan polynomials in terms of Kostant $\mf
u$-homology groups. The category $\Omn$ (for varying finite $m,n$)
and related parabolic categories are recovered via truncation
functors.

In Section~\ref{sec:BKL}, we put together all the pieces from
previous sections. We establish the compatibilities of the BKL
conjectures for adjacent sequences \eqref{ind:oddref}, between full
and parabolic BGG categories \eqref{ind:para}, as well as all
remaining steps \eqref{ind:infty}, \eqref{ind:SD} and
\eqref{ind:trunc}. The BKL conjecture follows.

\subsection{Notations}

Let $\mc{P}$ denote the set of partitions. For $\la\in\mc{P}$ we denote its conjugate and length by $\la'$ and $\ell(\la)$, respectively. We let $\Z$, $\Z_+$,
$\Z_-$, and $\N$ denote the sets of all, non-negative, non-positive,
and positive integers, respectively. Denote by
$\Z_2=\{\bar{0},\bar{1}\}$. For $k\in\N$ we denote by $[k]$ the set
$\{1,2,\ldots,k\}$. Similarly, we set
$[\ul{k}]:=\{\ul{1},\ul{2},\ldots,\ul{k}\}$ and also write
$[\ul{\infty}]$ for $\{\ul{1},\ul{2},\ul{3},\ldots\}$. The symmetric group on a set of $k$ elements is denoted by $\mf{S}_k$.

 \vspace{.3cm}

 {\bf Acknowledgments.} The first
author is partially supported by an NSC grant, and he thanks NCTS/TPE and the Department of
Mathematics of University of Virginia for support. The second author
is partially supported by an NSC grant, and he thanks NCTS/South for support. The third author is
partially supported by an NSF grant DMS--1101268, and he thanks the
Institute of Mathematics of Academia Sinica in Taiwan for providing
excellent working environment and support.

 \vspace{.3cm}

\noindent {\bf Notes added.}
After the present paper arxiv:1203.0092 was posted, a second and completely different proof of our main Theorem \ref{th:BKL}
was obtained by Brundan, Losev, and Webster in ``Tensor product categorifications and the super Kazhdan-Lusztig conjecture", arxiv:1310.0349.
Their argument is based on the uniqueness of certain categorifications,
and this allows them to further establish the Koszul graded lifts on $\mc{O}^{m|n}_{\bf b}$ as mentioned in Section \ref{sec:intro:app}.
A proof of our positivity Conjecture~ \ref{con:positive} (see also Remark~\ref{rem:pos}) of this paper is also found in that paper.

%%%
%%%
%%%
%%%
\part{Combinatorics}

\section{Fock spaces and Bruhat orderings}
\label{sec:Fock}

In this section, we introduce the Fock space ${\mathbb T}^{\bf b}$,
which is a tensor space of copies of the natural $U_q({\mathfrak
g\mathfrak l}_\infty)$-module and its restricted dual, associated to a
$0^m1^n$-sequence $\bf b$. We define the Bruhat ordering on $\mathbb
T^{\bf b}$. Some $q$-wedge spaces are also introduced.

\subsection{Quantum group}

Let $q$ be an indeterminate. The quantum group $U_q({\mathfrak
g\mathfrak l}_\infty)$ is defined to be the associative algebra over
$\Q(q)$ generated by $E_a, F_a, K_a, K^{-1}_a, a \in \Z$, subject to
the following relations ($a,b\in\Z$):
\begin{eqnarray*}
 K_a K_a^{-1} &=& K_a^{-1} K_a =1, \\
 K_a K_b &=& K_b K_a, \\
 K_a E_b K_a^{-1} &=& q^{\delta_{a,b} -\delta_{a,b+1}} E_b, \\
 K_a F_b K_a^{-1} &=& q^{\delta_{a,b+1}-\delta_{a,b}}
 F_b, \\
 E_a F_b -F_b E_a &=& \delta_{a,b} \frac{K_{a,a+1}
 -K_{a+1,a}}{q-q^{-1}}, \\
 E_a^2 E_b +E_b E_a^2 &=& (q+q^{-1}) E_a E_b E_a,  \quad \text{if } |a-b|=1, \\
 E_a E_b &=& E_b E_a,  \,\qquad\qquad\qquad \text{if } |a-b|>1, \\
 F_a^2 F_b +F_b F_a^2 &=& (q+q^{-1}) F_a F_b F_a,  \quad\, \text{if } |a-b|=1,\\
 F_a F_b &=& F_b F_a,  \qquad\ \qquad\qquad \text{if } |a-b|>1.
\end{eqnarray*}
Here $K_{a,a+1} :=K_aK_{a+1}^{-1}$. For $r\ge 1$, we introduce the
divided powers $E_a^{(r)} =E_a^{r}/[r]!$ and
$F_a^{(r)}=F_a^{r}/[r]!$, where $[r] =(q^r-q^{-r})/(q-q^{-1})$ and
$[r]!=[1][2]\cdots [r].$

Setting $\ov{q}=q^{-1}$ induces an automorphism on $\Q(q)$ denoted
by $^{-}$. Define the bar involution on $U_q({\mathfrak g\mathfrak
l}_\infty)$ to be the anti-linear automorphism ${}^-: U_q({\mathfrak
g\mathfrak l}_\infty)\rightarrow U_q({\mathfrak g\mathfrak
l}_\infty)$ determined by $\ov{E_a}= E_a$, $\ov{F_a}=F_a$, and
$\ov{K_a}=K_a^{-1}$. Here {\em anti-linear} means that
$\ov{fu}=\ov{f}\ov{u}$, for $f\in\Q(q)$ and $u\in U_q(\gl_\infty)$.
The quantum group $U_q(\mf{sl}_\infty)$ is the subalgebra generated
by $\{E_a,F_a,K_{a,a+1}\vert a\in\Z\}$.

Let $\mathbb V$ be the natural $U_q({\mathfrak g\mathfrak
l}_\infty)$-module with basis $\{v_a\}_{a\in\Z}$ and $\mathbb W
:=\mathbb V^*$, the restricted dual module with basis
$\{w_a\}_{a\in\Z}$ such that $\langle w_a ,v_b\rangle = (-q)^{-a}
\delta_{a,b}$.  The actions of $U_q(\gl_\infty)$ on $\mathbb V$ and
$\mathbb W$ are given by the following formulas:
\begin{align*}
&K_av_b=q^{\delta_{ab}}v_b,\qquad E_av_b=\delta_{a+1,b}v_a,\ \ \quad
F_av_b=\delta_{a,b}v_{a+1},\\
&K_aw_b=q^{-\delta_{ab}}w_b,\quad E_aw_b=\delta_{a,b}w_{a+1},\quad
F_aw_b=\delta_{a+1,b}w_{a}.
\end{align*}
We shall use the co-multiplication $\Delta$ on $U_q(\gl_\infty)$
defined by:
\begin{eqnarray*}
 \Delta (E_a) &=& 1 \otimes E_a + E_a \otimes K_{a+1, a}, \\
 \Delta (F_a) &=& F_a \otimes 1 +  K_{a, a+1} \otimes F_a,\\
 \Delta (K_a) &=& K_a \otimes K_{ a},
\end{eqnarray*}
which restricts to a co-multiplication on $U_q(\mf{sl}_\infty)$. Our
$\Delta$ here is consistent with the one used by Kashiwara, but
differs from \cite{Lu}.

For $m, n\in\Z_+$, recall that $[m]=\{1,2,\cdots,m\}$, and denote
the set of integer-valued functions on $[m+n]$ by $\Z^{m+n}$. We
shall also identify $f$ with the $(m+n)$-tuple
$\left(f(1),f(2),\ldots,f(m+n)\right)$ when convenient. Also, for a
subset $I\subseteq [m+n]$, we shall denote the restriction of $f$ to
$I$ by $f_I$. For example, if $I=\{i,i+1\}$, then the restriction of
$f$ to $I$ will be denoted by $f_I=f_{i,i+1}$. This notation remains
valid for functions defined on different domains as well.

\subsection{Fock spaces}

For $m,n\in\Z_+$, we let ${\bf b}=(b_1,b_2,\ldots,b_{m+n})$ be a
sequence of $m+n$ integers such that $m$ of the $b_i$'s are equal to
${0}$ and $n$ of them are equal to ${1}$. We call such a sequence a
{\em $0^m1^n$-sequence}.

We associate to such a ${0^m1^n}$-sequence ${\bf b}$ the following
tensor space over $\Q(q)$, called the {\em $\bf b$-Fock space} or simply
{\em Fock space}:
\begin{equation}  \label{eq:Fock}
{\mathbb T}^{\bf b} :={\mathbb V}^{b_1}\otimes {\mathbb
V}^{b_2}\otimes\cdots \otimes{\mathbb V}^{b_{m+n}},
\end{equation}
where we denote
$${\mathbb V}^{b_i}:=\begin{cases}
{\mathbb V}, &\text{ if }b_i={0},\\
{\mathbb W}, &\text{ if }b_i={1}.
\end{cases}$$
The tensors here and in similar settings later on are understood
to be over the field $\Q(q)$.
Note that the algebra $U_q(\gl_\infty)$ acts on $\mathbb T^{\bf b}$ via the
co-multiplication $\Delta$.

\begin{example}
Let $m=3$ and $n=2$. Then $({0},{0},{1},{1},{0})$,
$({0},{0},{0},{1},{1})$, and $({0},{1},{0},{1},{0})$ are
$0^31^2$-sequences. The associated $U_q(\gl_\infty)$-modules
${\mathbb T}^{\bf b}$ are respectively
\begin{align*}
{\mathbb V}\otimes{\mathbb V}\otimes{\mathbb W}
 \otimes{\mathbb W}\otimes{\mathbb V},
  \quad
{\mathbb V}\otimes{\mathbb V}\otimes{\mathbb V}
 \otimes{\mathbb W}\otimes{\mathbb W},
 \quad
{\mathbb V}\otimes{\mathbb W}\otimes{\mathbb V}
 \otimes{\mathbb W}\otimes{\mathbb V}.
\end{align*}
\end{example}

For $f\in \Z^{m+n}$ we define
\begin{equation}  \label{eq:Mf}
M^{\bf b}_f :=\texttt{v}^{b_1}_{f(1)}\otimes
\texttt{v}^{b_2}_{f(2)}\otimes\cdots\otimes
\texttt{v}^{b_{m+n}}_{f(m+n)},
\end{equation}
where we use the notation
$\texttt{v}^{b_i}:=\begin{cases}v,\text{ if }b_i={0},\\
w,\text{ if }b_i={1}.\end{cases}$ We refer to $\{M^{\bf b}_f |f\in
\Z^{m+n}\}$ as the {\em standard monomial basis} for ${\mathbb
T}^{\bf b}$.

\subsection{Bruhat ordering}\label{super:Bruhat}

Let $\texttt{P}$ be the free abelian group with orthonormal basis
$\{\varepsilon_r\vert r\in\Z\}$ with respect to a bilinear form
$(\cdot | \cdot)$. We define a partial order on $\texttt{P}$ by
declaring $\nu\ge\mu$, for $\nu,\mu\in \texttt{P}$, if $\nu-\mu$ is
a non-negative integral linear combination of
$\varepsilon_r-\varepsilon_{r+1}$, $r \in \Z$.

Fix a ${0^m1^n}$-sequence ${\bf b}=(b_1,\ldots,b_{m+n})$. For
$f\in\Z^{m+n}$ and $j\le m+n$, we define
\begin{align*}
\text{wt}^j_{\bf b}(f)
 :=\sum_{j\le i}(-1)^{b_i}\varepsilon_{f(i)}\in \texttt{P},\quad
\text{wt}_{\bf b}(f):=\text{wt}^{1}_{\bf b}(f)\in \texttt{P}.
\end{align*}

We define the {\em Bruhat ordering of type ${\bf b}$} on $\Z^{m+n}$,
denoted by $\preceq_{\bf b}$, in terms of the partially ordered set
$(\texttt{P}, \leq)$ as follows: $g\preceq_{\bf b}f$ if and only if
$\text{wt}_{\bf b}(g)=\text{wt}_{\bf b}(f)$ and $\text{wt}^j_{\bf
b}(g)\le\text{wt}^j_{\bf b}(f)$, for all $j$. Note that when $n=0$
this is simply the usual Bruhat ordering on the weight lattice
$\Z^m$ of $\gl(m)$.

Following \cite[\S2-b]{Br} we introduce the following notation in
our general setting:
\begin{align}\label{aux:sharp:Bruhat}
\sharp_{\bf b}(f,a,j):=\sum_{j\le i\in [m+n],f(i)\le a}(-1)^{b_i}, \quad
\text{  for } f\in\Z^{m+n}.
\end{align}
It is easy to see the following characterization of $\preceq_{\bf
b}$ holds: for $g,f\in\Z^{m+n}$,
\begin{equation}  \label{eq:sharpBr}
g\preceq_{\bf b}f\; \Leftrightarrow\; \sharp_{\bf b}(g,a,j)\le
\sharp_{\bf b}(f,a,j),\;
 \forall a\in\Z,j\in [m+n], \text{
with equality for $j=1$.}
\end{equation}

For $i\in [m+n]$ we let $d_i\in\Z^{m+n}$ be determined by
\begin{equation}  \label{eq:di}
d_i(j)=(-1)^{b_i}\delta_{ij}, \quad \text{ for } j\in [m+n].
\end{equation}
Let $f,g\in\Z^{m+n}$. We write 
$f\downarrow_{\bf b} g$ if one of the following holds:
\begin{align*}
\begin{cases}
g=f\cdot (i,j),&\text{ for }b_i=b_j={0},i<j, f(i)>f(j),\\
g=f\cdot (i,j),&\text{ for }b_i=b_j={1},i<j, f(i)<f(j),\\
g=f-d_i+d_j,&\text{ for }b_i\not=b_j, i<j, f(i)=f(j).
\end{cases}
\end{align*}
Here and further we denote the natural right action of $\mf S_{m+n}$ on
$\Z^{m+n}$ by $f\cdot \sigma:=f\circ\sigma$, for $f\in \Z^{m+n}$ and
$\sigma \in \mf S_{m+n}$. The following lemma is clear.

\begin{lem}\label{lem:super:Bru}
Let $f,g\in\Z^{m+n}$. If there exists a sequence of elements
$h_1,h_2,\ldots,h_k\in\Z^{m+n}$ such that $h_i\downarrow_{\bf
b}h_{i+1}$, for $1\le i\le k-1$, with $h_1=f$ and $h_k=g$, then
$f\succeq_{\bf b}g$.
\end{lem}

\begin{rem}
In the case of standard $0^m1^n$-sequence ${\bf b}_{\text{st}}
=(0,\ldots,0,1,\ldots,1)$ as well as the opposite sequence
$(1,\ldots,1,0,\ldots,0)$, the converse of Lemma~\ref{lem:super:Bru}
holds, according to \cite[Lemma~2.5]{Br}. However the converse is no
longer true in general.
 Take the $0^31^2$-sequence ${\bf b}=(0,1,0,1,0)$.
Take $f=(4,3,5,2,1)$ and $g=(1,2,4,3,5)$. Then $f\succ_{\bf b}g$.
%However, $f$ and $g$ are typical in the sense that $f(i) \neq f(j)$
%whenever $b_i\neq b_j$.
But there exists no such sequence $\{h_i\}$ as in the above lemma
moving $f$ down to $g$.

The Bruhat ordering $\preceq_{\bf b}$ on $\Z^{m+n}$ is defined to
fit with the definition on the weight lattice of $\glmn$ coming from
central characters (see \cite[Section~2.2]{CWbook}). But in this
paper it would work if we have defined $\preceq_{\bf b}$ to be the
transitive closure of $\downarrow_{\bf b}$.
\end{rem}

The following lemma will be useful in the sequel.

\begin{lem}\label{lem:f:to:g:finite}
The poset $(\Z^{m+n}, \preceq_{\bf b})$ satisfies the finite
interval property. That is, given $f, g$ with $g \preceq_{\bf b}f$,
the set $\{h\in\Z^{m+n}\vert g\preceq_{\bf b} h\preceq_{\bf b} f\}$
is finite.
\end{lem}

\begin{proof}
We shall prove the precise and stronger statement that if $g
\preceq_{\bf b}h\preceq_{\bf b}f$, then
\begin{align}\label{aux:eq:h:max}
|h(i)|\le \max\{|f(j)|,|g(j)|\text{ with } j\in [m+n]\},\quad\forall
i\in [m+n].
\end{align}
We prove \eqref{aux:eq:h:max} by contradiction. Suppose that
$|h(t)|>\max\{|f(j)|,|g(j)|\mid j\in [m+n]\}$,  for some $t\in
[m+n]$. Among the $t$'s with $|h(t)|=N$ maximal, we choose $i$ as
large as possible so that $h(i) =\pm N$. First suppose that $h(i)=N$.
Then $\sharp_{\bf b}(f,N-1,i)=\sharp_{\bf b}(g,N-1,i)$. But clearly
$\sharp_{\bf b}(h,N-1,i)\not=\sharp_{\bf b}(f,N-1,i)$. Now suppose that $h(i)=-N$.
Then $0=\sharp_{\bf b}(f,-N,i)=\sharp_{\bf b}(g,-N,i)$. But clearly
$\sharp_{\bf b}(h,-N,i)=(-1)^{b_i}$. So in either case we cannot have
$g\preceq_{\bf b} h\preceq_{\bf b} f$.
\end{proof}

\subsection{$q$-wedge spaces}

For $k\in\N$, denote by $\mf{S}_{k}$ the symmetric group of
permutations on $\{\ul{1},\ul{2},\ldots,\ul{k}\}$. Let
$\mf{S}_\infty=\bigcup_{k}\mf{S}_k$. Then $\mf{S}_{k}$ is generated
by the simple transpositions $s_{1} =(\ul{1},\ul{2}), s_{2}
=(\ul{2},\ul{3}), \ldots, s_{k-1} =(\ul{k-1},\ul{k})$.

The Iwahori-Hecke algebra associated to $\mf{S}_k$ (for
$k\in\N\cup\{\infty\}$) is the associative $\mathbb Q(q)$-algebra
$\mathcal H_{k}$ generated by $H_i$, $1 \le i \le k-1$, subject to
the relations
\begin{align*}
&(H_i -q^{-1})(H_i +q) = 0,\\
&H_i H_{i+1} H_i = H_{i+1} H_i H_{i+1},\\
&H_i H_j = H_j H_i, \quad\text{for } |i-j| >1.
\end{align*}
Associated to $\sigma \in \mf{S}_{k}$ with a reduced
expression $\sigma=s_{i_1} \cdots s_{i_r}$, we define
 $
H_\sigma :=H_{i_1} \cdots H_{i_r}. $ The bar involution on $\mathcal
H_{k}$ is the unique anti-linear automorphism defined by
$\overline{H_\sigma} =H_{\sigma^{-1}}^{-1}, \overline{q} =q^{-1},$
for all $\sigma \in \mf{S}_{k}$. Set for $k\in\N$
\begin{align}\label{formula:H0}
H_0:= \sum_{\sigma \in \mf{S}_{k}}
(-q)^{\ell(\sigma)-\ell(w^{(k)}_0)} H_\sigma,
\end{align}
where $w_0^{(k)}$ denotes the longest element in $\mf{S}_k$. It is
well known that (cf. \cite{KL}, \cite[Proposition~2.9]{So})
\begin{equation}  \label{eq:H0inv}
\ov{H_0}=H_0.
\end{equation}

Now consider the tensor spaces $\mathbb V^{\otimes k}$ and $\mathbb
W^{\otimes k}$, respectively. In either case we index the tensor
factors by $[\ul{k}]:=\{\ul{1},\ul{2},\ldots,\ul{k}\}$. Now for an
integer-valued function $f: [\ul{k}]\rightarrow\Z$, recall from
\eqref{eq:Mf} that
$M_f = \texttt{v}_{f(\ul{1})} \otimes \cdots \otimes
\texttt{v}_{f(\ul{k})},$ where $\texttt{v}=v$ for $\mathbb
V^{\otimes k}$ and $\texttt{v}=w$ for $\mathbb W^{\otimes k}$. The
algebra $\mathcal H_{k}$ acts on $\mathbb V^{\otimes k}$ and
respectively on $\mathbb W^{\otimes k}$ on the right by
\begin{eqnarray} \label{eq:heckeaction}
 M_f H_i = \left\{
 \begin{array}{ll}
 M_{f\cdot s_i}, & \text{if } f \prec_{\bf b} f \cdot s_i,  \\
 q^{-1} M_{f}, & \text{if } f = f \cdot s_i, \\
 M_{f \cdot s_i} - (q-q^{-1}) M_f, & \text{if } f \succ_{\bf b} f \cdot
 s_i.
 \end{array}
 \right.
\end{eqnarray}
Here ${\bf b}=({0}^k)$ for $\mathbb V^{\otimes k}$, and ${\bf
b}=({1^k})$ for $\mathbb W^{\otimes k}$.

\begin{lem} \cite{Jim} \label{lem:commute}
The actions of $U_q(\gl_\infty)$ and $\mathcal H_{k}$ on the tensor
space $\mathbb V^{\otimes k}$ (and respectively, on $\mathbb
W^{\otimes k}$) commute with each other.
\end{lem}

Different commuting actions of $U_q(\gl_\infty)$ and $\mc H_k$ on
$\mathbb V^{\otimes k}$ were used in \cite{KMS} to construct the
space $\wedge^k \mathbb V$ of finite $q$-wedges and then the space
of infinite $q$-wedges by taking an appropriate limit $k\to\infty$.
These spaces carry the action of $U_q(\gl_\infty)$ (as a limiting
case). The constructions in {\em loc.~cit.}~carry over using the
above actions of $U_q(\gl_\infty)$ and $\mc H_k$ as we shall sketch
below.

First consider the case when $k\in\N$. Following \cite{KMS}, we will
regard $\wedge^k \mathbb V$ as the quotient of $\mathbb V^{\otimes
k}$ by $\ker H_0$. Indeed $\ker H_0$ equals the sum of the kernels
of the operators $H_i -q^{-1}$, $1 \le i \le k-1$, by
\cite[Proposition 1.2]{KMS} (note that the Hecke algebra generator
$T_i$ used in {\em loc.~cit.}~corresponds to $-q H_i$.) We further
denote by the {\em $q$-wedges} $v_{a_1} \wedge \cdots \wedge
v_{a_k}$ the image in $\wedge^k \mathbb V$ of $v_{a_1} \otimes
\cdots \otimes v_{a_k}$ under the canonical map. We have
\begin{eqnarray}
\cdots \wedge v_{a_i} \wedge v_{a_{i+1}} \wedge \cdots
 &=& -q^{-1} (\cdots \wedge v_{a_{i+1}} \wedge v_{a_i} \wedge
\cdots), \quad \text{if } a_i < a_{i+1}; \nonumber\\
\cdots \wedge v_{a_i} \wedge v_{a_{i+1}} \wedge \cdots
 &=& 0, \quad \text{if } a_i = a_{i+1}.  \label{eq:straighten}
\end{eqnarray}
It follows that the elements $v_{a_1} \wedge \cdots \wedge
v_{a_k}$, where $a_1 > \cdots > a_k$ form a basis for $\wedge^k
\mathbb V$. By Lemma~\ref{lem:commute}, $U_q(\gl_\infty)$ acts
naturally on $\wedge^k \mathbb V$.

\begin{rem}
For finite $k$, recalling that $\mathbb V^{\otimes k} = \ker H_0
\oplus \text{Im} H_0$, cf.~\cite[Proposition~ 1.1]{KMS}, we may
regard equivalently $\wedge^k\mathbb V$ as the subspace $\text{Im}
H_0$ of $\mathbb V^{\otimes k}$.
\end{rem}

Now consider the limit $k\to\infty$.  Let $\mathbb{V}^{\infty}$ be
the subspace of $\mathbb V^{\otimes\infty}$ spanned by vectors of
the form
\begin{align}\label{semi:infinite:vp2}
v_{p_1}\otimes v_{p_2}\otimes v_{p_3}\otimes\cdots,
\end{align}
with $p_i=1-i$, for $i\gg 0$. Note that $U_q(\gl_\infty)$ and the
Hecke algebra act on this space. We define $\wedge^{\infty}\mathbb
V$ to be the quotient of $\mathbb{V}^{\infty}$ by the sum of the
kernels of $H_i-q^{-1}$, for all $i \ge 1$. The quantum group
$U_q(\gl_\infty)$ acts on $\wedge^{\infty}\mathbb V$ and this space
has a basis given by the (formal) infinite $q$-wedges
\begin{align*}
v_{p_1} \wedge v_{p_2} \wedge & v_{p_3} \wedge \cdots,
\end{align*}
where  $p_1 > p_2  >p_3> \cdots,$ and $p_i =1-i$ for $i \gg0$.
Here the infinite $q$-wedge is defined to be the image of the
corresponding vector in \eqref{semi:infinite:vp2} under the
canonical quotient map. Alternatively, the space
$\wedge^\infty\mathbb V$ has a basis indexed by partitions given by
$$
|\la\rangle := v_{\la_1} \wedge v_{\la_2-1} \wedge v_{\la_3-2}\wedge
\cdots,
$$
where $\la =(\la_1, \la_2, \cdots)$ runs over the set $\mc P$ of all
partitions.

Let
\begin{align}  \label{eq:Zk+}
\begin{split}
\Z^{k}_+ &=\{f:[\ul{k}] \rightarrow \Z \mid f(\ul{1})> f(\ul{2})>
\cdots> f(\ul{k})\},  \text{ for } k\in\N,
  \\
\Z^{\infty}_+ &=\{f:[\ul{\infty}] \rightarrow \Z \mid f(\ul{1})>
f(\ul{2})> \cdots; f(\ul{t})=1-t \text{ for }t\gg 0\}.
\end{split}
\end{align}
%
%Let $\Z^{k}_+$ be the set of $\Z$-valued functions on $[\ul{k}]$
%satisfying $f(\ul{1})> f(\ul{2})> \cdots> f(\ul{k})$.
%If $k=\infty$, we let $\Z^\infty_{+}$ be the $\Z$-valued functions
%on $[\ul{\infty}]$ with $f(\ul{1})> f(\ul{2})>f(\ul{3})>\cdots$ and
%$f(\ul{t})=1-t$, for $t\gg 0$.
For $f\in\Z^{k}_+$, we denote
\begin{align*}
\mathcal{V}_f=v_{f(\ul{1})}\wedge v_{f(\ul{2})}\wedge\cdots \wedge
v_{f(\ul{k})}.
\end{align*}
Then $\{\mathcal{V}_f|f\in\Z^k_+\}$ is a basis for $\wedge^k\mathbb
V$, for $k\in\N\cup\{\infty\}$.

For $\ul{i}\in [\ul{k}]$, define $d_{\ul{i}}: [\ul{k}]\rightarrow\Z$
by letting $d_{\ul{i}}(\ul{j})=\delta_{ij}$, for $1\le j \le k$, in
the case of $\mathbb V^{\otimes k}$. Then $\wedge^k \mathbb V$ is
naturally a $U_q(\mf{gl}_\infty)$-module, where the action of the
Chevalley generators $E_a,F_a,K_a$, for $a\in\Z$, is given as
follows:

\begin{align}\label{action:wedgeV}
\begin{split}
&E_a \mathcal{V}_f =\begin{cases}
\sum_{\ul{i}}\delta_{a+1, f(\ul{i})}
\mathcal{V}_{f-d_{\ul{i}}},&\text{ if }f-d_{\ul{i}}\in\Z^k_+,\\
0,&\text{ otherwise}.
\end{cases}\\
&F_a \mathcal{V}_f=\begin{cases}
\sum_{\ul{i}}\delta_{a, f(\ul{i})}
\mathcal{V}_{f+d_{\ul{i}}},&\text{ if }f+d_{\ul{i}}\in\Z^k_+,\\
0,&\text{ otherwise}.
\end{cases}\\
&K_a \mathcal{V}_f=
q^{\sum_{\ul{i}}\delta_{a, f(\ul{i})}} \mathcal{V}_{f}.
\end{split}
\end{align}

A similar construction gives rise to $\wedge^k\mathbb W$, for
$k\in\N\cup\{\infty\}$. The space $\wedge^\infty\mathbb W$ has a
basis consisting of
\begin{align*}
w_{p_1}\wedge w_{p_2}\wedge  w_{p_3} & \wedge\cdots,
 \end{align*}
where $p_1<p_2<p_3<\cdots$, and $p_i=i$, for $i\gg 0$.
Alternatively,
the space $\wedge^\infty\mathbb W$ has a basis indexed by partitions
given by
$$
|\la_*\rangle := w_{1-\la_1} \wedge w_{2-\la_2} \wedge
w_{3-\la_3}\wedge \cdots, $$ where $\la =(\la_1, \la_2, \cdots)$
runs over the set of all partitions.

Let
\begin{align}  \label{eq:Zk-}
\begin{split}
\Z^{k}_- &=\{f:[\ul{k}] \rightarrow \Z \mid f(\ul{1})<f(\ul{2})<
\cdots< f(\ul{k})\},  \text{ for }  k\in\N,
  \\
\Z^{\infty}_- &=\{f:[\ul{\infty}] \rightarrow \Z \mid f(\ul{1})<
f(\ul{2})< \cdots; f(\ul{t})=t \text{ for }t\gg 0\}.
\end{split}
\end{align}
For $f\in\Z^{k}_-$ we write
\begin{align*}
\mathcal{W}_{f}=w_{f(\ul{1})}\wedge w_{f(\ul{2})}\wedge\cdots\wedge
w_{f(\ul{k})}.
\end{align*}
Then $\{\mathcal{W}_f|f\in\Z^k_-\}$ is a basis for $\wedge^k\mathbb
W$, for $k\in\N\cup\{\infty\}$.

For $\ul{i}\in [\ul{k}]$, define $d_{\ul{i}}: [\ul{k}]\rightarrow\Z$
by letting $d_{\ul{i}}(\ul{j})=-\delta_{ij}$, for $1\le j \le k$, in
the case of $\mathbb W^{\otimes k}$. Then $\wedge^k \mathbb W$ is
naturally a $U_q(\mf{gl}_\infty)$-module, where the action of the
Chevalley generators $E_a,F_a,K_a$, for $a\in\Z$, on $\wedge^k
\mathbb W$ is given as follows:
\begin{align}\label{action:wedgeW}
\begin{split}
& E_a \mathcal{W}_f=\begin{cases}
\sum_{\ul{i}}\delta_{a, f(\ul{i})}
 \mathcal{W}_{f-d_{\ul{i}}},&\text{ if }f-d_{\ul{i}}\in\Z^k_-,\\
0,&\text{ otherwise}.
\end{cases}\\
&F_a \mathcal{W}_f=\begin{cases}
\sum_{\ul{i}}\delta_{a+1, f(\ul{i})}
 \mathcal{W}_{f+d_{\ul{i}}},&\text{ if }f+d_{\ul{i}}\in\Z^k_-,\\
0,&\text{ otherwise}.
\end{cases}\\
&K_a \mathcal{W}_f=
q^{-\sum_{\ul{i}}\delta_{a, f(\ul{i})}} \mathcal{W}_{f}.
\end{split}
\end{align}

We define a $\Q(q)$-linear isomorphism
$\natural:\wedge^\infty\mathbb V\rightarrow \wedge^\infty\mathbb W$
by
\begin{align*}
\natural(|\la\rangle):=|\la'_*\rangle,\quad  \text{ for }
\la\in\mc{P}.
\end{align*}

\begin{prop}\label{wedgeV:isom:wedgeW}
\cite[Theorem 6.3]{CWZ} The map $\natural:\wedge^\infty\mathbb
V\rightarrow \wedge^\infty\mathbb W$ is an isomorphism of
$U_q(\mf{sl}_\infty)$-modules (both are isomorphic to the basic
module of $U_q(\mf{sl}_\infty)$).
\end{prop}

\section{Canonical bases and Brundan-Kazhdan-Lusztig polynomials}
\label{sec:BKLpol}

In this section, we introduce the $A$- and $B$-completions of
$\mathbb T^{\bf b}$. Then we define bar-involution, canonical and
dual canonical bases in the $B$-completion  of $\mathbb T^{\bf b}$.
The Brundan-Kazhdan-Lusztig polynomials are also introduced.

\subsection{Quasi-$\mc R$-matrix}

Let $M$ be a $U_q(\gl_\infty)$-module equipped with a
$\Q(q)$-anti-linear bar involution $\bar{\ }:M\rightarrow M$, such
that $\ov{um}=\bar{u}\ov{m}$, for all $u\in U_q(\gl_\infty)$ and
$m\in M$. Suppose furthermore that $M$ has a basis $B$ consisting of
bar-invariant weight vectors. We shall refer to $(M,B)$ or simply
$M$ as a {\em weakly based module}. We note that Lusztig introduced
the notion of a based module in \cite[27.1.2]{Lu}, which is a weakly
based module satisfying additional conditions.

\begin{example}\label{ex:inv:basis}
The $U_q(\gl_\infty)$-modules $\mathbb V$ and $\mathbb W$ have bar
involutions defined by $\ov{v_a}=v_a$ and $\ov{w_a}=w_a$,
respectively, that are compatible with the actions of the quantum
group. Thus, $(\mathbb V,\mathbb B^{0})$ and $(\mathbb W,\mathbb
B^{1})$ are weakly based modules, where we denote $\mathbb
B^{0}=\{v_a|a\in\Z\}$ and $\mathbb B^{1}=\{w_a|a\in\Z\}$. It follows
from \eqref{action:wedgeV} and \eqref{action:wedgeW} that the
$U_q(\gl_\infty)$-modules $\wedge^k\mathbb V$ and $\wedge^k\mathbb
W$ are also weakly based modules with basis given by
$\{\mathcal{V}_f|f\in\Z^{k}_+\}$, and
$\{\mathcal{W}_{f}|f\in\Z^k_-\}$, respectively, for $k\in\N$. The
same is true for $k=\infty$ so that $(\wedge^\infty\mathbb
V,\{|\la\rangle|\la\in\mc{P}\})$ and $(\wedge^\infty\mathbb
W,\{|\la_*\rangle|\la\in\mc{P}\})$ are also weakly based modules.
(Actually these are all examples of based modules in the sense of
Lusztig, but we will not need this fact.)
\end{example}

In what follows we shall apply results from \cite{Lu} and
\cite{Jan}. To translate their results to our setting, we need to
replace $q^{-1}$ therein by $q$, and interchange $E_a$ with $F_a$,
for all $a\in\Z$, in order to match our co-multiplication with
theirs. From Lusztig's theory of based modules \cite[Chapter
27]{Lu}, using the quasi-$\mc R$-matrix $\Theta$, one can construct
from $k$ weakly based modules $(M_i,B_i)$ two distinguished bases of
the $U_q(\gl_\infty)$-module $M_1\otimes M_2\otimes\cdots\otimes
M_k$, called {canonical} and {dual canonical basis}, respectively.
We shall review and extend these constructions below, as strictly
speaking Lusztig's construction was carried out for finite-rank
quantum groups.

In order to construct a bar involution on the tensor product of two
weakly based modules, we will first define the quasi-$\mc R$-matrix
$\Theta$, which in turn is based on the existence of a PBW-type
basis.

Denote by $\Phi^+$ the standard positive root system of
$U_q(\gl_\infty)$, and set
$$
P^+ =\sum_{\alpha\in\Phi^+}\Z_+\alpha.
$$
For $k\in\N$, let $U_q(\gl_{|k|})$ be the subalgebra of
$U_q(\gl_\infty)$ generated by $\{E_a,F_a,K_a^{\pm 1},K_{a+1}^{\pm
1}\}$ for $-k\le a\le k-1$. Then we have $U_q(\gl_{|k|})\subseteq
U_q(\gl_{|k+1|})$ and $\bigcup_{k}U_q(\gl_{|k|})=U_q(\gl_\infty)$.
Furthermore, $\mc{U}^\pm_{|k|}\subseteq \mc{U}^\pm_{|k+1|}$ and
$\bigcup_k\mc{U}^\pm_{|k|}=\mc{U}^\pm$, where $\mc{U}^\pm_{|k|}$ and
$\mc{U}^\pm$ denote the positive and negative parts of
$U_q(\gl_{|k|})$ and $U_q(\gl_\infty)$, respectively.

For $k\in\N$, let $\mf{S}_{|k|}$ denote the symmetric group on the
set $\{-k,-k+1,\ldots,0,1,\ldots,k\}$, and let $w_0^{|k|}$ denote
the longest element in $\mf S_{|k|}$. Then there exists a reduced
expression $w'\in\mf{S}_{|k+1|}$ such that
\begin{align*}
w_0^{|k+1|}=w_0^{|k|}w', \quad \text{ where }
\ell(w_0^{|k+1|})=\ell(w_0^{|k|}) + \ell(w').
\end{align*}
Hence there exists an infinite sequence of simple roots $\alpha_1,
\alpha_2, \alpha_3, \ldots$ such that for each $k$ we have a reduced
expression for $w_0^{|k|}$ as
\begin{align}  \label{wn:wn+1}
w^{|k|}_0=s_{\alpha_1}s_{\alpha_2}\cdots s_{\alpha_{N}},  \quad
\text{ where } N =k(2k+1).
\end{align}
Associated to a simple root $\alpha$ one can define an automorphism
$T_{\alpha}:U_q(\gl_\infty)\rightarrow U_q(\gl_\infty)$
\cite[8.14]{Jan}.
For a sequence of non-negative integers $(a_i)$ indexed by $[N]$, we
define the element
\begin{align}\label{PBW:basis}
T_{\alpha_1}T_{\alpha_2}\ldots
T_{\alpha_{N-1}}(E^{a_N}_{\alpha_N})\cdots
T_{\alpha_1}T_{\alpha_2}(E^{a_{3}}_{\alpha_{3}}) \cdot
T_{\alpha_1}(E^{a_2}_{\alpha_2}) \cdot E_{\alpha_1}^{a_1}\in
U_q(\gl_{|k|}).
\end{align}
Then, the set of all such elements form a basis for the positive
part $\mc U^+_{|k|}$ \cite[Theorem~ 8.24]{Jan}. Taking the limit
$k\to\infty$ we obtain a basis for $\mc U^+$ consisting of
elements \eqref{PBW:basis} with $N$ arbitrarily large.
Replacing the $E_{\alpha_i}$'s in \eqref{PBW:basis} by the
corresponding $F_{\alpha_i}$'s, we obtain a basis for $\mc U^-$. For
$\mu\in P^+$, denote by $\mc U^+_\mu$ the corresponding $\mu$-weight
space of $\mc U^+$, and by $\mc U^-_{-\mu}$ the corresponding
$(-\mu)$-weight space of $\mc U^-$.

The quasi-$\mc R$-matrix $\Theta$ is an element in some suitable
completion of $\mc U^+\otimes \mc U^-$. For later use let us write
down an explicit formula for $\Theta$ by mimicking the construction
in \cite[8.30(2)]{Jan}. Associated to the positive root
$s_{\alpha_1}s_{\alpha_2}\cdots s_{\alpha_{t-1}}(\alpha_t)$ for
$t\in \N$, we define
\begin{align}  \label{theta:t}
\Theta_{[t]}:=\sum_{r\ge 0}q^{r(r-1)/2}\frac{(q-q^{-1})^r}{[r]!}
T_{\alpha_1}\cdots T_{\alpha_{t-1}}(E^r_{\alpha_t})\otimes
T_{\alpha_1}\cdots T_{\alpha_{t-1}}(F^r_{\alpha_t}).
\end{align}
Now let $\mu\in P^+$. We choose $k$ large enough so that $\mu$ is a
weight of $\gl_{|k|}$. We set $\Theta_\mu\otimes\Theta_{-\mu}$ to be
the $\mc U^+_{\mu}\otimes\mc U^-_{-\mu}$-component of the product
$\Theta_{[N]}\cdots\Theta_{[2]}\Theta_{[1]}$. This definition is
independent of sufficiently large $k$ and $N$, and we define the
quasi-$\mc R$-matrix $\Theta$  for $U_q(\gl_\infty)$ as
\begin{align}\label{quasi:R:matrix}
\Theta=\sum_{\mu\in P^+}\Theta_\mu\otimes\Theta_{-\mu}.
\end{align}
Formally, we have just made sense of the infinite product
$\Theta=\cdots\Theta_{[3]}\Theta_{[2]}\Theta_{[1]}$. Similarly, the
quasi-$\mc R$-matrix $\Theta^{(k)}$ for $U_q(\gl_{|k|})$ is defined
as \cite[Chapter~8]{Jan}
\begin{align}  \label{eq:thetan}
\Theta^{(k)}=
\Theta_{[N]}\cdots\Theta_{[3]}\Theta_{[2]}\Theta_{[1]},  \quad
\text{ where } N =k(2k+1).
\end{align}

\subsection{Completions}
\label{sec:completions}

Let ${\bf b}$ be a fixed ${0^m1^n}$-sequence. Let $k\in\N$ and let
$\mathbb T^{\bf b}_{\le |k|}$ be the (truncated) $\Q(q)$-subspace of
$\mathbb T^{\bf b}$, spanned by the elements $M_f^{\bf b}$ defined
in \eqref{eq:Mf}, with $-k\le f(i)\le k$, for all $i\in [m+n]$. Let
\begin{align}\label{def:pi:lek}
\pi_k:\mathbb T^{\bf b}\longrightarrow \mathbb T^{\bf b}_{\le |k|}
\end{align}
be the natural projection map with respect to the basis $\{M_f^{\bf
b}\}$ for $\mathbb T^{\bf b}$. The kernels of the $\pi_k$'s define a
linear topology on the vector space $\mathbb T^{\bf b}$. We then let
$\wt{\mathbb T}^{\bf b}$ be the completion of $\mathbb T^{\bf b}$
with respect to this topology. Formally, every element in
$\wt{\mathbb T}^{\bf b}$ is a possibly infinite linear combination
of $M_f$, for $f\in\Z^{m+n}$. We let $\widehat{\mathbb T}^{\bf b}$
denote the subspace of $\wt{\mathbb T}^{\bf b}$ spanned by elements
of the form
\begin{align}\label{vec:in:comp}
M_f+\sum_{g\prec_{\bf b}f}r_g M_g, \quad \text{ for } r_g\in\Q(q).
\end{align}

\begin{definition}\label{def:completion:T}
The $\Q(q)$-vector spaces $\wt{\mathbb T}^{\bf b}$ and
$\widehat{\mathbb T}^{\bf b}$ are called the $A$-{\em completion}
and $B$-{\em completion} of $\mathbb T^{\bf b}$, respectively.
\end{definition}

\begin{rem}
A similar completion was introduced by Brundan \cite[\S 2-d]{Br} for
the standard $0^m1^n$-sequence $ {\bf
b}_{\text{st}}=(\underbrace{0,\ldots,0}_m,\underbrace{1,\ldots,1}_n)
$.
\end{rem}

\subsection{Bar involution} % on $\widehat{\mathbb T}^{\bf b}$}

For two finite-dimensional weakly based modules $(M,B)$ and $(N,C)$
of a finite-rank quantum group, Lusztig \cite[27.3.1]{Lu} defined a
bar map $\psi$ on the tensor space $M\otimes N$ via the quasi-$\mc
R$-matrix $\Theta$ by
\begin{align}\label{def:bar:map}
\psi(m\otimes n) :=\Theta(\ov{m}\otimes\ov{n}),\quad \forall m\in M,
n\in N.
\end{align}
Then, the bar map $\psi$ is an involution by \cite[Corollary
4.1.3]{Lu} and furthermore is compatible with the action on
$M\otimes N$ induced by the co-multiplication $\Delta$ by
\cite[Lemma~ 24.1.2]{Lu}. We shall adapt this construction to
$U_q(\gl_\infty)$-modules below. However, because our modules are
not finite dimensional, we shall need to deal with completion
issues.

%%
%%%%%%%%%%

Consider the weakly based modules $(\mathbb V,\mathbb B^{0})$ and
$(\mathbb W,\mathbb B^{1})$. Let ${\bf b}=(b_1,b_2,\ldots,b_{m+n})$
be a fixed ${0^m1^n}$-sequence. Let $\mathbb T^{\bf b}$ be as in
\eqref{eq:Fock} and recall its $A$-completion $\wt{\mathbb T}^{\bf
b}$ from Definition~\ref{def:completion:T}. We first construct a
$\Q(q)$-anti-linear bar map $\ \bar{}: \mathbb T^{\bf
b}\rightarrow\wt{\mathbb T}^{\bf b}$. To be definite for now, we
regard $\mathbb T^{\bf b}$ in \eqref{eq:Fock} as taking tensor
product successively from left to right; that is,
${\mathbb T}^{\bf b} =((({\mathbb V}^{b_1}\otimes {\mathbb
V}^{b_2})
\otimes {\mathbb V}^{b_3})
\otimes\cdots \otimes{\mathbb V}^{b_{m+n}})$.
By \eqref{def:bar:map} we
can use the quasi-$\mc R$-matrix $\Theta^{(k)}$ to construct an involution $\psi^{(k)}$ on
$\mathbb T^{\bf b}_{\le |k|}$, which is a tensor product of
$U_q(\gl_{|k|})$-modules, for $k\in \N$. Recall the projection map
$\pi_k: \mathbb T^{\bf b} \rightarrow \mathbb T^{\bf b}_{\le |k|}$.

\begin{lem}\label{lem:stab:bar:map}
Let $f\in\Z^{m+n}$ and $k\in\N$ be such that $|f(i)|\le k$, for all
$i\in [m+n]$. Then we have  $M_f\in\mathbb T^{\bf b}_{\le|k|}$, and
\begin{align}\label{aux:eq:stab10}
\psi^{(k)}({M_f})=\pi_k\left(\psi^{(\ell)}({M_f})\right), \quad
\text{for } \ell\ge k.
\end{align}
\end{lem}

\begin{proof}
It is clear that  $M_f\in\mathbb T^{\bf b}_{\le|k|}$. We prove
\eqref{aux:eq:stab10} by induction on $m+n$. The case for $m+n\le 2$
is easily checked directly. For $\ell\ge k$ we compute that
{\allowdisplaybreaks
\begin{align*}
\pi_{k}&(\psi^{(\ell)}{M_f})\\
&=\pi_{k}\left(\Theta^{(\ell)}\left(\psi^{(\ell)}({\texttt{v}^{b_1}_{f(1)}\otimes
 \texttt{v}^{b_2}_{f(2)}\otimes\cdots\otimes
 \texttt{v}^{b_{m+n-1}}_{f(m+n-1)}})
 \otimes \texttt{v}^{b_{m+n}}_{f(m+n)}\right)\right)
 \\
&=\pi_{k}\left(\sum_{\mu}\Theta^{(\ell)}_{\mu}\left(\psi^{(\ell)}({\texttt{v}^{b_1}_{f(1)}
 \otimes \texttt{v}^{b_2}_{f(2)}\otimes\cdots\otimes
 \texttt{v}^{b_{m+n-1}}_{f(m+n-1)}})\right)\otimes
 \Theta^{(\ell)}_{-\mu}\left(\texttt{v}^{b_{m+n}}_{f(m+n)}\right)\right)
 \\
&=\pi_{k}\left(\sum_{\mu}\Theta^{(k)}_{\mu}\left(\psi^{(\ell)}({\texttt{v}^{b_1}_{f(1)}
 \otimes \texttt{v}^{b_2}_{f(2)}\otimes\cdots\otimes
 \texttt{v}^{b_{m+n-1}}_{f(m+n-1)}})\right)\otimes
 \Theta^{(k)}_{-\mu}\left(\texttt{v}^{b_{m+n}}_{f(m+n)}\right)\right)
\\
&=\pi_{k}\left(\sum_{\mu}\Theta^{(k)}_{\mu}\pi_{k}
 \left(\psi^{(\ell)}({\texttt{v}^{b_1}_{f(1)}\otimes \texttt{v}^{b_2}_{f(2)}
 \otimes\cdots\otimes \texttt{v}^{b_{m+n-1}}_{f(m+n-1)}})\right)\otimes
 \Theta^{(k)}_{-\mu}\left(\texttt{v}^{b_{m+n}}_{f(m+n)}\right)\right)
 \\
&\stackrel{(*)}{=}\pi_{k}\left(\sum_{\mu}\Theta^{(k)}_{\mu}
 \left(\psi^{(k)}({\texttt{v}^{b_1}_{f(1)}\otimes \texttt{v}^{b_2}_{f(2)}
 \otimes\cdots\otimes \texttt{v}^{b_{m+n-1}}_{f(m+n-1)}})\right)\otimes
 \Theta^{(k)}_{-\mu}\left(\texttt{v}^{b_{m+n}}_{f(m+n)}\right)\right)
 \\
&=\pi_{k}\left(\psi^{(k)}\left(({\texttt{v}^{b_1}_{f(1)}\otimes
 \texttt{v}^{b_2}_{f(2)}\otimes\cdots\otimes \texttt{v}^{b_{m+n-1}}_{f(m+n-1)}})
 \otimes \texttt{v}^{b_{m+n}}_{f(m+n)}\right)\right)
 \\
&=\psi^{(k)}({M_f}).
\end{align*}}
The third identity above uses the fact that
$\pi_k(y\otimes\Theta^{(\ell)}(\texttt{v}^{b_{m+n}}_{t}))
=\pi_k(y\otimes\Theta^{(k)}(\texttt{v}^{b_{m+n}}_{t}))$, for any
$y\in \mathbb T^{\bf b'}$ with ${\bf b'}=(b_1,\ldots,b_{m+n-1})$ and
$|t|\le k$, and the induction hypothesis is used in the identity
$(*)$.
\end{proof}

It follows that the element
$\lim_{\ell\to\infty}\psi^{(\ell)}(M_f)$, for any
$f\in\Z^{m+n}$, is a well-defined element in $\wt{\mathbb T}^{\bf
b}$. We define
\begin{align}  \label{eq:barMf}
\psi({M}_f) :=\lim_{\ell\to\infty}\psi^{(\ell)} (M_f).
\end{align}
It follows immediately from \eqref{aux:eq:stab10} that
$\psi({M}_f)\in\wt{\mathbb T}^{\bf b}$ and
$\psi^{(k)}(M_f)=\pi_k(\psi({M}_f))$.

A different tensor order on $\mathbb T^{\bf b}$ and hence on
$\mathbb T^{\bf b}_{\le |k|}$ would give a different inductive way
to define a map $'\psi^{(k)}$, similar to $\psi^{(k)}$, for
$k\in\N$. For $f\in\Z^{m+n}$, choose $\ell$ so that $\ell \geq
\max_i\{|f(i)|\}$. By \cite[4.2.4]{Lu} we have
$\psi^{(\ell)}(M_f)={'\psi}^{(\ell)}(M_f)$, and hence
$\pi_k(\psi^{(\ell)}(M_f))=\pi_k({'\psi}^{(\ell)}(M_f))$ whenever
$\ell \geq k$. Thus, $\lim_{\ell\to\infty}{'\psi}^{(\ell)}M_f$ is a
well-defined element in $\wt{\mathbb T}^{\bf b}$, and it coincides
with $\lim_{\ell\to\infty}\psi^{(\ell)}M_f$. Hence we have proved
the following.

\begin{prop}\label{bar:indep:order}
The bar map $\ \bar{}:\mathbb T^{\bf b}\rightarrow\wt{\mathbb
T}^{\bf b}$, given by $\ov{M}_f =\psi(M_f)$ (see \eqref{eq:barMf}),
is well-defined, namely it is independent of the tensor order on
$\mathbb T^{\bf b}$.
\end{prop}

For $m=2$, $n=2$ and ${\bf b}=({0},{1},{0},{1})$ we may regard
$\mathbb T^{\bf b}=\left((\mathbb V\otimes \mathbb W)\otimes\mathbb
V\right)\otimes\mathbb W$, and apply the quasi-$\mc R$-matrix
$\Theta$ repeatedly from left to right and get a bar map on
${\mathbb T}^{\bf b}$ as above. We can also regard $\mathbb T^{\bf
b}$ as $\left(\mathbb V\otimes \mathbb (\mathbb W\otimes\mathbb
V)\right)\otimes\mathbb W$, and use this order to get a bar map.
Proposition~\ref{bar:indep:order} says that the two bar maps
coincide. Recall the $B$-completion $\widehat{\mathbb T}^{\bf b}$
from Definition \ref{def:completion:T}.

\begin{prop}\label{bar:inv:Mf}
Let $f\in\Z^{m+n}$ and $M_f\in\mathbb T^{\bf b}$. We have
\begin{align*}
\ov{M}_f=M_f+\sum_{g\prec_{\bf b}f}r_{gf}(q) M_g,
\end{align*}
where $r_{gf}(q) \in\Z[q,q^{-1}]$ and the sum is possibly infinite.
Hence, we have $\ \bar{}:\mathbb T^{\bf
b}\rightarrow\widehat{\mathbb T}^{\bf b}$.
\end{prop}
\begin{proof}
By \eqref{aux:eq:stab10} we have $\ov{M}_f\in\wt{\mathbb T}^{\bf
b}$. Making use of the explicit form \eqref{quasi:R:matrix} of the
quasi-$\mc R$-matrix $\Theta$, we first observe that the proposition
holds in the cases when $m+n \le 2$. We now proceed by induction on
$m+n$.

Let $f\in\Z^{m+n}$ and set $f'=f_{[m+n-1]}$.  Furthermore, for ${\bf
b}=(b_1,b_2,\ldots,b_{m+n})$, we set $'{\bf
b}=(b_1,b_2,\ldots,b_{m+n-1})$. We have \begin{align*}
M_f=\texttt{v}^{b_1}_{f(1)}\otimes
\texttt{v}^{b_2}_{f(2)}\otimes\cdots\otimes
\texttt{v}^{b_{m+n}}_{f(m+n)}.
\end{align*}
By the inductive assumption we compute that
\begin{align*}
\ov{M_f}
 &=\Theta\left(\ov{\texttt{v}^{b_1}_{f(1)}\otimes
\texttt{v}^{b_2}_{f(2)}  \otimes\cdots\otimes
\texttt{v}^{b_{m+n-1}}_{f(m+n-1)}}\otimes
 \ov{\texttt{v}^{b_{m+n}}_{f(m+n)}}\right)
 \\
 &=\sum_{\mu}\Theta_{\mu}\left(\ov{\texttt{v}^{b_1}_{f(1)}\otimes
 \texttt{v}^{b_2}_{f(2)}\otimes\cdots\otimes \texttt{v}^{b_{m+n-1}}_{f(m+n-1)}}\right)
 \otimes \Theta_{-\mu}\left(\texttt{v}^{b_{m+n}}_{f(m+n)}\right)
 \\
&=\sum_{\mu}\Theta_{\mu}\left(M_{f'}+\sum_{g'\prec_{'{\bf b}}
f'}s_{g'f'}(q) M_{g'}\right) \otimes
\Theta_{-\mu}\left(\texttt{v}^{b_{m+n}}_{f(m+n)}\right),
\end{align*}
where $s_{g'f'}(q) \in\Z[q,q^{-1}]$. Now recall that
$\Theta_{\mu}\otimes\Theta_{-\mu}$ is a $\Q(q)$-linear combination
of products of the form $T_{\alpha_1}\cdots
T_{\alpha_{k-1}}(E^r_{\alpha_k})\otimes T_{\alpha_1}\cdots
T_{\alpha_{k-1}}(F^r_{\alpha_k})$. From the explicit formulas for
these expressions in \cite[8.14(7)]{Jan} and the cases with $m+n=2$,
it follows that such an element, when applied to an element of the
form $M_h\otimes \texttt{v}^{b_{m+n}}_{b}$, for $h\in\Z^{m+n-1}$ and
$b\in\Z$, gives a $\Q(q)$-linear combination of elements of the form
$M_t\otimes \texttt{v}^{b_{m+n}}_c$, for $t\in\Z^{m+n-1}$ and
$c\in\Z$; moreover we have a sequence of weights $(h_i,c_i)$,
$i=1,\cdots,k$, such that
$$
(h,b)=(h_1,c_1)\downarrow_{{\bf b}}(h_2,c_2)\downarrow_{\bf
b}\cdots\downarrow_{{\bf b}}(h_k,c_k)=(t,c).$$ By
Lemma~\ref{lem:super:Bru}, we have $(t,c)\preceq_{\bf b}(h,b)$, and
hence $\ov{M}_f=M_f+\sum_{g\prec_{\bf b}f}r_{gf}(q) M_g$, for
$r_{gf}(q) \in\Q(q)$.

It remains to show that  $r_{gf}(q)\in \Z[q,q^{-1}]$. For this we
first observe that the $\Z[q,q^{-1}]$-span of the standard monomial
basis elements in $\mathbb T^{\bf b}$ is invariant under the action
of $K_a^{\pm 1}$, and the divided powers $E^{(j)}_a$ and
$F_a^{(j)}$, for $a\in\Z$ and $j\in\N$. From this observation,
\cite[8.14(7)]{Jan}, and  formula \eqref{theta:t} for
$\Theta_{[t]}$, it follows that $r_{gf}(q)\in\Z[q,q^{-1}]$.
\end{proof}

As we have already noted, due to the infinite-dimensionality of $\mathbb
V$ and $\mathbb W$, the bar involution does not preserve the space
${\mathbb T}^{\bf b}$. However, we have the following.

\begin{lem}   \label{lem:barTb}
The bar map $\ \bar{}:\mathbb T^{\bf b}\rightarrow\widehat{\mathbb
T}^{\bf b}$ extends to $\ \bar{}:\widehat{\mathbb T}^{\bf
b}\rightarrow\widehat{\mathbb T}^{\bf b}$. Furthermore, the bar map
on $\widehat{\mathbb T}^{\bf b}$ is an involution.
\end{lem}

\begin{proof}
To show that the bar map extends to $\widehat{\mathbb T}^{\bf b}$ we
need to show that if $y=M_f+\sum_{g\prec_{\bf b}f}r_{g}(q) M_g\in
\widehat{\mathbb T}^{\bf b}$, $r_{g}(q) \in\Q(q)$, then
$\ov{y}\in\widehat{\mathbb T}^{\bf b}$. By Proposition~
\ref{bar:inv:Mf} and the definition of $\widehat{\mathbb T}^{\bf
b}$, it remains to show that $\ov{y}\in\wt{\mathbb T}^{\bf b}$. To
see this, we note that if the coefficient of $M_h$ in $\ov{y}$ is
nonzero, then there exists $g\preceq_{\bf b}f$ such that $r_{hg}(q)\not=0$. Thus we have $h\preceq_{\bf b
}g\preceq_{\bf b}f$.  However, by Lemma \ref{lem:f:to:g:finite} there
are only finitely many such $g$'s. Thus, only finitely many $g$'s
can contribute to the coefficient of $M_h$ in $\ov{y}$, and hence
$\ov{y}\in\wt{\mathbb T}^{\bf b}$.

To show that $\ \bar{ }\ $ is an involution, we need to show that
for fixed $f,g\in\Z^{m+n}$ with $g\preceq_{\bf b}f$ we have
\begin{align}\label{aux:107}
\sum_{g\preceq_{\bf b}h\preceq_{\bf b}f} r_{gh}(q)\ov{r_{hf}(q)}=\delta_{fg}.
\end{align}
By Lemma~\ref{lem:f:to:g:finite}, there are only finitely many such
$h$'s with $g\preceq_{\bf b}h\preceq_{\bf b}f$. This together with
Lemma \ref{lem:stab:bar:map} implies that \eqref{aux:107} is
equivalent to the same identity on the finite-dimensional space
$\mathbb T^{\bf b}_{\le|k|}$, for $k\gg 0$. But in this case
\cite[Corollary 4.1.3]{Lu} is applicable. So we conclude that
\eqref{aux:107} holds and so the bar map is an involution.
\end{proof}

\subsection{Canonical basis}% in $\widehat{\mathbb T}^{\bf b}$}

For $r(q)\in\Q(q)$ recall that $\ov{r(q)} =r(q^{-1})$. A version of
the following lemma goes back to \cite{KL}. We note that \cite[Lemma 24.2.1]{Lu} is stated in a slightly different form, and also that, although
(vi)--(viii) are not listed there, the same
proof therein can be used to establish them.

\begin{lem}\label{lem:lusztig} \cite[Lemma 24.2.1]{Lu}
Let $(I,\preceq)$ be a partially ordered set satisfying the finite
interval property. Assume that for every $i\preceq j$ we are given
elements $r_{ij}\in\Z[q,q^{-1}]$ such that
\begin{itemize}
\item[(i)] $r_{ii}=1$, for all $i\in I$,
\item[(ii)] $\sum_{h,i\preceq h\preceq j}r_{ih}\ov{r}_{hj}=\delta_{ij}$.
\end{itemize}
Then there exists a unique family of elements $t_{ij}\in\Z[q]$ for
all $i\preceq j$ such that
\begin{itemize}
\item[(iii)] $t_{ii}=1$, for all $i\in I$,
\item[(iv)] $t_{ij}\in q\Z[q]$, for all $i\prec j$,
\item[(v)]
$t_{ij}=\sum_{h,i\preceq h\preceq j}r_{ih}\ov{t}_{hj}$, for all
$i\preceq j$.
\end{itemize}
Furthermore, there exists a unique family of elements
$\ell_{ij}\in\Z[q^{-1}]$ for all $i\preceq j$ such that
\begin{itemize}
\item[(vi)] $\ell_{ii}=1$, for all $i\in I$,
\item[(vii)] $\ell_{ij}\in q^{-1}\Z[q^{-1}]$, for all $i\prec j$,
\item[(viii)]
$\ell_{ij}=\sum_{h,i\preceq h\preceq j}r_{ih}\ov{\ell}_{hj}$, for
all $i\preceq j$.
\end{itemize}
\end{lem}

We shall now apply Lemma \ref{lem:lusztig} to the partially ordered
set $(\Z^{m+n},\preceq_{\bf b})$. Note first that the finite
interval condition in Lemma \ref{lem:lusztig} is satisfied due to
Lemma \ref{lem:f:to:g:finite}. Recall from
Proposition~\ref{bar:inv:Mf} that $\ov{M}_f=M_f+\sum_{g\prec_{\bf
b}f}r_{gf}(q) M_g$. So Property~(i) is clear, and (ii) follows
readily by applying the anti-linear bar-involution $\ \bar{}\ $ in
Lemma~\ref{lem:barTb} to the above identity. Hence we have
established the following.

\begin{prop}\label{prop:existence:can1}
The $\Q(q)$-vector space $\widehat{\mathbb T}^{\bf b}$ has unique
bar-invariant topological bases
\begin{align*}
\{T^{\bf b}_f|f\in\Z^{m+n}\}\text{ and }\{L^{\bf b}_f|f\in\Z^{m+n}\}
\end{align*}
such that
\begin{align*}
T^{\bf b}_f=M^{\bf b}_f+\sum_{g\prec_{\bf b}f}t_{gf}^{\bf b}(q)
M^{\bf b}_g,
 \qquad
L^{\bf b}_f=M^{\bf b}_f+\sum_{g\prec_{\bf b}f}\ell_{gf}^{\bf b}(q)
M^{\bf b}_g,
\end{align*}
with $t_{gf}^{\bf b}(q)\in q\Z[q]$, and $\ell_{gf}^{\bf b}(q)\in
q^{-1}\Z[q^{-1}]$, for $g\prec_{\bf b}f$. (We will also write
$t_{ff}^{\bf b}(q)=\ell_{ff}^{\bf b}(q)=1$, $t_{gf}^{\bf
b}=\ell_{gf}^{\bf b}=0$ for $g\npreceq_{\bf b}f$.)
\end{prop}

\begin{definition}  \label{def:BKLpolyn}
$\{T^{\bf b}_f|f\in\Z^{m+n}\}\text{ and }\{L^{\bf
b}_f|f\in\Z^{m+n}\}$ are called the {\em canonical basis} and {\em
dual canonical basis} for $\widehat{\mathbb T}^{\bf b}$,
respectively. Also, $t_{gf}^{\bf b}(q)$ and $\ell_{gf}^{\bf b}(q)$
are called {\em Brundan-Kazhdan-Lusztig (BKL) polynomials}.
\end{definition}

Recall $d_i \in \Z^{m+n}$ from \eqref{eq:di}. We define
\begin{equation} \label{eq:1mn}
\texttt{1}_{m|n} :=\sum_{i=1}^{m+n}(-1)^{b_i}d_i \in\Z^{m+n}.
\end{equation}

\begin{prop}\label{prop:shift:can:p}
For each $p\in\Z$ and $f,g\in\Z^{m+n}$, we have
\begin{align*}
t^{\bf b}_{gf}=t^{\bf
b}_{g+p\texttt{1}_{m|n},f+p\texttt{1}_{m|n}},\quad \ell^{\bf
b}_{gf}=\ell^{\bf b}_{g+p\texttt{1}_{m|n},f+p\texttt{1}_{m|n}}.
\end{align*}
\end{prop}

\begin{proof}
Define a $\Q(q)$-linear shift map $\texttt{sh}:{\mathbb T}^{\bf
b}\rightarrow {\mathbb T}^{\bf b}$ by
\begin{align*}
\texttt{sh}(M^{\bf b}_f):=M^{\bf b}_{f+\texttt{1}_{m|n}}.
\end{align*}
Since for $f\succeq_{\bf b}g$ if and only if
$f+\texttt{1}_{m|n}\succeq_{\bf b}g+\texttt{1}_{m|n}$, the map
$\texttt{sh}$ extends to a $\Q(q)$-linear map on the $B$-completion
$\widehat{\mathbb T}^{\bf b}$. Now $\texttt{sh}$ also commutes with
the bar map, since the quasi-$\mc R$-matrix is invariant under an
overall index shift by $1$. Thus, we conclude that
$\texttt{sh}(T^{\bf b}_f)=T^{\bf b}_{f+\texttt{1}_{m|n}}$ and
$\texttt{sh}(L^{\bf b}_f)=L^{\bf b}_{f+\texttt{1}_{m|n}}$. The
proposition follows.
\end{proof}

\subsection{Positivity}

For $r\ge 1$ and $a\in\Z$,  recall the divided powers $E_a^{(r)}$
and $F_a^{(r)}$. The following was conjectured in \cite[Conjecture
2.28(iii),(iv)]{Br}, in the case of the standard $0^m1^n$-sequence
${\bf b}_{\text{st}}$. Part (1) is a variant of  \cite[Theorem
3.3.6(3)]{Zh}.

\begin{thm}  \label{th:positivity}
Let ${\bf b}$ be a $0^m1^n$-sequence and $f\in\Z^{m+n}$. Let
$a\in\Z$, $r\ge 1$.
\begin{enumerate}
\item
The elements $E_a^{(r)}T^{\bf b}_f$ and $F_a^{(r)}T^{\bf b}_f$ can
be written as (possibly infinite) sums of $\{T^{\bf b}_g\vert
g\in\Z^{m+n}\}$ with coefficients in $\N[q,q^{-1}]$.

\item
The elements $E_a^{(r)}L^{\bf b}_f$ and $F_a^{(r)}L^{\bf b}_f$ can
be written as (possibly infinite) sums of $\{L^{\bf b}_g\vert
g\in\Z^{m+n}\}$ with coefficients in $\N[q,q^{-1}]$.
\end{enumerate}
\end{thm}

\begin{proof}
For $k\in \N$, consider the quantum group $U(\mf{gl}_{|k|})$ acting
on the finite-dimensional module $\mathbb T^{\bf b}_{\le|k|}$. Let
us denote the canonical and dual canonical basis elements of the
$U(\gl_{|k|})$-module $\mathbb T^{\bf b}_{\le|k|}$ by $T^{(k)}_f$
and $L^{(k)}_f$, respectively, for $f\in\Z^{m+n}_{\le |k|}:=\{f\in\Z^{m+n}\mid|f(i)|\le k$, $\forall
i\in[m+n]\}$. The proof of \cite[Lemma 24.2.1]{Lu} (cf.~our Lemma~\ref{lem:lusztig}) implies that the coefficients $t^{\bf
b}_{gf}$ and $\ell^{\bf b}_{gf}$ are uniquely determined by the
coefficients $r_{hf}(q)$ coming from the bar-involution with
$g\preceq_{\bf b}h\preceq_{\bf b}f$. Recall that such an $h$
satisfies $|h(i)|\le\max\{|f(j)|,|g(j)|\text{ with } j\in[m+n]\}$,
$\forall i\in[m+n]$, by \eqref{aux:eq:h:max}.
This together with the stability  \eqref{aux:eq:stab10} of the bar
involutions for varying $k$ implies that
\begin{align}\label{aux:eq:stab12}
\pi_k({T^{\bf b}_f})=T^{(k)}_f \quad
 \text{and}\quad \pi_k({L^{\bf
b}_f})=L^{(k)}_f,\quad\forall f\in\Z^{m+n}_{\le|k|}.
\end{align}

Now, we let $a\in\Z$ be fixed, and let $Y=T,L$. Observe that for $k>
|a|+1$ the map $\pi_k$ commutes with the action of $X=E_a,F_a$.
Thus, the map $X:\mathbb T^{\bf b}\rightarrow\mathbb T^{\bf b}$
given by letting $X$ act on the left extends uniquely to a
continuous map $X:\wt{\mathbb T}^{\bf b}\rightarrow\wt{\mathbb
T}^{\bf b}$, and hence the expression $X^{(r)}Y^{\bf b}_f$ is a
well-defined element in the $A$-completion $\wt{\mathbb T}^{\bf b}$.

Let $f\in\Z^{m+n}_{\le|k|}$ and choose $k>|a|+1$.
We write
\begin{align*}
X^{(r)}Y^{(k)}_f=\sum_{g\in\Z^{m+n}_{\le|k|}}b_g^{(k)}(q)Y^{(k)}_g,\quad
\text{ for } b_g^{(k)}(q)\in\Q(q).
\end{align*}
We compute, for $\ell\ge k$,
\begin{align*}
\pi_k\circ\pi_{\ell}(X^{(r)}Y^{\bf b}_f) =\pi_k(X^{(r)}Y^{(\ell)}_f)
=\pi_k(\sum_{g\in\Z^{m+n}_{\le|\ell|}}b^{(\ell)}_g(q)Y^{(\ell)}_g) =
\sum_{g\in\Z^{m+n}_{\le|k|}}b^{(\ell)}_g(q) Y^{(k)}_g.
\end{align*}
On the other hand, we compute
\begin{align*}
\pi_k\circ\pi_{\ell}(X^{(r)}Y^{\bf b}_f)=\pi_k(X^{(r)}Y^{\bf
b}_f)=X^{(r)}Y^{(k)}_f=\sum_{g\in\Z^{m+n}_{\le|k|}}b^{(k)}_g(q)
Y^{(k)}_g.
\end{align*}
It follows that $b^{(k)}_g(q)=b^{(\ell)}_g(q)$, for all
$g\in\Z^{m+n}_{\le|k|}$ and $\ell\ge k$. Thus, we obtain
\begin{align*}
X^{(r)}Y^{\bf b}_f=\sum_{g}b_{g}(q)Y^{\bf b}_g,
\end{align*}
where $b_{g}(q)=b^{(k)}_g(q)$, for $g\in\Z^{m+n}_{\le|k|}$. It
remains to show that $b_g(q)$ lie in $\N[q,q^{-1}]$.

We first prove Part (1). In the case when $Y=T$, Zheng \cite[Theorem
~3.3.6(3)]{Zh} proved that $b^{(k)}_g(q)\in\N[q,q^{-1}]$,  $\forall
g\in\Z^{m+n}_{\le|k|}$ and $\forall k$. Thus, we conclude that in
the case of $Y=T$ we have $b_g(q)\in\N[q,q^{-1}]$, for all
$g\in\Z^{m+n}$, which proves (1).

Since the dual of the natural $U(\mf{sl}_{|k|})$-module $\mathbb W_{\le |k|}$ is
isomorphic to the exterior power $\wedge^{2k}(\mathbb V_{\le|k|})$,
\cite[Theorem~ 11]{Br4} is applicable to $\mathbb T^{\bf b}_{\le
|k|}$. By \cite[Theorem~ 11]{Br4} there exists a symmetric bilinear
form $(\cdot\mid\cdot)$ on $\mathbb T^{\bf b}_{\le |k|}$ for which
the bases $\{T^{(k)}_f\}$ and $\{L^{(k)}_f\}$ are dual to each other
up to some change of labeling. Furthermore, for $u,v\in \mathbb
T^{\bf b}_{\le|k|}$ we have $(E^{(r)}_a u|v)=(u|F^{(r)}_av)$, for
all $k>|a|+1$, by \cite[Lemma 3]{Br4}. From this we conclude by
\cite[Theorem 3.3.6(3)]{Zh} again that the positivity in (2) holds
in the setting of $\mathbb T^{\bf b}_{\le |k|}$ for dual canonical
basis elements $L^{(k)}_f$. Now the same argument as in (1) proves
(2) as well.
\end{proof}

The following conjecture is a generalization of Brundan
\cite[Conjecture~ 2.28(i),(ii)]{Br}, who conjectured it for the
standard $0^m1^n$-sequence ${\bf b}_{\text{st}}$.

\begin{conjecture} \label{con:positive}
Let ${\bf b}$ be a $0^m1^n$-sequence. For $f,g\in\Z^{m+n}$, we have
 $t^{\bf b}_{gf}(q)\in\N[q]$, and
 $\ell^{\bf b}_{gf}(-q^{-1})\in\N[q]$.
\end{conjecture}

\begin{rem}
\label{rem:pos}
As mentioned in ``Notes added" at the end of Introduction, this conjecture has been established by
Brundan, Losev, and Webster (BLW).
Conjecture \ref{con:positive} can indeed be derived from the proof of Theorem \ref{th:positivity} above
directly without using categorification, as suggested by one referee (which should be known to BLW too).
Namely, the stability conditions \eqref{aux:eq:stab12} allow us to
regard these polynomials as the coefficients of canonical and dual canonical basis elements in $\mathbb T^{\bf b}_{\le|k|}$.
But then as in the proof above we may identify the truncated Fock space $\mathbb T^{\bf b}_{\le |k|}$ with another using $\VV_{\le |k|}$ alone,
i.e., replacing $\WW_{\le |k|}$ by $\wedge^{2k}\VV_{\le |k|}$.
Since the latter provides a reformulation of parabolic KL conjectures for type A Lie algebras,
the polynomials $t^{\bf b}_{gf}(q)$ are identified with certain type $A$ parabolic KL polynomials (cf. e.g. \cite[Remark 14]{Br4}),
which are known to be positive. On the other hand, the polynomials $\ell^{\bf b}_{gf}(-q^{-1})$ are not necessarily inverse parabolic KL polynomials,
but since they can be interpreted in terms of $\text{Ext}$-groups 
in singular blocks of a parabolic category $\mathcal O$ (see Backelin \cite{Ba}) they are positive. 

The observation in the proof above
on identifying the truncated Fock space $\mathbb T^{\bf b}_{\le |k|}$ with the one using $\VV_{\le |k|}$ alone
played a fundamental role in the BLW approach to the BKL conjecture.
\end{rem}
\section{Comparisons of canonical and dual canonical bases}
\label{sec:compareCB}

In this section, we introduce truncation maps to compare the (dual)
canonical bases on Fock spaces involving $\wedge^k\VV$ or
$\wedge^k\WW$ for varying $k$. We formulate a combinatorial version
of super duality. A precise relationship between (dual) canonical
bases of a Fock space and those of its various $q$-wedge subspaces
is then established.

\subsection{Truncation map}

Let $k\in\N\cup\{\infty\}$. We introduce the following notations.
For $f=(f_{[m+n]},f_{[\ul{k}]})\in\Z^{m+n}\times\Z^k_+$, set
\begin{align*}
{M}^{{\bf b},0}_f :=M^{\bf b}_{f_{[m+n]}}\otimes
\mc{V}_{f_{[\ul{k}]}}.
\end{align*}
Then $\{{M}^{{\bf b},0}_f\}$ forms a basis, called the {\em standard
monomial basis}, for the $\Q(q)$-vector space ${\mathbb T}^{\bf
b}\otimes\wedge^k\mathbb V$. Similarly, ${\mathbb T}^{\bf
b}\otimes\wedge^k\mathbb W$ admits a {\em standard monomial basis} given
by
\begin{align*}
M^{{\bf b},1}_g:=M^{\bf b}_{g_{[m+n]}}\otimes \mc{W}_{g_{[\ul{k}]}},
\end{align*}
where $g=(g_{[m+n]},g_{[\ul{k}]})\in\Z^{m+n}\times\Z^k_-$. Since
$\wedge^k\mathbb V$ and $\wedge^k\mathbb W$ are weakly based
modules, we can define bar maps for ${\mathbb T}^{\bf
b}\otimes\wedge^k\mathbb V$ and ${\mathbb T}^{\bf
b}\otimes\wedge^k\mathbb W$ by means of the quasi-$\mc R$-matrix as
in \eqref{def:bar:map} and \eqref{eq:barMf}. In this subsection we
prove the existence of canonical and dual canonical bases in the
$B$-completions of these vector spaces. We shall give the details
only for ${\mathbb T}^{\bf b}\otimes\wedge^k\mathbb W$, as the case
of ${\mathbb T}^{\bf b}\otimes\wedge^k\mathbb V$ is analogous.

First suppose that $k\in\N$. Since $H_0$ is bar-invariant by
\eqref{eq:H0inv}, we may embed $\wedge^k\mathbb W$ into $\mathbb
W^{\otimes k}$ as weakly based modules by sending $\mathcal W_h$ to
$M^{({1}^k)}_{h\cdot w^{(k)}_0} H_0$, for $h\in\Z^k_-$. Thus, we
have $\mathbb T^{\bf b}\otimes\wedge^k\mathbb W\subseteq \mathbb
T^{\bf b}\otimes\mathbb W^{\otimes k}$, and hence $\mathbb T^{\bf
b}\wotimes\wedge^k\mathbb W\subseteq \mathbb T^{\bf
b}\wotimes\mathbb W^{\otimes k}{\equiv}\widehat{{\mathbb T}}^{({\bf b},1^k)}$, where $\mathbb T^{\bf
b}\wotimes\wedge^k\mathbb W$ is a similarly defined $B$-completion
with respect to the Bruhat ordering of type $({\bf b},1^k)$,
following Definition \ref{def:completion:T}. Expanding
$M^{({1}^k)}_{h\cdot w^{(k)}_0} H_0$ in terms of the
$M^{({1}^k)}_e$'s, and using Propositions~\ref{bar:indep:order} and
\ref{bar:inv:Mf}, we conclude that
\begin{align}\label{bar:for:wedge}
\ov{M^{{\bf b},1}_f}=M^{{\bf b},1}_f+\sum_{g\prec_{({\bf
b},{1^k})}f}r_{gf}(q) M^{{\bf b},1}_g,
\end{align}
where $r_{gf}(q)\in\Z[q,q^{-1}]$, and the sum running over
$g\in\Z^{m+n}\times\Z^k_-$ is possibly infinite. It follows that we
have obtained a bar-involution $\ \bar{} : \mathbb T^{\bf
b}\wotimes\wedge^k\mathbb W \rightarrow \mathbb T^{\bf
b}\wotimes\wedge^k\mathbb W$, exactly as in Lemma~\ref{lem:barTb}.

Now let $k=\infty$. For $d\in\N$, let $\left[\mathbb T^{\bf
b}\otimes\wedge^{\infty}\mathbb W\right]_{\le |d|}$ be the subspace
of $\mathbb T^{\bf b}\otimes\wedge^{\infty}\mathbb W$ spanned by
vectors $M^{{\bf b},1}_f$, for $f\in\Z^{m+n}\times\Z^\infty_-$ with
$|f(i)|\le d$, for $i\in[m+n]\sqcup[\ul{d}]$. We let
$$
\pi'_d:\mathbb T^{\bf b}\otimes\wedge^{\infty}\mathbb W\rightarrow
\left[\mathbb T^{\bf b}\otimes\wedge^{\infty}\mathbb W\right]_{\le
|d|}
$$
be the natural projection map. Then we may use the $\ker\pi'_d$'s to
define the $A$-completion $\mathbb T^{\bf
b}\wt{\otimes}\wedge^{\infty}\mathbb W$ of $\mathbb T^{\bf
b}\otimes\wedge^{\infty}\mathbb W$, following Definition
\ref{def:completion:T}.

Let $M^{{\bf b},1}_f\in \left[\mathbb T^{\bf
b}\otimes\wedge^{\infty}\mathbb W\right]_{\le |d|}$. Using
\eqref{aux:eq:stab10} and an argument similar to its proof we can
show that, for $\ell\ge d$,
\begin{align}\label{aux:eq:stab11}
\pi'_d( \psi^{(\ell)}(M^{{\bf b},1}_f)) = \pi'_d(\psi^{(d)}(M^{{\bf
b},1}_f)).
\end{align}
It follows that the expression $\ov{M^{{\bf b},1}_f}$, defined as in
\eqref{def:bar:map} and \eqref{eq:barMf} on the tensor product of
the two weakly based modules $\mathbb T^{\bf b}$ and
$\wedge^\infty\mathbb W$, lies in $\mathbb T^{\bf
b}\wt{\otimes}\wedge^{\infty}\mathbb W$. Let $\mathbb T^{\bf
b}\wotimes\wedge^{\infty}\mathbb W$ be the $B$-completion of
$\mathbb T^{\bf b}\otimes\wedge^{\infty}\mathbb W$, following
Definition \ref{def:completion:T}.

For $h\in\Z^\infty_-$ and $k\in\N$ recall that $h_{[\ul{k}]}$
denotes the restriction of $h$ to $[\ul{k}]$. We define the
$\Q(q)$-linear {\em truncation map} $\texttt{Tr}: {\mathbb T}^{\bf
b}\otimes \wedge^\infty{\mathbb W}\rightarrow {\mathbb T}^{\bf
b}\otimes \wedge^k \mathbb W$, for $k\in\N$, as follows.  For $m\in
{\mathbb T}^{\bf b}$ and $h\in\Z^\infty_{-}$ we set
\begin{align*}
\texttt{Tr}(m\otimes \mc{W}_{h})=
\begin{cases}
m\otimes \mc{W}_{h_{[\ul{k}]}}, &\text{ if }h(\ul{i})=i,\text{ for }i\ge k+1,\\
0,&\text{ otherwise}.
\end{cases}
\end{align*}

\begin{lem}\label{trunc:comp:bruhat}
Let $k\in\N$. The truncation map extends naturally to a
$\Q(q)$-linear map $\texttt{Tr}:{\mathbb T}^{\bf
b}\wotimes\wedge^\infty\mathbb W \rightarrow{\mathbb T}^{\bf
b}\wotimes\wedge^k\mathbb W$.
\end{lem}

\begin{proof}
By definition of the $B$-completions, it is enough to prove that the
two Bruhat orderings on ${\mathbb T}^{\bf
b}\wotimes\wedge^\infty\mathbb W$ and ${\mathbb T}^{\bf
b}\wotimes\wedge^k\mathbb W$ are compatible under the truncation
map.

Let $f,g\in\Z^{m+n}\times\Z^\infty_{-}$ with $f\succ_{({\bf
b},{1}^\infty)} g$. Suppose first that $\texttt{Tr}(M^{{\bf
b},1}_f)=M^{{\bf b},1}_{f'}\not=0$ and $\texttt{Tr}(M^{{\bf
b},1}_g)=M^{{\bf b},1}_{g'}\not=0$. Then we have $f_{[m+n]\sqcup
[\ul{k}]}=f'$, $g_{[m+n]\sqcup [\ul{k}]}=g'$, and
$f(\ul{i})=g(\ul{i})=i$, for all $i\ge k+1$. It follows from the
very definition of the Bruhat ordering that $f'\succ_{({\bf
b},{1^k})}g'$.

Now suppose that $\texttt{Tr}(M^{{\bf b},1}_f)=0$ and $f\succ_{({\bf
b},{1}^\infty)}g$. If $f_{[\ul{\infty}]}=g_{[\ul{\infty}]}$, then
clearly $\texttt{Tr}(M^{{\bf b},1}_g)=0$. If not, then let $\ul{i}$
with $i$ minimal so that $f(\ul{i})\not=g(\ul{i})$. If $i\le k$,
then again $\texttt{Tr}(M^{{\bf b},1}_g)=0$. So suppose that $i>k$.
Since $f\succ_{({\bf b},{1}^\infty)}g$, we must have $g(\ul{i})<
f(\ul{i})\le i$, and so again we have $\texttt{Tr}(M^{{\bf
b},1}_g)=0$. Thus, we have shown that $\texttt{Tr}(M^{{\bf
b},1}_f)=0$ implies that $\texttt{Tr}(M^{{\bf b},1}_g)=0$.
\end{proof}

\begin{lem}\label{lem:trunc:bar}
$\texttt{Tr}:{\mathbb T}^{\bf b}\wotimes\wedge^\infty\mathbb W
\rightarrow{\mathbb T}^{\bf b}\wotimes\wedge^k\mathbb W$ commutes
with the bar maps, that is,
$$
\ov{\texttt{Tr}(M^{{\bf b},1}_f)}=\texttt{Tr}\left(\ov{M^{{\bf
b},1}_f}\right), \text{ for } f\in\Z^{m+n}\times\Z^\infty_-.
$$
Moreover, we have
\begin{equation}  \label{eq:barMfb1}
\ov{M^{{\bf b},1}_f} =M^{{\bf b},1}_f+\sum_{g\prec_{({\bf
b},1^\infty)}f} r_{gf}(q)M^{{\bf b},1}_g, \text{ for }
r_{gf}(q)\in\Z[q,q^{-1}].
\end{equation}
\end{lem}

\begin{proof}
Recall from \eqref{def:bar:map} and \eqref{eq:barMf} that the bar
maps on the tensor spaces ${\mathbb T}^{\bf b}\otimes\wedge^\infty\mathbb W$ and ${\mathbb T}^{\bf b}\otimes\wedge^k\mathbb W$ are defined by the formula $\ov{m\otimes
n}=\Theta(\ov{m}\otimes\ov{n})$, where
$\Theta=\sum_{\mu\in\sum\Z_+\Phi^+}\Theta_{\mu}\otimes\Theta_{-\mu}$
is given in \eqref{quasi:R:matrix} with $\Theta_{-\mu}\in\mc{U}^-$.
It is easily verified that the map
$\texttt{Tr}_{>}:\wedge^{\infty}\WW\rightarrow \wedge^k\WW$ defined
by
\begin{align*}
\texttt{Tr}_{>}(\mc{W}_h)=
\begin{cases}
\mc{W}_{h_{[\ul{k}]}}, &\text{ if }h(\ul{i})=i\;(\text{ for }i\ge k+1),\\
0,&\text{ otherwise},
\end{cases}, \quad \text{ for } h\in\Z^\infty_{-},
\end{align*}
is a $\mc{U}^-$-module homomorphism. Therefore we have
\begin{align*}
\sum_\mu &\texttt{Tr}\left(\Theta_{\mu}\ov{m}
\otimes\Theta_{-\mu} \mc{W}_h \right) =\sum_\mu\Theta_{\mu}\ov{m}\otimes\Theta_{-\mu}
 \texttt{Tr}_{>}(\mc{W}_h).
\end{align*}
From this, it follows that $\texttt{Tr}: {\mathbb T}^{\bf b}\wotimes
\wedge^\infty{\mathbb W}\rightarrow {\mathbb T}^{\bf b}\wotimes
\wedge^k \mathbb W$ commutes with the bar maps.
Now \eqref{eq:barMfb1} follows from \eqref{bar:for:wedge}.
\end{proof}
From Lemma~\ref{lem:trunc:bar} and the bar-involutions on $\mathbb
T^{\bf b}\wotimes\wedge^k\mathbb W$ for $k\in \N$, we obtain a
bar-involution $\ \bar{} :{\mathbb T}^{\bf
b}\wotimes\wedge^\infty\mathbb W \rightarrow {\mathbb T}^{\bf
b}\wotimes\wedge^\infty\mathbb W$. Summarizing the above, and
applying Lemmas \ref{lem:lusztig} and \ref{lem:trunc:bar} we have
proved the following.

\begin{prop}  \label{prop:TbkW}
Let $k\in\N\cup\{\infty\}$. The bar map $\ \bar{} :{\mathbb T}^{\bf
b}\wotimes\wedge^k\mathbb W \rightarrow {\mathbb T}^{\bf
b}\wotimes\wedge^k\mathbb W$ is an involution. The space ${\mathbb
T}^{\bf b}\wotimes\wedge^k\mathbb W$ has unique bar-invariant topological bases
\begin{align*}
\{T^{{\bf b},1}_f|f\in\Z^{m+n}\times\Z^k_-\}\text{ and }\{L^{{\bf
b},1}_f|f\in\Z^{m+n}\times\Z^k_-\}
\end{align*}
such that
\begin{align*}
T^{{\bf b},1}_f=M^{{\bf b},1}_f+\sum_{g\prec_{({\bf b},{1^k})} f}
t^{{\bf b},1}_{gf}(q) M^{{\bf b},1}_g,
  \qquad
L^{{\bf b},1}_f=M^{{\bf b},1}_f+\sum_{g\prec_{({\bf b},{1^k})} f}
\ell^{{\bf b},1}_{gf}(q) M^{{\bf b},1}_g,
\end{align*}
with $t^{{\bf b},1}_{gf}(q)\in q\Z[q]$, and $\ell^{{\bf
b},1}_{gf}(q)\in q^{-1}\Z[q^{-1}]$. (We will write $t_{ff}^{{\bf
b},1}(q)=\ell_{ff}^{{\bf b},1}(q)=1$, $t_{gf}^{{\bf
b},1}=\ell_{gf}^{{\bf b},1}=0$, for $g \npreceq_{({\bf
b},{1^k})}f$.)
\end{prop}

We call $\{T^{{\bf b},1}_f\}$ and $\{L^{{\bf b},1}_f\}$ the {\em
canonical and dual canonical bases} of ${\mathbb T}^{\bf
b}\wotimes\wedge^k\mathbb W$, respectively. We shall use
$f^{\ul k}\in\Z^{m+n}\times\Z^k_-$ as a short-hand notation for the
restriction $f_{[m+n]\sqcup[\ul{k}]}$ of a function $f \in
\Z^{m+n}\times \Z_-^\infty$.

\begin{prop}\label{prop:can:trunc}
Let $k\in\N$. The truncation map $\texttt{Tr}:\mathbb T^{\bf
b}\wotimes\wedge^\infty\mathbb W\rightarrow \mathbb T^{\bf
b}\wotimes\wedge^k\mathbb W$ preserves the standard, canonical, and
dual canonical bases in the following sense: for $Y=M,L,T$ and
$f\in\Z^{m+n}\times\Z^\infty_{-}$ we have
\begin{align*}
\texttt{Tr}\left( Y^{{\bf b},1}_{f} \right)=
\begin{cases}
Y^{{\bf b},1}_{f^{\ul k}},
&\text{ if }f(\ul{i})=i,\text{ for }i\ge k+1,\\
0,&\text{ otherwise}.
\end{cases}
\end{align*}
Consequently, we have $t^{{\bf b},1}_{gf}(q)=t^{{\bf
b},1}_{g^{\ul k}f^{\ul k}}(q)$ and $\ell^{{\bf b},1}_{gf}(q)=\ell^{{\bf
b},1}_{g^{\ul k}f^{\ul k}}(q)$, for $g, f\in \Z^{m+n}\times \Z_-^\infty$ such
that $f(\ul{i})=g(\ul{i})=i,$ for $i\ge k+1$.
\end{prop}

\begin{proof}
By definition the statement is true for $Y=M$. By Lemmas
\ref{trunc:comp:bruhat} and \ref{lem:trunc:bar} the map
$\texttt{Tr}:{\mathbb T}^{\bf b}\wotimes\wedge^\infty\mathbb
W\rightarrow{\mathbb T}^{\bf b}\wotimes\wedge^k\mathbb W$ is
compatible with canonical and dual canonical bases of these two
spaces.
\end{proof}

The constructions and statements when replacing $\wedge^k\mathbb W$
by $\wedge^k\mathbb V$, for $k\in\N\cup\{\infty\}$ are entirely
analogous, and so we will skip the analogous proofs. We construct a
$B$-completion ${\mathbb T}^{\bf b}\wotimes\wedge^k\mathbb V$. For
$k\in\N$, the truncation map $\texttt{Tr}:{\mathbb T}^{\bf
b}\otimes\wedge^{\infty}\mathbb V\rightarrow {\mathbb T}^{\bf
b}\otimes\wedge^k\mathbb V$, is defined by
\begin{align*}
\texttt{Tr}(m\otimes \mc{V}_h)=
\begin{cases}
m\otimes \mc{V}_{h_{[\ul{k}]}},
&\text{ if }h(\ul{i})=1-i,\text{ for }i\ge k+1,\\
0,&\text{ otherwise},
\end{cases}
\end{align*}
where $m\in {\mathbb T}^{\bf b}$ and $h\in\Z^\infty_{+}$. The map
$\texttt{Tr}$ extends to the $B$-completions. The following is a
$\wedge^k\mathbb V$-analogue of Proposition~\ref{prop:TbkW}.

\begin{prop}  \label{prop:CBdcb}
Let $k\in\N\cup\{\infty\}$. We have a bar-involution $\ \bar{}
:{\mathbb T}^{\bf b}\wotimes\wedge^k\mathbb V \rightarrow {\mathbb
T}^{\bf b}\wotimes\wedge^k\mathbb V$. The space ${\mathbb T}^{\bf
b}\wotimes\wedge^k\mathbb V$ has unique bar-invariant topological bases
\begin{align*}
\{T^{{\bf b},0}_f|f\in\Z^{m+n}\times\Z^k_+\}\text{ and }\{L^{{\bf
b},0}_f|f\in\Z^{m+n}\times\Z^k_+\}
\end{align*}
such that
\begin{align*}
T^{{\bf b},0}_f=M^{{\bf b},0}_f+\sum_{g\prec_{({\bf b},{0^k})} f}
t^{{\bf b},0}_{gf}(q) M^{{\bf b},0}_g,
 \qquad
L^{{\bf b},0}_f=M^{{\bf b},0}_f+\sum_{g\prec_{({\bf b},{0^k})} f}
\ell^{{\bf b},0}_{gf}(q) M^{{\bf b},0}_g,
\end{align*}
with $t^{{\bf b},0}_{gf}(q)\in q\Z[q]$, and $\ell^{{\bf
b},0}_{gf}(q)\in q^{-1}\Z[q^{-1}]$. (We will  write $t_{ff}^{{\bf
b},0}(q)=\ell_{ff}^{{\bf b},0}(q)=1$, $t_{gf}^{{\bf
b},0}=\ell_{gf}^{{\bf b},0}=0$, for $g\npreceq_{({\bf
b},{0^k})}f$.)
\end{prop}

We call $\{T^{{\bf b},0}_f\}$ and $\{L^{{\bf b},0}_f\}$ the {\em
canonical and dual canonical bases} of ${\mathbb T}^{\bf
b}\wotimes\wedge^k\mathbb V$, respectively.

We shall also use $f^{\ul k}\in\Z^{m+n}\times\Z^k_+$ as a short-hand
notation for the restriction $f_{[m+n]\sqcup[\ul{k}]}$ of a function
$f \in \Z^{m+n}\times \Z_+^\infty$. The following is a
$\wedge^k\mathbb V$-analogue of Lemma~\ref{lem:trunc:bar} and
Proposition~\ref{prop:can:trunc}.

\begin{prop}\label{prop:trunc:fock}
The truncation map $\texttt{Tr}:{\mathbb T}^{\bf
b}\wotimes\wedge^{\infty}\mathbb V\rightarrow {\mathbb T}^{\bf
b}\wotimes\wedge^k\mathbb V$ commutes with the bar involutions.
Moreover, the truncation map $\texttt{Tr}$ preserves the standard,
canonical, and dual canonical bases; that is, for $Y=M,L,T$, and for
$f\in\Z^{m+n}\times\Z^\infty_{+}$, we have
\begin{align*}
\texttt{Tr}\left( Y^{{\bf b},0}_{f} \right)=
\begin{cases}
Y^{{\bf b},0}_{f^{\ul k}},
&\text{ if }f(\ul{i})=1-i,\text{ for }i\ge k+1,\\
0,&\text{ otherwise}.
\end{cases}
\end{align*}
Consequently, we have $t^{{\bf b},0}_{gf}(q)=t^{{\bf
b},0}_{g^{\ul k}f^{\ul k}}(q)$ and $\ell^{{\bf b},0}_{gf}(q)=\ell^{{\bf
b},0}_{g^{\ul k}f^{\ul k}}(q)$, for $g, f\in \Z^{m+n}\times \Z_+^\infty$ such
that $f(\ul{i})=g(\ul{i})=1-i,$ for $i\ge k+1$.
\end{prop}

\begin{definition}
The polynomials $\ell^{{\bf b},0}_{gf}(q)$, $\ell^{{\bf
b},1}_{gf}(q)$, $t^{{\bf b},0}_{gf}(q)$, $t^{{\bf b},1}_{gf}(q)$ are
called {\em Brundan-Kazhdan-Lusztig (BKL) polynomials}. (They
include $\ell^{\bf b}_{gf}(q)$ and $t^{\bf b}_{gf}(q)$ in
Definition~\ref{def:BKLpolyn} as special cases).
\end{definition}
Note that the BKL polynomials reduce to the usual (parabolic)
Kazhdan-Lusztig polynomials when the underlying Fock spaces involve
only $\mathbb V$ (or only $\WW$).

\subsection{Combinatorial super duality}

Recall from Proposition~\ref{wedgeV:isom:wedgeW} the isomorphism of
$U_q(\mf{sl}_\infty)$-modules $\natural:\wedge^\infty\mathbb
V\rightarrow\wedge^\infty\mathbb W$ defined by
$\natural(\vert\la\rangle)={\vert\la'_*\rangle}$.  This map induces
an isomorphism
\begin{align*}
\natural_{\bf b}\stackrel{\text{def}}{=} {1_{{\bf
b}}\otimes\natural}:
 {{\mathbb T}^{\bf b}}\otimes\wedge^\infty\mathbb V {\longrightarrow}
{{\mathbb T}^{\bf b}}\otimes\wedge^\infty\mathbb W,
\end{align*}
where $1_{{\bf b}}$ denotes the identity map on ${{\mathbb T}^{\bf
b}}$. Let $f\in\Z^{m+n}\times\Z^\infty_{+}$. There exists a unique
$\la\in\mc{P}$ such that $|\la\rangle=\mc V_{f_{[\ul{\infty}]}}$. We
define $f^\natural$ to be the unique element in
$\Z^{m+n}\times\Z^\infty_{-}$ determined by $f^\natural(i)=f(i)$,
for $i\in[m+n]$, and $\mc
W_{f^\natural_{[\ul{\infty}]}}=|\la'_*\rangle$. The assignment
$f\mapsto f^\natural$ gives a bijection (cf. \cite{CWZ})
\begin{equation}  \label{naturalbi}
\natural: \Z^{m+n}\times\Z^\infty_+ \longrightarrow
\Z^{m+n}\times\Z^\infty_-.
\end{equation}

The following is the combinatorial counterpart of the super duality
in Theorem~\ref{thm:SD} later on representation theory.

\begin{thm}\label{TwedgeV:isom:TwedgeW}
\begin{enumerate}
\item
The isomorphism $\natural_{\bf b}$ respects the Bruhat orderings,
and hence extends to an isomorphism of the $B$-completions
$\natural_{\bf b}:{{\mathbb T}^{\bf b}}\wotimes\wedge^\infty\mathbb
V\rightarrow {{\mathbb T}^{\bf b}}\wotimes\wedge^\infty\mathbb W$.

\item
 $\natural_{\bf b}$ commutes with the bar involutions.

 \item
$\natural_{\bf b}$ preserves the canonical and dual canonical bases.
More precisely, we have $ \natural_{\bf b}(M^{{\bf b},0}_f)=M^{{\bf
b},1}_{f^\natural},
 \;\;
\natural_{\bf b}(T^{{\bf b},0}_f)=T^{{\bf b},1}_{f^\natural}
 \text{ and }\
\natural_{\bf b}(L^{{\bf b},0}_f)=L^{{\bf b},1}_{f^\natural}, \text{
for }f\in \Z^{m+n}\times\Z^\infty_+. $

\item
We have the following identifications of BKL polynomials:
$\ell^{{\bf b},0}_{gf}(q)=\ell^{{\bf
b},1}_{{g^\natural}{f^\natural}}(q)$, and $t^{{\bf
b},0}_{gf}(q)=t^{{\bf b},1}_{{g^\natural}{f^\natural}}(q)$, for all
$g, f\in \Z^{m+n}\times\Z^\infty_+. $
\end{enumerate}
\end{thm}

\begin{proof}
We first prove (2)--(4) by assuming that (1) holds. Since
$\wedge^\infty\mathbb V$ and $\wedge^\infty\mathbb W$ are isomorphic
$U_q(\mf{sl}_\infty)$-modules under $\natural$, $\natural_{\bf b}$
is compatible with the quasi-$\mc{R}$-matrices. Thus, $\natural_{\bf
b}$ commutes with the bar involutions (see \eqref{def:bar:map}),
whence (2). It follows by definition that $\natural_{\bf b}(M^{{\bf
b},0}_f)=M^{{\bf b},1}_{f^\natural}$. The commutativity of
$\natural_{\bf b}$ with the bar-involutions implies that canonical
and dual canonical bases are sent by $\natural_{\bf b}$ to bar
invariant elements. Now by the compatibility of the two Bruhat
orderings in (1), $\natural_{\bf b}(T^{{\bf b},0}_f)$ and $\natural_{\bf
b}(L^{{\bf b},0}_f)$ satisfy the same characterization properties as
the canonical and dual canonical basis elements $T^{{\bf
b},1}_{f^\natural}$ and $L^{{\bf b},1}_{f^\natural}$ respectively,
and hence (3) follows. Now (3) clearly implies (4).

It remains to show (1). To this end, let
$f,g\in\Z^{m+n}\times\Z^\infty_{+}$. We shall show that $f\succeq_{({\bf
b},{0^\infty})}g$ is equivalent to $f^\natural\succeq_{({\bf
b},{1}^\infty)}g^\natural$. Recall that we denote the restriction of $f$
to a subset $I$ by $f_I$.

\cite[Lemma 6.2]{CWZ} says that
\begin{equation} \label{cwzlem}
\{f(\ul{j})\vert j\in\N\}\sqcup\{f^\natural(\ul{j})\vert
j\in\N\}=\Z,
\end{equation}
(and similar claim when replacing $f$ by $g$). Choose $N\gg 0$ so
that $f(\ul{t})=g(\ul{t})=-t+1$,
$f^\natural(\ul{t})=g^\natural(\ul{t})=t$, for all $t> N$. By our
choice of $N$ and \eqref{cwzlem}, we have
\begin{align*}
\{f(\ul{j})\vert1\le j\le N \}\sqcup\{f^\natural(\ul{j})\vert1\le
j\le N \}
 &=\{g(\ul{j})\vert1\le j\le N\}\sqcup\{g^\natural(\ul{j})\vert1\le j\le N \}
  \\
 &=\{-N+1,-N+2,\ldots,N-1,N\}.
\end{align*}
From this we conclude that, for $-N+1 \le d\le N$,
\begin{align}\label{aux:201}
\sharp_{(0^N)}(f_{[\ul{N}]},d,\ul{1})-\sharp_{(1^N)}(f_{[\ul{N}]}^\natural,d,\ul{1})=N+d=
\sharp_{(0^N)}(g_{[\ul{N}]},d,\ul{1})-\sharp_{(1^N)}(g_{[\ul{N}]}^\natural,d,\ul{1}).
\end{align}
Here we recall the notation $\sharp_{\bf b}$ from \eqref{aux:sharp:Bruhat}
and that a Bruhat ordering $\succeq_{\bf b}$ is characterized in terms of $\sharp_{\bf b}$ by
\eqref{eq:sharpBr}.

The condition $f\succeq_{({\bf b},{0^\infty})}g$ is equivalent
to $f_{[m+n]\cup[\ul{N}]}\succeq_{({\bf
b},{0}^N)}g_{[m+n]\cup[\ul{N}]}$, which is equivalent by
\eqref{eq:sharpBr} to
\begin{align}  \label{eq:mnN}
& \sharp_{({\bf b},0^N)}(f_{[m+n]\cup[\ul{N}]},d,j)\ge \sharp_{({\bf
b},0^N)}(g_{[m+n]\cup[\ul{N}]},d,j), \quad \forall j\in[m+n] \cup
[\ul{N}], \forall d\in\Z, \\
 & \text{with equality holding for $j=1$.} \notag
\end{align}
On the other hand, $f^\natural\succeq_{({\bf
b},{1}^\infty)}g^\natural$ is equivalent to
$f^\natural_{[m+n]\cup[\ul{N}]}\succeq_{({\bf
b},{1}^N)}g^\natural_{[m+n]\cup[\ul{N}]}$, which in turn, via
\eqref{eq:sharpBr}, is equivalent to
\begin{align}  \label{eq:mnN:1}
& \sharp_{({\bf b},1^N)}(f^\natural_{[m+n]\cup[\ul{N}]},d,j)\ge
\sharp_{({\bf b},1^N)}(g^\natural_{[m+n]\cup[\ul{N}]},d,j), \quad
\forall j\in[m+n] \cup [\ul{N}], \forall d\in\Z, \\
 & \text{with equality holding for $j=1$.} \notag
\end{align}
Thus, we are reduced to show the equivalence between \eqref{eq:mnN}
and \eqref{eq:mnN:1}. We separate into several cases below.

(a) Consider the case for $j\in[\ul{N}]$ (and $d$ arbitrary). Denote
by $\la=(\la_1,\la_2,\ldots)$ the partition defined by letting
$\la_i =f(\ul{i}) +i-1$, and by $\mu=(\mu_1,\mu_2,\ldots)$ the
partition defined by letting $\mu_i=g(\ul{i})+i-1$, for $i\ge 1$.
Then \eqref{eq:mnN} in the case for $j\in[\ul{N}]$ is equivalent to
$\la\subseteq\mu$. This is equivalent to $\la'\subseteq\mu'$, which
in turn is equivalent to \eqref{eq:mnN:1} in the case for $j\in[\ul{N}]$.

(b) Consider the case for $d<-N+1$ and $j \in [m+n]$. Then
$\sharp_{({\bf b},0^N)}(f_{[m+n]\cup[\ul{N}]},d,j)=\sharp_{({\bf b},1^N)}(f^\natural_{[m+n]\cup[\ul{N}]},d,j)$,
and similarly, $\sharp_{({\bf b},0^N)}(g_{[m+n]\cup[\ul{N}]},d,j)=
\sharp_{({\bf b},1^N)}(g^\natural_{[m+n]\cup[\ul{N}]},d,j)$. Thus, it follows that
\eqref{eq:mnN} and \eqref{eq:mnN:1} are equivalent in this case.

(c) Now consider the case for $d> N$ and $j \in [m+n]$. We have
\begin{align*}
\sharp_{({\bf b},0^N)}(f_{[m+n]\cup[\ul{N}]},d,j)=\sharp_{{\bf b}}(f_{[m+n]},d,j)+N,\\
\sharp_{({\bf b},1^N)}(f^\natural_{[m+n]\cup[\ul{N}]},d,j)=\sharp_{{\bf b}}(f_{[m+n]},d,j)-N,\\
\sharp_{({\bf b},0^N)}(g_{[m+n]\cup[\ul{N}]},d,j)=\sharp_{{\bf b}}(g_{[m+n]},d,j)+N,\\
\sharp_{({\bf b},1^N)}(g^\natural_{[m+n]\cup[\ul{N}]},d,j)=\sharp_{{\bf b}}(g_{[m+n]},d,j)-N.
\end{align*}
Thus, \eqref{eq:mnN} and \eqref{eq:mnN:1} are equivalent in this case as well.

(d) Finally, consider the case for $-N+1 \le d\le N$ and
$j\in[m+n]$. We have the following equivalent statements:
 {\allowdisplaybreaks
\begin{align*}
\sharp_{({\bf b},0^N)}& (f_{[m+n]\cup[\ul{N}]}, d,j)\ge
 \sharp_{({\bf b},0^N)}(g_{[m+n]\cup[\ul{N}]},d,j)
  \\ \Longleftrightarrow \;
&\sharp_{{\bf b}}(f_{[m+n]},d,j)+\sharp_{(0^N)}(f_{[\ul{N}]},d,\ul{1})\ge
 \sharp_{{\bf b}}(g_{[m+n]},d,j)+\sharp_{(0^N)}(g_{[\ul{N}]},d,\ul{1})
 \\ \stackrel{\eqref{aux:201}}{\Longleftrightarrow}\;
&\sharp_{{\bf b}}(f_{[m+n]},d,j)+\sharp_{(1^N)}(f_{[\ul{N}]}^\natural,d,\ul{1})\ge
 \sharp_{{\bf b}}(g_{[m+n]},d,j)+\sharp_{(1^N)}(g_{[\ul{N}]}^\natural,d,\ul{1})
 \\ \Longleftrightarrow\;
&\sharp_{({\bf b},1^N)}(f_{[m+n]\cup[\ul{N}]}^\natural,d,j)\ge
\sharp_{({\bf b},1^N)}(g_{[m+n]\cup[\ul{N}]}^\natural,d,j).
\end{align*}}
We observe that in (b)--(d) when $j=1$ we have equality in \eqref{eq:mnN} if and only if we have equality in \eqref{eq:mnN:1}.
Summarizing (a)--(d), we have shown the equivalence of
\eqref{eq:mnN} and \eqref{eq:mnN:1}, and hence, by
\eqref{eq:sharpBr}, the equivalence between $f\succeq_{({\bf
b},{0^\infty})}g$ and $f^\natural\succeq_{({\bf
b},{1}^\infty)}g^\natural$.

The compatibility of $\natural_{\bf b}$ with the two Bruhat
orderings implies that $\natural_{\bf b}$ extends to the respective
$B$-completions, $\natural_{\bf b}:{{\mathbb T}^{\bf
b}}\wotimes\wedge^\infty\mathbb V\longrightarrow {{\mathbb T}^{\bf
b}}\wotimes\wedge^\infty\mathbb W$. This completes the proof of (1)
and hence the proof of the theorem.
\end{proof}

\subsection{Tensor versus $q$-wedges}
 \label{sec:comp:para:tensor}

Let ${\bf b}$ be a fixed $0^m1^n$-sequence and let $k\in\N$. We
shall compare canonical and dual canonical bases of the space
$\mathbb T^{\bf b}\wotimes\wedge^k\mathbb V$ (respectively, $\mathbb
T^{\bf b}\wotimes\wedge^k\mathbb W$) with those of $\mathbb T^{\bf
b}\wotimes\mathbb V^{\otimes k}$ (respectively, $\mathbb T^{\bf
b}\wotimes\mathbb W^{\otimes k}$).  We shall only do this comparison
for $\mathbb T^{\bf b}\wotimes\wedge^k\mathbb V$ and $\mathbb T^{\bf
b}\wotimes\mathbb V^{\otimes k}$ in detail, the other case being
analogous.

For $k\in\N$, recall that we may regard $\mathbb T^{\bf
b}\wotimes\wedge^k\mathbb V\subseteq \mathbb T^{\bf b}\wotimes
\mathbb V^{\otimes k}$ with compatible bar involutions via the
identification $\mathcal V_h$ with $M^{({0}^k)}_{h\cdot w^{(k)}_0}
H_0$, for $h\in\Z^k_+$.
%
%%%%%%%%%%
 \iffalse%
%%%%%%%%%%
Fix $f\in\Z^{m+n}\times\Z^k_+$. If we write the dual canonical basis
element corresponding to $f$ in $\mathbb T^{\bf b}\wotimes\mathbb
V^{\otimes k}$ as
\begin{align*}
L_f^{({\bf b},0^k)}=\sum_{g\in\Z^{m+n}\times \Z^k}\ell^{{({\bf
b},{0^k})}}_{gf}(q) M^{{({\bf b},0^k)}}_{g},
\end{align*}
then the corresponding dual canonical basis element in $\mathbb
T^{\bf b}\wotimes\wedge^k\mathbb V$ is given by
\begin{align*}
L_f^{{\bf b},0}=\sum_{g\in\Z^{m+n}\times\Z^k_+}\ell^{{({\bf
b},{0^k})}}_{gf}(q) M^{{{\bf b},0}}_{g}.
\end{align*}
%%%%%%%%%
  \fi%%%%
%%%%%%%%%
%

Let $f\in\Z^{m+n}\times\Z^k_+$. As before, we write the dual
canonical basis element $L_f^{({\bf b},0^k)}$ in ${\mathbb T}^{\bf
b}\wotimes\mathbb V^{\otimes k}$ and the corresponding dual
canonical basis element $L_f^{{\bf b},0}$ in ${\mathbb T}^{\bf
b}\wotimes\wedge^k\mathbb V$ as
\begin{align}
L_f^{({\bf b},0^k)} &=\sum_{g\in\Z^{m+n}\times \Z^k}\ell^{{({\bf
b},{0^k})}}_{gf}(q) M^{{({\bf b},{0^k})}}_{g},
 \label{aux:108a} \\
 \qquad
L_f^{{\bf b},0}&= \sum_{g\in\Z^{m+n}\times \Z_+^k}\ell^{{{\bf
b},{0}}}_{gf}(q) M^{{{\bf b},0}}_{g}.
 \label{aux:108b}
\end{align}
The following proposition states that the BKL-polynomials $\ell$'s
in $\mathbb T^{\bf b}\wotimes\mathbb\wedge^k \mathbb V$ coincide
with their counterparts in $\mathbb T^{\bf b}\wotimes\mathbb
V^{\otimes k}$.

\begin{prop}\label{thm:aux:11}
Let $f,g \in\Z^{m+n}\times\Z^k_+$. Then $\ell^{{{\bf
b},{0}}}_{gf}(q) =\ell^{{({\bf b},{0^k})}}_{gf}(q)$.
\end{prop}

\begin{proof}
The argument below is adapted from the one given in \cite[Page
205]{Br}.

Via the identification $\mathcal V_{g_{[\ul{k}]}}\equiv M^{({0}^k)}_{g_{[\ul{k}]}\cdot
w^{(k)}_0} H_0$, \eqref{formula:H0} and \eqref{eq:heckeaction}, we
write $M^{{{\bf b},0}}_{g}=M^{{({\bf b},{0^k})}}_{g}+(*)$, where
$(*)$ is a $q^{-1}\Z[q^{-1}]$-linear combination of $M^{({\bf
b},0^k)}_h$, with $h$ satisfying $h\prec_{({\bf b},0^k)}g$ and
$h\not\in \Z^{m+n}\times\Z^k_+$. When combining this with
\eqref{aux:108b}, we can write $L_f^{{\bf b},0}=M^{({\bf
b},0^k)}_f+(**)$, where $(**)$ is a $q^{-1}\Z[q^{-1}]$-linear
combination of $M^{({\bf b},0^k)}_g$, with $g$ satisfying
$g\prec_{({\bf b},0^k)}f$. Since $L_f^{{\bf b},0}$ is also
bar-invariant, this expression equals $L_f^{({\bf b},0^k)}$, by the
uniqueness of the dual canonical basis. The proposition now follows
by comparing the coefficients of $M^{({\bf b},0^k)}_g$, for
$g\in\Z^{m+n}\times\Z^k_+$, in this expression and in
\eqref{aux:108a}.
\end{proof}

Let $f\in\Z^{m+n}\times\Z^k_+$. Similarly as before we write the canonical basis element $T^{({\bf
b},{0^k})}_f$ in ${\mathbb T}^{\bf b}\wotimes\mathbb V^{\otimes k}$
and the canonical basis element $T^{{\bf b},{0}}_f$ in ${\mathbb
T}^{\bf b}\wotimes\wedge^k\mathbb V$ respectively as
\begin{align}
T^{({\bf b},{0^k})}_f
 & =\sum_{g\in\Z^{m+n}\times \Z^k}t^{{({\bf
b},{0^k})}}_{gf}(q) M^{{({\bf b},{0^k})}}_{g},
   \label{TTtta}\\ \quad
T^{{\bf b},{0}}_f
 & =\sum_{g\in\Z^{m+n}\times \Z_+^k}t^{{{\bf
b},{0}}}_{gf}(q) M^{{{\bf b},0}}_{g}.
  \label{TTttb}
\end{align}
Recall $w^{(k)}_0$ is the longest element in $\mf S_k$.

\begin{prop}\label{thm:aux:12}
For $f, g\in\Z^{m+n}\times\Z^k_+$, we have
$$
t^{{{\bf b},{0}}}_{gf}(q)
=\sum_{\tau\in\mf{S}_k}(-q)^{\ell(w^{(k)}_0\tau )}t^{({\bf
b},0^k)}_{g\cdot\tau,f\cdot w^{(k)}_0}(q).
$$
\end{prop}

\begin{proof}
We write $w_0=w_0^{(k)}$ in this proof. We identify $\mathcal V_h$
in $\wedge^k\mathbb V$ with $M^{({0}^k)}_{h\cdot w_0} H_0\in\mathbb
V^{\otimes k}$, for $h\in\Z^k_+$, so that the ${\mathbb T}^{\bf
b}\wotimes\wedge^k\mathbb V$ may be identified with a subspace of
${\mathbb T}^{\bf b}\wotimes\mathbb V^{\otimes k}$. Then, as in
\cite[Lemma~3.8]{Br}, we have
\begin{align*}
T^{{\bf b},0}_f= T^{({\bf b},0^k)}_{f\cdot w_0} H_0.
\end{align*}
A straightforward variation of \cite[Lemma 3.4]{Br} using
\eqref{TTtta} gives us
\begin{align*}
T^{{\bf b},0}_f &= T^{({\bf b},0^k)}_{f\cdot w_0} H_0
 = \sum_{g} t^{({\bf b},0^k)}_{g,f\cdot w_0}M^{({\bf b},0^k)}_{g} H_0
 \\
&= \sum_{\tau\in \mf{S}_k}\sum_{g\in\Z^{m+n}\times\Z^k_+}
t^{({\bf b},0^k)}_{g\cdot\tau,f\cdot w_0} M^{({\bf b},0^k)}_{g\cdot\tau} H_0
 \\
&=\sum_{\tau\in \mf{S}_k}\sum_{g\in\Z^{m+n}\times\Z^k_+}
 t^{({\bf b},0^k)}_{g\cdot\tau,f\cdot w_0}
 (-q)^{\ell(\tau^{-1}w_0)} M^{{\bf b},0}_{g}
  \\
&=\sum_{g\in\Z^{m+n}\times\Z^k_+} \left(\sum_{\tau\in \mf{S}_k}
t^{({\bf b},0^k)}_{g\cdot\tau,f\cdot w_0}
(-q)^{\ell(w_0\tau)}\right) M^{{\bf b},0}_{g}.
\end{align*}
The proposition now follows by comparing with \eqref{TTttb}.
\end{proof}

\begin{rem}
Entirely analogous statements as Propositions \ref{thm:aux:11} and
\ref{thm:aux:12} hold when comparing the canonical and dual
canonical bases in ${\mathbb T}^{\bf b}\wotimes\mathbb W^{\otimes
k}$ with their counterparts in ${\mathbb T}^{\bf b}\wotimes
\wedge^k\mathbb W$. Such relations between Kazhdan-Lusztig
polynomials and their parabolic versions were due to \cite{Deo} (see
\cite[Proposition~3.4]{So}).
 \end{rem}

\section{Canonical bases for adjacent Fock spaces}
\label{sec:adjacent}

In this section, we develop a new approach to compare the canonical
as well as dual canonical bases in Fock spaces ${\mathbb T}^{\bf b}$
and ${\mathbb T}^{\bf b'}$, for adjacent $0^m1^n$-sequences $\bf b$
and $\bf b'$.

\subsection{Rank $2$ cases}
 \label{subsec:VW-WV}

Consider the Fock space $\mathbb T^{({0},{1})}=\mathbb
V\otimes\mathbb W$ and its $B$-completion $\mathbb V\wotimes\mathbb
W$ with respect to the Bruhat ordering $\succeq_{({0},{1})}$. It has
the standard monomial basis $M_f :=M^{({0},{1})}_f =v_{f(1)}\otimes
w_{f(2)}$, canonical basis $T_f :=T^{({0},{1})}_f$, and dual
canonical basis $L_f:= L^{({0},{1})}_f$, for $f\in\Z^{1+1}$. For
$f\in \Z^{1+1}$ with $f(1)=f(2)$ we introduce the following elements
in $\Z^{1+1}$:
\begin{align}\label{eq:updown}
\begin{split}
f^{\downarrow}(1)=f^{\downarrow}(2)=f(1)-1,\\
f^{\uparrow}(1)=f^{\uparrow}(2)=f(1)+1.
\end{split}
\end{align}
Define $f^{\downarrow
k}=(((f^{\downarrow})^{\downarrow})\cdots{}^\downarrow)$,
i.e.,~applying the operation $\downarrow$ $k$ times to $f$.
Similarly, we introduce the notation $f^{\uparrow k}$.

\begin{lem}\label{lem:formula:can:VW}  \cite[Example~2.19]{Br}
We have the following formulas for the canonical and dual canonical
bases in $\mathbb V\wotimes\mathbb W$:
\begin{align*}
&L_f=\begin{cases} M_f+\sum_{k=1}^\infty (-q)^{-k} M_{f^{\downarrow
k}},&\text{ if }f(1)=f(2),
  \\
M_f,&\text{ if }f(1)\not=f(2);
\end{cases}
  \\
&T_f=\begin{cases}
 M_f+q M_{f^\downarrow},&\text{ if }f(1)=f(2),
  \\
 M_f,&\text{ if }f(1)\not=f(2).
\end{cases}
\end{align*}
\end{lem}

Therefore we have the inversion formula
\begin{align}  \label{eq:inv}
M_f=\begin{cases}
 L_f+q^{-1}L_{f^{\downarrow}},&\text{ if
}f(1)=f(2),
  \\
L_f,&\text{ if }f(1)\not=f(2).
\end{cases}
\end{align}

\begin{definition} \label{def:UL}
Let $\mathbb L$ be the $\Q(q)$-subspace of $\mathbb V\wotimes\mathbb
W$ spanned by $\{L_f\vert f\in\Z^{1+1}\}$. Also, let $\mathbb U$ be
the $\Q(q)$-subspace of $\mathbb V{\otimes}\mathbb W$ spanned by
$\{T_f\vert f\in\Z^{1+1}\}$.
\end{definition}
It can be checked directly that applying the Chevalley generators
$E_a$ and $F_a$ to $T_f$ produces a finite linear combination of
$T_g$'s. This implies that $\mathbb U$ is a
$U_q(\gl_\infty)$-module, and hence $(\mathbb U,\{T_f\vert
f\in\Z^{1+1}\})$ is a weakly based module. Similarly, applying the
Chevalley generators $E_a$ and $F_a$  to $L_f$ produces a finite
linear combination of $M_g$'s. This implies by the inversion formula
\eqref{eq:inv} that $(\mathbb L,\{L_f\vert f\in\Z^{1+1}\})$ is also
a weakly based $U_q(\gl_\infty)$-module.

Next consider the $\Q(q)$-space $\mathbb T^{({1},{0})}{=}\mathbb
W{\otimes}\mathbb V$ and its $B$-completion $\mathbb
W\wotimes\mathbb V$ with respect to the Bruhat ordering
$\succeq_{({1},{0})}$. It has the standard monomial basis $M'_f:=
M^{({1},{0})}_f =w_{f(1)}\otimes v_{f(2)}$, canonical basis $T'_f
:=T^{(1,0)}_f$, and dual canonical basis $L'_f :=L^{(1,0)}_f$, for
$f\in\Z^{1+1}$. We have the following formulas in $\mathbb
W\wotimes\mathbb V$ similar to Lemma~\ref{lem:formula:can:VW}:
\begin{align*}
&L'_f=
\begin{cases}
M'_f+\sum_{k=1}^\infty (-q)^{-k}
M'_{(f^{\uparrow k})},&\text{ if }f(1)=f(2),
  \\
M'_f,&\text{ if }f(1)\not=f(2);
\end{cases}
  \\
&T'_f=
\begin{cases}
M'_f+q M'_{f^\uparrow},&\text{ if }f(1)=f(2),
  \\
M'_f,&\text{ if }f(1)\not=f(2).
\end{cases}
\end{align*}

\begin{definition} \label{def:UL'}
Let $\mathbb L'$ and $\mathbb U'$ be the $\Q(q)$-subspaces of
$\mathbb W\wotimes\mathbb V$ spanned by the sets $\{L'_f\vert
f\in\Z^{1+1}\}$ and $\{T'_f\vert f\in\Z^{1+1}\}$, respectively.
\end{definition}
It can be checked similarly as before that $\mathbb L'$ and $\mathbb
U'$ are also weakly based $U_q(\gl_\infty)$-modules. For $f\in
\Z^{1+1}$, define $f\cdot\tau\in \Z^{1+1}$ by $(f\cdot\tau)(1)=f(2)$
and $(f\cdot\tau)(2)=f(1)$. Define a $\Q(q)$-vector space isomorphism
$\mc R_L:\mathbb L\rightarrow \mathbb L'$ by
\begin{align*}
\mc{R}_L(L_f)=
\begin{cases}
L'_{f^{\uparrow}}, &\text{ if }f(1)=f(2),
 \\
L'_{f\cdot\tau},&\text{ if }f(1)\not=f(2).
\end{cases}
\end{align*}
Similarly, define a $\Q(q)$-vector space isomorphism $\mc
R_U:\mathbb U\rightarrow \mathbb U'$ by
\begin{align*}
\mc{R}_U(T_f)=
\begin{cases}
T'_{f^{\downarrow}}, &\text{ if }f(1)=f(2),
 \\
T'_{f\cdot\tau},&\text{ if }f(1)\not=f(2).
\end{cases}
\end{align*}

\begin{lem}  \label{lem:RTL}
The maps $\mc R_L:\mathbb L\rightarrow \mathbb L'$ and $\mc
R_U:\mathbb U\rightarrow \mathbb U'$ are isomorphisms of weakly
based $U_q(\gl_\infty)$-modules.
\end{lem}

\begin{proof}
This can be checked by considering the cases one by one. The
calculations are fairly easy, and as an illustration, we shall do
this only for two non-trivial cases for $\mc R_L$ below. Let us
identify $f\in\Z^{1+1}$ with the tuple $(f(1),f(2))$.

In the first case, for any $a\in\Z$ we have
\begin{align*}
\mc R_L\left(E_a L_{a,a}\right)=\mc R_L\left(v_a\otimes
w_{a+1}\right)=w_{a+1}\otimes v_a = E_a L'_{a+1,a+1} = E_a \mc
R_L\left(L_{a,a}\right).
\end{align*}

In the second case, for any $a\in\Z$,
\begin{align*}
\mc R_L\left(F_a L_{a,a+1}\right)
&=\mc R_L\left(\Delta(F_a) (v_a\otimes w_{a+1})\right)
 = \mc R_L\left(v_{a+1}\otimes w_{a+1} + q v_a\otimes w_a\right)
 \\
&= \mc R_L\left(L_{a+1,a+1}+q^{-1} L_{a,a}+ qL_{a,a}+L_{a-1,a-1}\right)
 \\
&= L'_{a+2,a+2}+q^{-1} L'_{a+1,a+1}+ qL'_{a+1,a+1}+L'_{a,a}.
\end{align*}
On the other hand, we compute
\begin{align*}
F_a \mc R_L\left(L_{a,a+1}\right)
&=F_a \mc R_L\left(v_a\otimes w_{a+1}\right)
 = \Delta(F_a) (w_{a+1}\otimes v_a)
= w_a\otimes v_a + q w_{a+1}\otimes v_{a+1}
 \\
&= L'_{a,a}+q^{-1} L'_{a+1,a+1}+q L'_{a+1,a+1}+ L'_{a+2,a+2}.
\end{align*}
So $\mc R_L\left(F_a L_{a,a+1}\right)= F_a \mc
R_L\left(L_{a,a+1}\right)$.
\end{proof}

\begin{rem}\label{rem:VWWV}
The definitions of the maps $\mc R_L$ and $\mc R_U$ are motivated by
and compatible in a suitable sense with the map from $\mathbb
V\otimes\mathbb W$ to $\mathbb W\wotimes\mathbb V$ given by the $\mc
R$-matrix. Explicitly, the map on the standard monomial basis is
given as follows:
$$
M_f \mapsto \left \{
 \begin{array}{ll}
  q^{-1} M_f' - (q-q^{-1}) \sum_{k \geq 1}
 (-q)^{-k} M'_{f^{\uparrow k}},  & \text{ if } f(1) =f(2),
 \\
  M_f', & \text{ if } f(1) \neq f(2).
 \end{array}
 \right.
$$
However, such an $\mc R$-map does not extend to the $B$-completion
$\mathbb V\wotimes\mathbb W$. Indeed one sees that the dual canonical bases of $\mathbb V\widehat{\otimes} \mathbb W$ and $\mathbb W\widehat{\otimes} \mathbb V$ are not compatible under such a map.  This is why we have to work with $\mc
R_L$ and $\mc R_U$ instead, and introduce suitable completions using
$\mathbb L$ and $\mathbb U$ in the next subsections.
\end{rem}

\subsection{Canonical basis revisited}
\label{sec:CB2}

In this subsection, we shall give another description of canonical
and dual canonical bases in the Fock spaces.

Fix a $0^m1^n$-sequence of the form ${\bf b}=({\bf b}^1,{0},{1},{\bf
b}^2)$, where ${\bf b}^1$ and ${\bf b}^2$ are ${0}^{m_1}{1}^{n_1}$-
and ${0}^{m_2}{1}^{n_2}$-sequences satisfying $m=m_1+m_2+1$ and
$n=n_1+n_2+1$. Set
\begin{equation}  \label{kappa}
\ka=m_1+n_1+1.
\end{equation}
Recalling the $U_q(\gl_\infty)$-module $\mathbb L$ from
Definition~\ref{def:UL}, we form a $\Q(q)$-vector space
\begin{align}\label{def:T:N}
\TL:=\mathbb T^{{\bf b}^1}\otimes \mathbb L\otimes \mathbb T^{{\bf
b}^2},
\end{align}
which admits the following basis
\begin{align*}
N_f :=\texttt{v}^{b_1}_{f(1)}\otimes\cdots\otimes
\texttt{v}^{b_{\ka-1}}_{f(\ka-1)}\otimes L_{f_{\ka,\ka+1}}\otimes
\texttt{v}^{b_{\ka+2}}_{f(\ka+2)}\otimes\cdots\otimes
\texttt{v}^{b_{m+n}}_{f(m+n)},\quad \text{ for } f\in\Z^{m+n}.
\end{align*}
Extending the definition \eqref{eq:updown} for $m=n=1$, to $f\in
\Z^{m+n}$ with $f(\ka)=f(\ka+1)$, we define $f^{\downarrow},
f^\uparrow\in\Z^{m+n}$ such that
\begin{align}  \label{eq:updown2}
\begin{split}
f^{\uparrow}(i)=f^{\downarrow}(i)=f(i), & \; \;\text{ for }
i\not=\ka,\ka+1,
  \\
f^{\uparrow}(\ka)=f^{\uparrow}(\ka+1)=f(\ka)+1, &\;\;\text{ and }
 f^{\downarrow}(\ka)=f^{\downarrow}(\ka+1)=f(\ka)-1.
\end{split}
\end{align}
It follows by definition and \eqref{eq:inv} that
\begin{align} \label{eq:MNf}
M_f^{\bf b} =
 \begin{cases}
 N_f +q^{-1}N_{f^\downarrow}, & \text{ if }f(\ka)=f(\ka+1),
  \\
 N_f, & \text{ if }f(\ka)\neq f(\ka+1).
 \end{cases}
\end{align}
As in Lemma~ \ref{lem:formula:can:VW}, we can write $N_f\in
\wt{\mathbb T}^{\bf b}$ as
\begin{align}\label{aux:N:in:M}
N_f=
 \begin{cases}
 M^{\bf b}_f+\sum_{k=1}^\infty (-q)^{-k} M^{\bf b}_{f^{\downarrow
k}},  & \text{ if }f(\ka)=f(\ka+1),
  \\
 M^{\bf b}_f, & \text{ if }f(\ka)\neq f(\ka+1).
 \end{cases}
\end{align}
Since $\mathbb L$ is a $U_q(\gl_\infty)$-module, we can write
$\ov{N}_f=\sum_{g} a_{gf} N_g$ as an element in the $A$-completion
$\TLwt$, which is defined similarly as in Definition \ref{def:completion:T}. Recalling that $\ov{M^{\bf b}_g}$ is of the form
$\ov{M^{\bf b}_g}=\sum_{h\preceq_{\bf b}g}r_{hg}M^{\bf b}_h$, for
$r_{hg}\in \Z[q,q^{-1}]$, we conclude from \eqref{eq:MNf} and \eqref{aux:N:in:M} that
\begin{align}\label{aux:N:Bruhat}
\ov{N}_f= N_f +\sum_{g\prec_{\bf b}f} a_{gf} N_g,\quad \text{ for }
a_{gf}\in\Z[q,q^{-1}].
\end{align}
We form the $B$-completion $\TLhat$ of $\TL$, which is the
$\Q(q)$-vector space spanned by elements of the form
$N_f+\sum_{g\prec_{\bf b} f}d_{gf}N_g$, following Definition
\ref{def:completion:T}. Note that $\ov{N}_f\in \TLhat$ by
\eqref{aux:N:Bruhat}.

The following lemma follows directly from \eqref{aux:N:Bruhat} and
Lemma~ \ref{lem:lusztig}.

\begin{lem}  \label{lem:DCBTLhat}
There exists a unique bar-invariant topological basis $\{\texttt{L}_f\vert f\in\Z^{m+n}\}$ in $\TLhat$
such that
\begin{align}\label{dual:canonical:in:N}
\texttt{L}_f=\sum_{g\preceq_{\bf b}f}\check{\ell}_{gf}(q) N_g,
\end{align}
where $\check{\ell}_{ff}=1$ and $\check{\ell}_{gf}(q) \in
q^{-1}\Z[q^{-1}]$, for $g\prec_{\bf b}f$.
\end{lem}

We call $\{\texttt{L}_f\vert f\in\Z^{m+n}\}$ the {\em dual
canonical basis} of $\TLhat$.

Recalling now  $\mathbb U$ from Definition~\ref{def:UL}, we form the
$\Q(q)$-vector space
\begin{align}\label{def:T:U}
\TU:=\mathbb T^{{\bf b}^1}\otimes \mathbb U\otimes \mathbb T^{{\bf
b}^2},
\end{align}
which has the following basis
\begin{align*}
U_f:=\texttt{v}^{b_1}_{f(1)}\otimes\cdots\otimes
\texttt{v}^{b_{\ka-1}}_{f(\ka-1)}\otimes T_{f_{\ka,\ka+1}}\otimes
\texttt{v}^{b_{\ka+2}}_{f(\ka+2)}\otimes\cdots\otimes
\texttt{v}^{b_{m+n}}_{f(m+n)},\quad \text{ for }f\in\Z^{m+n}.
\end{align*}
It follows by definition and Lemma~\ref{lem:formula:can:VW} that
\begin{align} \label{eq:UMf}
U_f =
 \begin{cases}
 M_f^{\bf b} +q M_{f^\downarrow}^{\bf b}, & \text{ if }f(\ka)=f(\ka+1),
  \\
 M_f^{\bf b}, & \text{ if }f(\ka)\neq f(\ka+1).
 \end{cases}
\end{align}

By a similar argument as for \eqref{aux:N:Bruhat} we have
\begin{align}\label{aux:U:Bruhat}
\ov{U}_f=\sum_{g\preceq_{\bf b}f} d_{gf} U_g,\quad
d_{gf}\in\Z[q,q^{-1}].
\end{align}
Let $\TUhat$ denote the $B$-completion of $\TU$, following
Definition \ref{def:completion:T}. Then $\ov{U}_f \in \TUhat$ by
\eqref{aux:U:Bruhat}.

The following lemma is immediate from \eqref{aux:U:Bruhat} and
Lemma~ \ref{lem:lusztig}.

\begin{lem} \label{lem:CBTUhat}
There exists a unique bar-invariant topological basis $\{\texttt{T}_f\vert
f\in\Z^{m+n}\}$ in $\TUhat$ such that
\begin{align}\label{dual:canonical:in:U}
\texttt{T}_f=\sum_{g\preceq_{\bf b}f}\check{t}_{gf}(q) U_g,
\end{align}
where $\check{t}_{ff}=1$ and $\check{t}_{gf}(q) \in q\Z[q]$, for
$g\prec_{\bf b}f$.
\end{lem}
We call $\{\texttt{T}_f\vert f\in\Z^{m+n}\}$ the {\em
canonical basis} of $\TUhat$.

Recall that $\{L^{\bf b}_f | f\in\Z^{m+n}\}$ denotes the dual
canonical basis and $\{T^{\bf b}_f | f\in\Z^{m+n}\}$ denotes the
canonical basis in $\widehat{\mathbb T}^{\bf b}$. Note that
$\TLhat\subseteq\widehat{\mathbb T}^{\bf b}$ by definition and
\eqref{aux:N:in:M}. Similarly, we have
$\TUhat\subseteq\widehat{\mathbb T}^{\bf b}$. Hence we may naturally
regard $\texttt{L}_f\in\TLhat$ as an element in $\widehat{\mathbb
T}^{\bf b}$ and $\texttt{T}_f \in\TUhat$ as an element in
$\widehat{\mathbb T}^{\bf b}$.

\begin{prop}\label{LTsame}
We have the following identification of canonical and dual canonical
bases:
$\texttt{L}_f=L^{\bf b}_f\in \TLhat\subseteq\widehat{\mathbb T}^{\bf
b}$ and $ \texttt{T}_f=T^{\bf b}_f\in
\TUhat\subseteq\widehat{\mathbb T}^{\bf b}$, for $f\in\Z^{m+n}$.
\end{prop}

\begin{proof}
The proofs of the two identities are analogous, and we will prove
the first one.

 Recall that we have
$ L^{\bf b}_f=M^{\bf b}_f +\sum_{g\prec_{\bf b}f}\ell^{\bf
b}_{gf}(q) M^{\bf b}_g,  \text{where } \ell^{\bf b}_{gf}(q) \in
q^{-1}\Z[q^{-1}]. $ Using \eqref{aux:N:in:M} and
\eqref{dual:canonical:in:N} we obtain an expression for the
bar-invariant element $\texttt{L}_f$ that equals $M_f$ plus a
$q^{-1}\Z[q^{-1}]$-linear combination of $M_g$ with $g\prec_{\bf
b}f$. By the uniqueness of the dual canonical basis in
$\widehat{\mathbb T}^{\bf b}$ (see Lemma~ \ref{lem:lusztig}), this
must be equal to $L^{\bf b}_f$. Hence, $\texttt{L}_f=L^{\bf b}_f$.
\end{proof}
We call  $\{N_f | f\in\Z^{m+n}\}$ and $\{U_f | f\in\Z^{m+n}\}$ the
{\em parabolic monomial bases} for $\widehat{\mathbb T}^{\bf b}$.

\begin{rem}
It is natural to conjecture that the polynomials
$\check{\ell}_{gf}(q)$ defined in Lemma~\ref{lem:DCBTLhat} and the
polynomials $\check{t}_{gf}(q)$ in Lemma~\ref{lem:CBTUhat} satisfy
the positivity property: $\check{t}_{gf}(q)\in\N[q]$ and
$\check{\ell}_{gf}(-q^{-1})\in\N[q]$. In light of \eqref{aux:N:in:M}
and \eqref{eq:UMf}, the positivity of $\check{t}_{gf}(q)$ and
$\check{\ell}_{gf}(-q^{-1})$ implies
Conjecture~\ref{con:positive}.
\end{rem}

Two $0^m1^n$-sequences are said to be {\em adjacent} to each other
if they are identical except for a switch of a neighboring pair
$\{0,1\}$, that is, the $0^m1^n$-sequences ${\bf b}':=({\bf
b}^1,{1},{0},{\bf b}^2)$ and ${\bf b}=({\bf b}^1,{0},{1},{\bf b}^2)$
are adjacent. The constructions below for $\bf b'$ are analogous to
the above constructions for $\bf b$, so we will merely set up the
necessary notations for later use. Recall the spaces $\mathbb L'$
and $\mathbb U'$ from Definition~\ref{def:UL'}. We form the
$\Q(q)$-vector space
\begin{align}\label{def:T:N'}
\TL' :=\mathbb T^{{\bf b}^1}\otimes \mathbb L'\otimes \mathbb
T^{{\bf b}^2},
\end{align}
which admits a {\em parabolic monomial basis}
\begin{align*}
N'_f:=\texttt{v}^{b_1}_{f(1)}\otimes\cdots\otimes
\texttt{v}^{b_{\ka-1}}_{f(\ka-1)}\otimes L'_{f_{\ka,\ka+1}}\otimes
\texttt{v}^{b_{\ka+2}}_{f(\ka+2)}\otimes\cdots\otimes
\texttt{v}^{b_{m+n}}_{f(m+n)},\quad \text{ for } f\in\Z^{m+n}.
\end{align*}
Here we recall $\ka$ from \eqref{kappa}. We proceed as before to
define the $B$-completion $\TLhat'$ of $\TL'$, and then obtain
{\em dual canonical basis elements} in $\TLhat'$ denoted accordingly by
$\{\texttt{L}'_f\}$. Note that $\TLhat'\subseteq\widehat{\mathbb
T}^{{\bf b}'}$.

In addition, we form the $\Q(q)$-vector space
\begin{align}\label{def:T:U'}
\TU' :=\mathbb T^{{\bf b}^1}\otimes \mathbb U'\otimes \mathbb
T^{{\bf b}^2}.
\end{align}
which has a {\em parabolic monomial basis}
\begin{align*}
U_f' :=\texttt{v}^{b_1}_{f(1)}\otimes\cdots\otimes
\texttt{v}^{b_{\ka-1}}_{f(\ka-1)}\otimes T_{f_{\ka,\ka+1}}'\otimes
\texttt{v}^{b_{\ka+2}}_{f(\ka+2)}\otimes\cdots\otimes
\texttt{v}^{b_{m+n}}_{f(m+n)},\quad \text{ for } f\in\Z^{m+n}.
\end{align*}
We proceed as before to define the $B$-completion $\widehat{\mathbb
T}_U'$ of $\TU'$, and then obtain the {\em canonical basis} in
$\TUhat'$ denoted accordingly by $\{\texttt{T}'_f\}$. Note that
$\TUhat'\subseteq\widehat{\mathbb T}^{{\bf b}'}$.

Hence we may naturally regard $\texttt{L}_f'\in\TLhat'$ as an
element in $\widehat{\mathbb T}^{\bf b'}$ and $\texttt{T}_f'
\in\TUhat'$ as an element in $\widehat{\mathbb T}^{\bf b'}$. The
following analogue of Propositions~ \ref{LTsame} can be proved in
exactly the same way.

\begin{prop}\label{L'T'same}
We have the following identification of canonical and dual canonical
bases:
$\texttt{L}_f'=L^{\bf b'}_f \in  \TLhat'\subseteq\widehat{\mathbb
T}^{\bf b'}$ and $ \texttt{T}_f'=T^{\bf b'}_f\in
\TUhat'\subseteq\widehat{\mathbb T}^{\bf b'}$, for $f\in\Z^{m+n}$.
%\begin{enumerate}
%\item
%Regarding $\texttt{L}'_f \in\TL'$ as an element in $\widehat{\mathbb
%T}^{{\bf b}'}$ for $f\in\Z^{m+n}$, we have $\texttt{L}'_f=L^{\bf b'}_f. $
%
%\item Regarding $\texttt{T}'_f \in \TU'$ as an element in
%$\widehat{\mathbb T}^{{\bf b}'}$ for $f\in\Z^{m+n}$, we have
%$\texttt{T}'_f=T^{\bf b'}_f. $
%\end{enumerate}
\end{prop}

\subsection{Adjacent canonical bases}
\label{sec:adjCB}

We continue to use the notations of the adjacent sequences $\bf b$ and $\bf
b'$ as well as $\ka$  from \S\ref{sec:CB2}.

For $f\in\Z^{m+n}$, define $f\cdot\tau\in\Z^{m+n}$ by letting
\begin{align*}
(f\cdot\tau)(i)=f(i), & \;\; \text{ for } i\not=\ka,\ka+1,
  \\
(f\cdot\tau)(\ka)=f(\ka+1),\;\; &\text{ and }
(f\cdot\tau)(\ka+1)=f(\ka).
\end{align*}
Recall $f^\uparrow, f^\downarrow$ from \eqref{eq:updown2}. The
following notations will be convenient in the sequel: for
$f\in\Z^{m+n}$ set
\begin{align} \label{eq:fL}
f^{\mathbb L}=
\begin{cases}
f\cdot\tau,&\text{ if }f(\ka)\not=f(\ka+1),\\
f^\uparrow,&\text{ if }f(\ka)=f(\ka+1),
\end{cases}
\end{align}
\begin{align} \label{eq:fU}
f^{\mathbb U}=\begin{cases}
f\cdot\tau,&\text{ if }f(\ka)\not=f(\ka+1),\\
f^\downarrow,&\text{ if }f(\ka)=f(\ka+1).
\end{cases}
\end{align}

Recall from Lemma~\ref{lem:RTL} the $U_q(\gl_\infty)$-module
isomorphism ${\mc R}_L$ of the weakly based modules $\mathbb L$ and
$\mathbb L'$. Recall $\TL =\mathbb T^{{\bf b}_1} \otimes \mathbb L
\otimes \mathbb T^{{\bf b}_2}$ and $\mathbb T'_{\mathbb L} =\mathbb T^{{\bf b}_1} \otimes \mathbb L'
\otimes \mathbb T^{{\bf b}_2}$.
%Let $1_{{\bf b}_1}$ and $1_{{\bf b}_2}$ be the identity maps on
%$\mathbb T^{{\bf b}_1}$ and $\mathbb T^{{\bf b}_2}$, respectively.
Then
\begin{align*}
\mc R \stackrel{\text{def}}{=} 1_{{\bf b}_1}\otimes\mc{R}_L\otimes
1_{{\bf b}_2} :&\  \TL \longrightarrow \TL'
  \\
\mc R(N_f)=N'_{f^{\mathbb L}}, & \quad \forall f,
\end{align*}
is an isomorphism of $U_q(\gl_\infty)$-modules.

Define the truncated subspaces $\left[\TL\right]_{\le|k|}$ and
$\left[\TL'\right]_{\le|k|}$ for $k\in \N$ as in
\S\ref{sec:completions}, and then form the topological
$A$-completions $\TLwt$ and $\TLwt'$ as in Definition~
\ref{def:completion:T} with corresponding projection maps
$\pi_{\mathbb L,k}: \TL \rightarrow \left[\TL\right]_{\le|k|}$ and
$\pi_{\mathbb L,k}': \TL' \rightarrow \left[\TL'\right]_{\le|k|}$ as in \eqref{def:pi:lek}. We have $\mc
R(\left[\TL\right]_{\le|k|})\subseteq \left[\TL'\right]_{\le|k+1|}$,
and hence $\mc R$ extends to a linear isomorphism $\mc
R:\TLwt\rightarrow \TLwt'$, which  is actually a homeomorphism of
topological vector spaces.
%
%Indeed, if $x\in\ker\pi_{\mathbb
%L,k+1}'$, then $\pi_{\mathbb L,k}\mc{R}^{-1}(x)=0$, so the induced
%map on completions is actually a homeomorphism.

\begin{definition}
\begin{enumerate}
\item
The {\em partial ordering $\ \preceq_{\bf b,b'}$} on $\Z^{m+n}$ is
defined as follows: $g\preceq_{\bf b,b'}f$ if and only if
$g\preceq_{\bf b}f$ and $g^{\mathbb L}\preceq_{\bf b'}f^{\mathbb
L}$, for $f,g \in \Z^{m+n}$.

\item
The {\em partial ordering $\ \preceq_{\bf b,b'}^*$} on $\Z^{m+n}$ is
defined as follows: $g\preceq_{\bf b,b'}^*f$ if and only if
$g\preceq_{\bf b'}f$ and $g^{\mathbb U}\preceq_{\bf b}f^{\mathbb
U}$, for $f,g \in \Z^{m+n}$.

\item
The {\em $C$-completion} of $\TL$, denoted by $\TLC$, is the
$\Q(q)$-subspace of $\TLwt$ spanned by vectors of the form
$N_f+\sum_{g\prec_{\bf b,b}f}r_{g}N_g$, for $r_g\in\Q(q)$.

\item
The {\em $C$-completion} of $\TL'$, denoted by $\TLC'$, is the
$\Q(q)$-subspace of $\TLwt'$ spanned by vectors of the form
$N_{f}'+\sum_{g\prec_{\bf b,b'}^*f}r_{g}N_{g}'$, for $r_g\in\Q(q)$.
\end{enumerate}
\end{definition}
In other words, $\TLC$ is simply the $B$-completion of $\TL$ with
respect to the refined partial ordering $\preceq_{\bf b,b'}$, while
$\TLC'$ is the $B$-completion of $\TL'$ with respect to
$\preceq_{\bf b,b'}^*$. By definition, the $A$-, $B$- and
$C$-completions of $\TL$ and $\TL'$ in \eqref{def:T:N} and
\eqref{def:T:N'} are related as follows:
$$
\TLC \subseteq \TLhat \subseteq \TLwt,\qquad
\TLC' \subseteq \TLhat' \subseteq \TLwt'.
$$
Since $(f^{\mathbb L})^{\mathbb U} =(f^{\mathbb U})^{\mathbb L}=f$,
we have
\begin{equation}  \label{2order}
g\prec_{\bf b,b'} f \ \text{ if and only if }\ g^{\mathbb L}
\prec_{\bf b,b'}^*  f^{\mathbb L}.
\end{equation}

\begin{rem}
The reason for defining the three different $A$-, $B$-, $C$-completions is roughly as follows.
The $A$-completions are the simplest completions on which all constructions are based.
While they are suitable for our purpose of comparing with finite-dimensional Fock spaces, they are not adequate for other purposes.
For example, they are not suitable for defining the bar map and hence for the formulation of the BKL conjectures.
For this purpose the $B$-completions (where $B$ stands for Bruhat or Brundan)
are the most natural candidates. However, as shown in Remark \ref{rem:VWWV}, the $B$-completions themselves
are not adequate for the purpose of comparing the dual canonical bases of two adjacent Fock spaces. For this, we introduce the $C$-completions.
\end{rem}

\begin{thm}\label{S:L:to:L}
Let ${\bf b}$ and ${\bf b}'$ be adjacent $0^m1^n$-sequences. Then
\begin{enumerate}
\item
the restriction of  $\mc R:\TLwt\stackrel{\sim}{\rightarrow} \TLwt'$
gives a $\Q(q)$-linear isomorphism $\mc R:\TLC\rightarrow \TLC'$;

\item
$\TLC$ and $\TLC'$ are bar-invariant subspaces of $\ \TLhat$ and
$\TLhat'$, respectively. Moreover, the dual canonical bases of $\
\TLhat$ and $\TLhat'$ lie in $\TLC$ and $\TLC'$, respectively;

\item
$\mc R(N_f)=N'_{f^{\mathbb L}},$
$\mc R(\texttt{L}_f)=\texttt{L}'_{f^{\mathbb L}}$, and
$\mc R({L}^{\bf b}_f)={L}^{\bf b'}_{f^{\mathbb L}}$, for all $f \in
\Z^{m+n}.$
\end{enumerate}
\end{thm}

\begin{proof}
By definition, we have
\begin{equation}  \label{eq:RNf}
\mc R(N_f)=N'_{f^{\mathbb L}}, \quad \forall f.
\end{equation}
Part (1)  follows from this, \eqref{2order}, and the definition of
$C$-completions $\TLC$, $\TLC'$.

We now first work with the $A$-completions and with $\mc
R:\TLwt\rightarrow \TLwt'$. It follows from \eqref{eq:RNf} and the
definition of the bar map on tensor modules (cf. \eqref{def:bar:map}
and \eqref{eq:barMf}) that
\begin{equation} \label{eq:RbarN}
\mc R(\ov{N}_f)=\ov{\mc R(N_f)}= \ov{N'_{f^{\mathbb L}}}.
\end{equation}

%Write $\ov{N}_f\in\TLhat\subseteq\TLwt$ with $f\in\Z^{m+n}$ as
%\begin{align}
%\ov{N}_f=N_f+\sum_{g\prec_{\bf b}f}a_{gf} N_g.
%\end{align}

Hence from \eqref{aux:N:Bruhat}, \eqref{eq:RNf} and \eqref{eq:RbarN}
we obtain that
\begin{equation}  \label{eq:barN'}
\ov{N'_{f^{\mathbb L}}}=N'_{f^{\mathbb L}}+\sum_{g\prec_{\bf b}f}
a_{gf}N'_{g^{\mathbb L}}.
\end{equation}
On the other hand, $\ov{N'_{f^{\mathbb
L}}}\in\TLhat'\subseteq\TLwt'$ can be written in the form
\begin{equation}  \label{eq:barN'2}
\ov{N'_{f^{\mathbb L}}}=N'_{f^{\mathbb L}}+\sum_{g^{\mathbb
L}\prec_{\bf b'}f^{\mathbb L}}a_{gf}' N'_{g^{\mathbb L}}.
\end{equation}
The comparison between \eqref{eq:barN'} and \eqref{eq:barN'2}
implies that $a_{gf}' =a_{gf}$ if $g\prec_{\bf b, \bf b'}f$, and
\begin{equation}  \label{eq:barN'3}
\ov{N'_{f^{\mathbb L}}}=N'_{f^{\mathbb L}}+\sum_{g\prec_{\bf b, \bf
b'}f} a_{gf}N'_{g^{\mathbb L}}=N'_{f^{\mathbb L}}+\sum_{g^{\mathbb
L} \prec_{\bf b,b'}^*  f^{\mathbb L}} a_{gf}N'_{g^{\mathbb L}}.
\end{equation}
Using the inverse $\mc R^{-1}$ instead and arguing similarly as
above, we then obtain the following counterpart of
\eqref{eq:barN'3}:
\begin{equation}  \label{eq:barN2}
\ov{N_{f}}=N_{f}+\sum_{g\prec_{\bf b, \bf b'}f} a_{gf} N_{g}.
\end{equation}
Hence, $\ov{N_{f}}$ actually lies in the $C$-completion $\TLC$ (and
then in $B$-completion $\TLhat$) and also $\ov{N'_{f^{\mathbb L}}}
\in \TLC' \subset \TLhat'$. The first half of (2) now follows from
\eqref{eq:barN'3} and \eqref{eq:barN2}.

Now we can work within the $B$-completions. By
Lemma~\ref{lem:DCBTLhat}, \eqref{eq:barN2}, and the uniqueness part
of Lemma ~\ref{lem:lusztig} applied to the partial ordering
$\preceq_{\bf b, \bf b'}$, the dual canonical basis element
$\texttt{L}_f$ in $\TLhat$ satisfies the refined partial ordering as
follows:
\begin{align}\label{DCB2}
\texttt{L}_f=N_f +\sum_{g\prec_{\bf b, \bf b'}f}
\check{\ell}_{gf}(q) N_g,
\end{align}
where $\check{\ell}_{gf}(q) \in q^{-1}\Z[q^{-1}]$, for ${g\prec_{\bf
b, \bf b'}f}$. This strengthens \eqref{dual:canonical:in:N}. Hence
$\texttt{L}_f \in \TLC$. Similarly, by Lemma~\ref{lem:CBTUhat} and
\eqref{eq:barN'3} we have
\begin{align}\label{DCB2'}
\texttt{L}_f'=N_f' +\sum_{g\prec_{\bf b, \bf b'}^*f}
\check{\ell}_{gf}'(q) N_g',
\end{align}
where $\check{\ell}_{gf}'(q) \in q^{-1}\Z[q^{-1}]$, for
${g\prec_{\bf b, \bf b'}^*f}$, and hence $\texttt{L}_{f}' \in
\TLC'$. This proves the second part of (2).

Thanks to \eqref{eq:RNf} and \eqref{DCB2}, $\mc R (\texttt{L}_f)$
satisfies the same characterization as the dual canonical basis
element $\texttt{L}'_{f^{\mathbb L}}$ (similar to
Lemma~\ref{lem:DCBTLhat}). Hence $\mc
R(\texttt{L}_f)=\texttt{L}'_{f^{\mathbb L}}$ by the uniqueness of
dual canonical basis.

Now $\mc R({L}^{\bf b}_f)={L}^{\bf b'}_{f^{\mathbb L}}$ follows from
the identifications in Propositions~\ref{LTsame} and \ref{L'T'same}.
\end{proof}

\begin{cor}  \label{cor:llsame}
In the notations of \eqref{DCB2} and \eqref{DCB2'}, we have
$\check{\ell}_{gf}(q) =\check{\ell}_{g^{\mathbb L}f^{\mathbb
L}}'(q).$
\end{cor}

Similarly, we recall $\mc R_U: \mathbb U \rightarrow \mathbb U'$
from Lemma~\ref{lem:RTL}, $\TU =\mathbb T^{{\bf b}_1} \otimes
\mathbb U \otimes \mathbb T^{{\bf b}_2}$ from \eqref{def:T:U}, and
$\TU'=\mathbb T^{{\bf b}_1} \otimes
\mathbb U' \otimes \mathbb T^{{\bf b}_2}$ from \eqref{def:T:U'}. Then we have an isomorphism of
$U_q(\gl_\infty)$-modules
\begin{align*}
\mc R^u \stackrel{\text{def}}{=} 1_{{\bf b}_1}\otimes\mc{R}_U\otimes
1_{{\bf b}_2} :&\  \TU \longrightarrow \TU'
  \\
\mc R^u(U_f)=U'_{f^{\mathbb U}}, & \quad \forall f,
\end{align*}
The isomorphism $\mc R^{u}$ extends to an isomorphism on the
$A$-completions $\mc R^{u} : \TUwt \rightarrow \TUwt'$ as before.
The {\em $C$-completions} $\TUC$ and $\TUC'$  of $\TU$ and $\TU'$
are defined as the $B$-completions of $\TU$ and $\TU'$
respectively with respect to some suitably refined partial orderings
(given by the conditions in the sums \eqref{CB2} and \eqref{CB2'}
below).

\begin{thm}\label{S:T:to:T}
Let ${\bf b}$ and ${\bf b}'$ be adjacent $0^m1^n$-sequences. Then
\begin{enumerate}
\item
the restriction of  $\mc R^{u}:\TUwt\stackrel{\sim}{\rightarrow}
\TUwt'$ gives an isomorphism $\mc R^{u}:\TUC\rightarrow \TUC'$;

\item
the canonical bases of $\TUhat$ and $\TUhat'$ lie in $\TUC$ and
$\TUC'$, respectively;

\item
$ \mc R^{u}(U_f)=U'_{f^{\mathbb U}},$
$\mc R^{u} (\texttt{T}_f)=\texttt{T}'_{f^{\mathbb U}}$, and
$\mc R^{u} ({T}^{\bf b}_f)= {T}^{\bf b'}_{f^{\mathbb U}}$, for all
$f \in \Z^{m+n}.$
\end{enumerate}
\end{thm}
Theorem~\ref{S:T:to:T} is the canonical basis analogue of Theorem
\ref{S:L:to:L}, where (3) uses the identification provided by
Propositions~\ref{LTsame} and \ref{L'T'same}. While we will skip the
entirely analogous proof, we note that we obtain the following
analogues of \eqref{DCB2} and \eqref{DCB2'} in the process of proof:
\begin{align}\label{CB2}
\texttt{T}_f=U_f +\sum_{g\prec_{\bf b}f, g^{\mathbb U}\prec_{\bf
b'}f^{\mathbb U}} \check{t}_{gf}(q) U_g, \quad \text{ where }\
\check{t}_{gf}(q) \in q\Z[q].
\end{align}
\begin{align}\label{CB2'}
\texttt{T}_f'=U_f' +\sum_{g\prec_{\bf b'}f, g^{\mathbb L}\prec_{\bf
b}f^{\mathbb L}} \check{t}_{gf}'(q) U_g', \quad \text{ where }\
\check{t}_{gf}'(q) \in q\Z[q].
\end{align}
We also have the following corollary to Theorem~\ref{S:T:to:T}.

\begin{cor}  \label{cor:samett}
In the notations of \eqref{CB2} and \eqref{CB2'}, we have
$\check{t}_{gf}(q) =\check{t}_{g^{\mathbb U}f^{\mathbb U}}'(q).$
\end{cor}

%Thanks to the identification provided by Propositions~\ref{LTsame}
%and \ref{L'T'same}, Theorems~\ref{S:L:to:L} and \ref{S:T:to:T} allow
%us to compare the canonical and dual canonical bases in the adjacent
%Fock spaces ${\mathbb T}^{\bf b}$ and ${\mathbb T}^{\bf b'}$.

%%%
%%%
%%%
%%%
\part{Representation Theory}

\section{BGG category for basic Lie superalgebras}
 \label{sec:repn:prep}

In this section, we establish some basic properties of the BGG
 category $\mc O$ of $\glmn$-modules. We show that the category $\mc
 O$ is independent of the choice of non-conjugate Borel
 subalgebras. We then make systematic comparisons of the Verma,
 simple and tilting modules with respect to different Borel subalgebras.
 Finally, we introduce certain parabolic Verma modules associated to a pair of
 adjacent Borel subalgebras. All the results in this section remain valid for
arbitrary basic Lie superalgebras.

\subsection{Preliminaries}\label{repn:prep:prelim}

Let $\C^{m|n}$ be the complex superspace of dimension $(m|n)$. The
general linear Lie superalgebra $\gl(m|n)$ is the Lie superalgebra
of linear transformations from $\C^{m|n}$ to itself. Thus, with
respect to a given ordered basis of $\C^{m|n}$, $\gl(m|n)$ may be
realized in terms of $(m+n)\times (m+n)$ matrices over $\C$. Let
$\{e_1,\ldots,e_m\}$ and $\{e_{m+1},\ldots,e_{m+n}\}$ be the
standard bases for the even subspace $\C^{m|0}$ and the odd subspace
$\C^{0|n}$, respectively, so that their union is a homogeneous basis
for $\C^{m|n}$. Then with respect to this ordered basis we let
$e_{ij}$, $1\le i,j\le m+n$, denote the $(i,j)$th elementary matrix.
The Cartan subalgebra of diagonal matrices is denoted by $\h_{m|n}$,
which is spanned by $\{e_{ii}|1\le i\le m+n\}$. We denote by
$\{\ep_i|1\le i\le m+n\}$ the basis in $\h_{m|n}^*$ dual to
$\{e_{ii}|1\le i\le m+n\}$, and the lattice  of {\em integral
weights} for $\glmn$ by
$$
\wtl =\sum_{i=1}^{m+n} \Z \ep_i.
$$

The supertrace form on $\gl(m|n)$ induces a non-degenerate symmetric
bilinear form $(\cdot|\cdot)$ on $\h_{m|n}^*$ determined by
\begin{align*}
(\ep_i|\ep_j)=(-1)^{|i|}\delta_{ij}, \quad \text{ for }1\le i,j\le m+n,
\end{align*}
where we use the notation
$|i|:=\begin{cases} {0}, &\text{ if }1\le i\le m,\\
 {1}, &\text{ if }m+1\le i\le m+n. \end{cases}
$
The subalgebra of upper triangular matrices with respect to this
standard basis is called the {\em standard Borel subalgebra} and
denoted by $\mf{b}_{\text{st}}$.

In this paper we shall need to deal with various Borel subalgebras
of $\gl(m|n)$ that may not be conjugate to $\mf{b}_{\text{st}}$. For
this purpose, let ${\bf b}=(b_1,b_2,\ldots,b_{m+n})$ be a
${0^m1^n}$-sequence. Such a sequence ${\bf b}$ gives rise to a {\em ${\bf
b}$-ordered basis} $\{e^{\bf b}_1, e^{\bf b}_2,\ldots,e^{\bf
b}_{m+n}\}$ for $\C^{m|n}$ by rearranging its standard basis  as
follows: Let $1\le i_1<i_2<\ldots<i_m\le m+n$ be such that
$b_{i_k}={0}$, and $1\le j_1<j_2<\ldots<j_n\le m+n$ be such that
$b_{j_\ell}={1}$. Then
$$
e^{\bf b}_{i_k} =e_{k}\;\; (1\le k\le m), \qquad e^{\bf b}_{j_\ell}
=e_{m+\ell} \;\;(1\le \ell\le n).
$$
For example, for the standard sequence ${\bf
b}_{\text{st}}=(0,\ldots,0,1,\ldots,1)$, the ${\bf
b}_{\text{st}}$-ordered basis is the standard basis, i.e.,
$e^{\bf b}_i=e_i$, for $1\le i\le m+n$. On the other hand, if ${\bf
b}$ consists of $n$ ${1}$'s followed by $m$ ${0}$'s, then the ${\bf
b}$-ordered basis is
$\{e_{m+1},e_{m+2},\ldots,e_{m+n},e_1,\ldots,e_m\}$.

We also realize $\gl(m|n)$ as $(m+n)\times(m+n)$ matrices with
respect to the ${\bf b}$-ordered basis. The $(i,j)$th elementary
matrix here is denoted by $e^{\bf b}_{ij}$. The Borel subalgebra
$\mf{b}$ corresponding to ${\bf b}$ (also denoted by $\mf{b}_{\bf
b}$ if necessary) is the subalgebra generated by $e^{\bf b}_{ij}$
for $1\le i\le j\le m+n$. The algebras $\mf{b}$'s for different
${\bf b}$'s are non-conjugate under the (even) group $G_{ \bar 0}$,
and the corresponding simple systems associated to different ${\bf
b}$'s are representatives among all simple systems for
$(\glmn,\h_{m|n})$ under the conjugation by its Weyl group $W=\mf
S_m\times \mf S_n$.

The Cartan subalgebras of ${\mf b}$ consisting of diagonal matrices
are all equal to $\h_{m|n}$ (independent of ${\bf b}$). Let
$\ep^{\bf b}_i\in\h_{m|n}^*$ be defined by
$$
\langle\ep^{\bf b}_i,e_{jj}^{\bf b}\rangle=\delta_{ij}, \qquad
\text{ for } 1\le i,j \le m+n.
$$
We have
$$
(\ep^{\bf b}_i|\ep^{\bf b}_j)=(-1)^{b_i}\delta_{ij}, \qquad \text{
for } 1\le i,j \le m+n.
$$
The simple system with respect to the Borel subalgebra $\mf{b}$ associated to
${\bf b}$ is
$$
\Pi({\bf b}):=\{\ep^{\bf b}_i-\ep^{\bf b}_{i+1}|1\le i\le m+n-1\},
$$
where the parity of $\ep^{\bf b}_i-\ep^{\bf b}_{i+1}$ is $\bar{0}$,
if $b_i=b_{i+1}$, and $\bar{1}$ otherwise. Let $\Phi^+_{{\bf
b},\bar{0}}$ and $\Phi^+_{{\bf b},\bar{1}}$ be the corresponding
sets of positive even and positive odd roots of $\mf{b}$.

Let $\la\in\h_{m|n}^*$. Fix a $0^m1^n$-sequence $\bf b$ and hence a
Borel subalgebra $\mf b$. Let $\C_\la$ be the one-dimensional $\h_{m|n}$-module
that transforms by $\la$, which is extended to a $\mf{b}$-module by
letting $e_{ij}^{\bf b}$ act trivially, for $i<j$. The {\em $\bf b$-Verma
module} of highest weight $\la$ is defined to be
\begin{align*}
M_{\bf b}(\la):={\rm Ind}_{\mf{b}}^{\gl(m|n)}\C_\la,
\end{align*}
and as usual it has a unique irreducible quotient $\glmn$-module,
denoted by $L_{\bf b}(\la)$. We denote by $\text{ch} M$ the (formal)
character of a $\glmn$-weight module $M$ as usual. We have the
following character formula of the $\bf b$-Verma module:
\begin{align*}
\text{ch}M_{\bf b}(\la)
 &=e^{\la}\frac{\prod_{\gamma\in\Phi^+_{{\bf
b},\bar{1}}}(1+e^{-\gamma})}{\prod_{\beta\in\Phi^+_{{\bf
b},\bar{0}}}(1-e^{-\beta})}.
 \end{align*}

 % {\color{red}  A module M over a Lie superalgebra is always understood in the $\Z_2$-graded sense. Every highest weight $\glmn$-module with respect to any Borel subalgebra $%\mf{b}_{\bf b}$ in non-graded sense has exactly two $\Z_2$-gradationsso that it becomes a $\glmn$-module.}

%%
\subsection{Odd reflection}

We follow the notation in \S\ref{sec:CB2}. Take a $0^m1^n$-sequence
of the form ${\bf b}=({\bf b}^1,{0},{1},{\bf b}^2)$, where ${\bf
b}^1$ and ${\bf b}^2$ are ${0}^{m_1}{1}^{n_1}$- and
${0}^{m_2}{1}^{n_2}$-sequences satisfying $m=m_1+m_2+1$ and
$n=n_1+n_2+1$. Recall from \eqref{kappa} that $\ka=m_1+n_1+1.$ Take
the $0^m1^n$-sequence ${\bf b}'=({\bf b}^1,{1},{0},{\bf b}^2)$, {\em
adjacent} to the sequence ${\bf b}$.

Note the simple system $\Pi(\bf{b'})$ is obtained from $\Pi(\bf{b})$
by an odd reflection with respect to the odd simple root $\alpha=\ep^{\bf
b}_\ka-\ep^{\bf b}_{\ka+1}$. The corresponding positive systems are
related by $\Phi^+_{\bf b'} =\Phi^+_{\bf b} \cup \{-\alpha\}
\backslash \{\alpha\}$.
%The Borel subalgebras $\mf b$ and $\mf b'$ associated to
%the $0^m1^n$-sequence %${\bf b}$ and ${\bf b}'$ differ by

\begin{lem}\label{lem:verma:borels}
Let ${\bf b},{\bf b}'$ be adjacent $0^m1^n$-sequences as above. Let
$\alpha=\ep^{\bf b}_\ka-\ep^{\bf b}_{\ka+1}$. Then
\begin{align*}
{\rm ch}\, M_{\bf b}(\la)={\rm ch}\, M_{\bf b'}(\la-\alpha).
\end{align*}
\end{lem}

\begin{proof}
Let $\Psi:=\Phi^+_{{\bf b},\bar{1}}\setminus\{\alpha\}$. Then
$\Phi^+_{{\bf b'},\bar{1}}=\Psi\cup\{-\alpha\}$. Also $\Phi^+_{{\bf
b},\bar{0}}=\Phi^+_{{\bf b'},\bar{0}}$. Thus, we have
\begin{align*}
{\rm ch}\, M_{\bf b}(\la)
 &=e^{\la}\frac{\prod_{\gamma\in\Phi^+_{{\bf
b},\bar{1}}}(1+e^{-\gamma})}{\prod_{\beta\in\Phi^+_{{\bf
b},\bar{0}}}(1-e^{-\beta})}
=e^{\la}(1+e^{-\alpha})
\frac{\prod_{\gamma\in\Psi}(1+e^{-\gamma})}{\prod_{\beta\in\Phi^+_{{\bf
b},\bar{0}}}(1-e^{-\beta})}
 \\
 &
 %=e^{\la-\alpha}(e^\alpha+1)
 %\frac{\prod_{\gamma\in\Psi}(1+e^{-\gamma})}{\prod_{\beta\in
 %\Phi^+_{{\bf b},\bar{0}}}(1-e^{-\beta})}
= e^{\la-\alpha}\frac{\prod_{\gamma\in\Phi^+_{{\bf
b'},\bar{1}}}(1+e^{-\gamma})}{\prod_{\beta\in\Phi^+_{{\bf
b},\bar{0}}}(1-e^{-\beta})}
 ={\rm ch}\, M_{\bf b'}(\la-\alpha).
\end{align*}
This proves the lemma.
\end{proof}

For $\alpha=\ep^{\bf b}_\ka-\ep^{\bf b}_{\ka+1}$,
we introduce the following notation:
\begin{align} \label{eq:laL}
\la^{\mathbb L} =
\begin{cases}
 \la,         &\text{ if }(\la,\alpha)=0,\\
\la-\alpha,   &\text{ if }(\la,\alpha)\neq 0,
\end{cases}
\qquad \text{for } \la \in \wtl.
\end{align}
The following odd reflection lemma is well known (see e.g.~\cite{PS,
KW2}). %For completeness, we include a short proof.

\begin{lem}\label{lem:irred:borels}
Let ${\bf b},{\bf b}'$ be two adjacent $0^m1^n$-sequences as above
and let $\alpha=\ep^{\bf b}_\ka-\ep^{\bf b}_{\ka+1}$.
Then $L_{\bf b}(\la)=L_{\bf b'}(\la^{\mathbb L})$.
\end{lem}

%\begin{proof}
%Let $\{e_\alpha,e_{-\alpha},h_\alpha\}$ be the Chevalley generators
%associated with $\alpha$. Let $v_\la$ be a nonzero $\mf{b}_{\bf
%b}$-highest weight vector of $L_{\bf b}(\la)$. We consider the two cases separately.

%First suppose that $(\la,\alpha)=0$. Then by irreducibility of $L_{\bf
%b}(\la)$ we have $e_{-\alpha} v_\la=0$. Thus, $v_\la$ is a
%$\mf{b}_{{\bf b}'}$-singular vector as well, and so $L_{\bf b}(\la)=L_{{\bf b}'}(\la)$.

%Now suppose that $(\la,\alpha)\not=0$. We have $e_\alpha e_{-\alpha}
%v_\la=h_\alpha v_\la\not=0$ in $L_{\bf b}(\la)$, and so $e_{-\alpha}
%v_\la\not=0$. Now observe that $e_{-\alpha} (e_{-\alpha}
%v_\la)=e_{-\alpha}^2 v_\la=0$, and also, for any
%$\beta\in\Phi^+_{\bf b}\setminus\{\alpha\}$, if $\beta-\alpha$ is a
%root, then $\beta-\alpha\in\Phi^+_{\bf b}$. It follows that, for any
%root vector $X_\beta$ corresponding to such that a $\beta$, we have
%$X_\beta e_{-\alpha} v_\la=0$. Thus, $e_{-\alpha} v_\la$ is a
%$\mf{b}_{{\bf b}'}$-singular vector and hence $L_{\bf b}(\la)=L_{{\bf b}'}(\la-\alpha)$.
%\end{proof}

%%
\subsection{BGG category}
\label{sec:BGGmn}

For $\mu\in\h_{m|n}^*$ and a $\gl(m|n)$-module $M$ we denote the
$\mu$-weight space of $M$ as usual by $M_\mu=\{x\in M\vert
hx=\mu(h)x,\forall h\in \h_{m|n}\}$.

\begin{definition}
Let ${\bf b}$ be a $0^m1^n$-sequence. The Bernstein-Gelfand-Gelfand
(BGG) category $\Omn_{\bf b}$ is the category of finitely generated
$\h_{m|n}$-semisimple $\gl(m|n)$-modules $M$ such that
\begin{itemize}
\item[(i)]
$M=\bigoplus_{\mu\in \wtl}M_\mu$ and $\dim M_\mu<\infty$;

\item[(ii)]
there exist finitely many weights ${}^1\la,{}^2\la,\ldots,{}^k\la\in X(m|n)$
(depending on $M$) such that if $\mu$ is a weight in $M$, then
$\mu\in{{}^i\la}-\sum_{\alpha\in{\Pi({\bf b}})}\Z_+\alpha$, for
some $i$.
\end{itemize}
The morphisms in $\Omn_{\bf b}$ are all (not necessarily even)
homomorphisms of $\gl(m|n)$-modules.
\end{definition}
In short, $\Omn_{\bf b}$ is the category of finitely generated
integral weight $\gl(m|n)$-modules that are $\mf{b}$-locally finite.
The $\gl(m|n)$-modules $L_{\bf b}(\la)$ and $M_{\bf b}(\la)$,
for $\la\in \wtl$, are objects in the BGG category $\Omn_{\bf b}$.

Let $M\in\Omn_{\bf b}$ so that $M=\bigoplus_{\gamma\in \wtl}M_\gamma$.
For $\vartheta\in\Z_2$ we define
$$
\wtl_\vartheta:= \big\{\gamma\in \wtl \vert
\sum_{i=m+1}^n\langle\gamma,e_{ii}\rangle\equiv\vartheta \big\}.
$$
Introduce the following $\glmn$-module whose $\Z_2$-grading
is specified by the action of $\gl(m|n)$ (cf.~\cite[Section 2.5]{CL})
\begin{align}\label{Z2:gradation}
M' =M'_{\bar 0} \oplus M'_{\bar 1}, \qquad \text{where }
M'_{\vartheta}:=\bigoplus_{\gamma\in \wtl_\vartheta}M_\gamma
\;\;(\vartheta\in \Z_2).
\end{align}
Then $M'\in\Omn_{\bf b}$, and the identity map (which does not necessarily preserves the $\Z_2$-gradation) gives an isomorphism $M\cong M'$ in the category
$\Omn_{\bf b}$. Let us denote by $\Omn_{{\bf b},\bar{0}}$ the full
subcategory of $\Omn_{\bf b}$ consisting of objects with
$\Z_2$-gradation given by \eqref{Z2:gradation}. Then all morphisms
in $\Omn_{{\bf b},\bar{0}}$ are automatically even, and hence
$\Omn_{{\bf b},\bar{0}}$ is an abelian category.  Since the
categories $\Omn_{{\bf b},\bar{0}}$ and $\Omn_{\bf b}$ have
isomorphic skeleton subcategories, $\Omn_{{\bf b},\bar{0}}$ and
$\Omn_{\bf b}$ are equivalent categories. It follows that $\Omn_{\bf
b}$ is an abelian category.

We adopt the following convention. When dealing with the BGG category associated to the standard
$0^m1^n$-sequence $\bf b_{\rm st}$ and the standard Borel subalgebra ${\mf
b}_{\rm st}$, we will drop  the subscript ${\bf b}_{\text{st}}$ and
write the corresponding category, Verma module, and irreducible module as
$\Omn$, $M(\la)$, and $L(\la)$, respectively.

\begin{prop}\label{prop:sameO}
The categories $\Omn_{\bf b}$ are identical, for all
$0^m1^n$-sequences ${\bf b}$.
\end{prop}

\begin{proof}
We shall show that the category $\Omn_{\bf b}$ for a fixed $\bf b$
is identical to $\Omn$.
% (associated to $\bf b_{\rm st}$).

It is clear that any two $0^m1^n$-sequences, say $\bf b$ and $\bf
b_{\rm st}$, can be connected via a sequence of $0^m1^n$-sequences
such that any two neighboring sequences are adjacent. Accordingly, the
Borel subalgebras $\mf b$ and $\mf b_{\rm st}$ can be converted to
one another via a sequence of odd reflections. It follows by this
observation and Lemma~ \ref{lem:irred:borels} that the categories
$\Omn_{\bf b}$ and $\Omn$ have the same collection of simple objects, denoted by
$\text{Irr} \mathcal O$.

Since every Verma module $M(\la) \in\Omn$ has finite length when
regarded as a $\gl(m|n)_{\bar{0}}$-module, it has finite length as a
$\gl(m|n)$-module as well. It follows by this and the character
comparisons in Lemmas~\ref{lem:verma:borels} and
\ref{lem:irred:borels} that every $\bf b'$-Verma module $M_{\bf
b}(\la) \in\Omn_{\bf b}$ has finite length with composition factors
in $\text{Irr} \mathcal O$. So we conclude that both categories
$\Omn$ and $\Omn_{\bf b}$ can be characterized as the category of
integral weight $\h_{m|n}$-semisimple $\gl(m|n)$-modules that have
finite composition series with composition factors in $\text{Irr}
\mathcal O$, and hence $\Omn_{\bf b}=\Omn$.
\end{proof}

\subsection{Weyl vectors}\label{sec:super:Burhat}

The supertrace function
\begin{equation} \label{strace}
\mathrm{Str}=\sum_{i=1}^m\ep_i-\sum_{j=1}^n\ep_{m+j}
\end{equation}
satisfies the fundamental property that
\begin{equation} \label{eq:strace}
(\mathrm{Str}\vert \gamma) =0, \qquad \forall \gamma \in \Phi.
\end{equation}

Let $\mf{b}$ be the Borel subalgebra corresponding to the
$0^m1^n$-sequence $\bf b$. Recall $\Phi^+_{{\bf b},\bar{0}}$ and
$\Phi^+_{{\bf b},\bar{1}}$ denote the sets of positive even
and positive odd roots of $\mf{b}$, respectively. We define the Weyl
vector $\wt{\rho}_{\bf b}$ and its normalized version $\rho_{\bf b}$
by
\begin{align}  \label{eq:2rho}
\begin{split}
\wt{\rho}_{\bf b}
 &:=\hf\sum_{\alpha\in \Phi^+_{{\bf
b},\bar{0}}}\alpha-\hf\sum_{\beta\in\Phi^+_{{\bf b},\bar{1}} }\beta,
  \\
\rho_{\bf b}
 &:=\wt{\rho}_{\bf b}+\frac{m-n+1}{2}\mathrm{Str}.
 \end{split}
\end{align}
We shall always use the normalized $\rho_{\bf b}$, which behaves as
well as $\wt{\rho}_{\bf b}$ in most circumstances and is more
convenient for our purpose.

\begin{lem}  \label{lem:rho}
The element $\rho_{\bf b} \in \h_{m|n}^*$ is characterized by the
following two properties:
\begin{itemize}
\item[(i)]
$(\rho_{\bf b}\vert \beta) =\hf(\beta\vert \beta)$, for every simple root
$\beta\in \Pi(\bf{b})$.

\item[(ii)]
$%\begin{eqnarray*}
(\rho_{\bf b}\vert \ep_{m+n}) =\left\{
 \begin{array}{ll}
 0, & \text{ if } b_{m+n} =1,
 \\
 1, & \text{ if } b_{m+n} =0.
 \end{array}
 \right.
$% \end{eqnarray*}
\end{itemize}
Moreover, we have $\rho_{\bf b} \in \wtl$, for every $\bf b$.
\end{lem}
Lemma~\ref{lem:rho} implies that $\rho_{\bf b}$ here coincides with
the one used by Kujawa \cite[2.7]{Ku}.

\begin{proof}
Clearly an element satisfying (i) and (ii) is unique. Thanks to
\eqref{eq:strace}, we have $(\rho_{\bf b}\vert \beta) =(\wt{\rho}_{\bf
b}\vert \beta) =\hf(\beta\vert\beta)$, for every simple root $\beta\in
\Pi(\bf{b})$. So it remains to show that $\rho_{\bf b}$ defined in
\eqref{eq:2rho} satisfies (ii).

A direct computation shows that, for ${\bf b_{\text st}}
=(0,\ldots,0,1,\ldots,1)$,
$$
\rho_{\bf b_{\text st}}=\sum_{i=1}^m(m+1-i-n)\ep_i +
\sum_{j=m+1}^{m+n}(m+n-j)\ep_j,
$$
and so $\rho_{\bf b_{\text st}}$ satisfies (ii). As observed in
proof of Proposition~\ref{prop:sameO}, the Borel subalgebra $\mf b$
(associated to $\bf b$) and $\mf b_{\rm st}$ can be converted to one
another via a sequence of odd reflections. So it remains to verify
the following consistency of (ii): if the property (ii) holds for
${\bf b}=({\bf b}^1,{0},{1},{\bf b}^2)$ (with $0$ and $1$ at $\ka$th
and $(\ka+1)$th places) then it holds for the adjacent sequence
${\bf b'}=({\bf b}^1,{1},{0},{\bf b}^2)$. Note that $\rho_{\bf{b'}}
=\rho_{\bf{b}}+\alpha$, where $\alpha =\ep_\ka^{\bf
b}-\ep_{\ka+1}^{\bf b}$. Now this consistency of (ii) follows from
by a quick case-by-case checking, depending on whether or not
${\bf b}^2$ is empty.
\end{proof}

Define a bijection
\begin{equation}  \label{eq:bijXZ}
\wtl \longrightarrow \Z^{m+n}, \qquad  \la \mapsto f^{\bf b}_\la,
\end{equation}
where $f^{\bf b}_\la\in\Z^{m+n}$ is defined by letting
\begin{align}\label{la:to:fla}
f^{\bf b}_\la(i):=(\la+\rho_{\bf b}\vert\ep_i^{\bf b}),\quad \forall i\in [m+n].
\end{align}
The {\em Bruhat ordering} with respect to the Borel subalgebra
$\mf{b}$ is the partial ordering on $\wtl$ induced by the Bruhat
ordering $\succeq_{\bf b}$ on $\Z^{m+n}$ under the above bijection.
This terminology is justified by the role it plays in representation
theory of $\glmn$ (see \cite[Section~2.2]{CWbook}; also see
\cite{Br, Se} in the case of the standard Borel ${\mf
b}_{\text{st}}$). The Bruhat ordering on $\wtl$ will also be denoted
by $\succeq_{\bf b}$ by abuse of notation. Recall $d_i$ from
\eqref{eq:di}. For adjacent sequences ${\bf b}$ and ${\bf b}'$ (for
notations see \S\ref{sec:CB2} or the above proof of
Lemma~\ref{lem:rho}), we have
\begin{align*}
f_{\la}^{{\bf b}'}=f_{\la}^{\bf
b}+d_{\ka}-d_{\kappa+1},\quad\forall\la\in \wtl.
\end{align*}

Now consider the standard sequence ${\bf b}_{\text{st}}$ and the
standard Borel subalgebra ${\mf b}_{\text{st}}$. A weight $\la \in \wtl$ is called
{\em typical} if $f^{{\bf b}_{\text{st}}}_\la(i) \neq f^{{\bf
b}_{\text{st}}}_\la(j)$ for all $i,j$ such that $1\le i \le m< j\le
m+n$, and $\la$ is {\em anti-dominant} if $f^{{\bf
b}_{\text{st}}}_\la(1)\leq f^{{\bf b}_{\text{st}}}_\la(2)\ldots \leq
f^{{\bf b}_{\text{st}}}_\la(m)$ and $f^{{\bf
b}_{\text{st}}}_\la(m+1) \geq \ldots \geq f^{{\bf
b}_{\text{st}}}_\la(m+n)$.

\subsection{Tilting modules}

Recall that the Lie superalgebra $\gl(m|n)$ has an automorphism
$\tau$ given by the formula:
\begin{align*}
\tau(e_{ij}):= -(-1)^{|i|(|i|+|j|)} e_{ji}.
\end{align*}
For an object $M =\oplus_{\nu\in \wtl} M_\nu\in\Omn$, we let
$$
M^\vee:=\oplus_{\nu\in \wtl} M_\nu^*
$$
be the restricted dual of $M$.
We may define an action of $\gl(m|n)$ on $M^\vee$ by $(g\cdot
f)(x):=-f(\tau(g)x)$, for $f\in M^\vee$, $g\in\gl(m|n)$, and $x\in
M$.  We denote the resulting $\gl(m|n)$-module by $M^\tau$, which is
an object in $\Omn$. An object $M\in\Omn$ is called {\em self-dual},
if $M^\tau\cong M$. Clearly, $L(\la)$ is self-dual, for all $\la\in
\wtl$.

Fix an arbitrary $0^m1^n$-sequence ${\bf b}$.  An object $M\in\Omn$
is said to have a {\em ${\bf b}$-Verma flag} (respectively, a {\em
dual ${\bf b}$-Verma flag}), if $M$ has a filtration
\begin{align*}
M_0=0\subseteq M_1\subseteq M_2\subseteq\cdots\subseteq M_t=M,
\end{align*}
such that $M_i/M_{i-1}\cong M_{\bf b}(\gamma_i)$ (respectively,
$M_i/M_{i-1}\cong M_{\bf b}(\gamma_i)^\tau$), for some $\gamma_i\in \wtl$ and
$1\le i\le t$.

\begin{definition}\label{def:tilt}
Associated with each $\la\in \wtl$, a ${\bf b}$-{\em tilting module}
$T_{\bf b}(\la)$ is an indecomposable $\gl(m|n)$-module in $\mc
O^{m|n}$ satisfying the following two conditions:
\begin{itemize}
\item[(i)]
$T_{\bf b}(\la)$ has a ${\bf b}$-Verma flag with $M_{\bf b}(\la)$ at
the bottom.

\item[(ii)]
$\text{Ext}^1_{\Omn}(M_{\bf b}(\mu),T_{\bf b}(\la))=0$, for all
$\mu\in \wtl$.
\end{itemize}
\end{definition}

Combining \cite[Theorem 6.3, Lemma 7.3]{Br2} with \cite[Theorem
6.4]{Br}, as a super generalization of \cite{So2}, we conclude that
the ${\bf b}$-tilting module $T_{\bf b}(\la)$, for every $\la\in \wtl$,
in the category $\Omn$ exists and is unique (nevertheless, it
depends on $\bf b$). Let $\mc O_{\bf b}^{m|n,\Delta}$ denote the
full subcategory of $\mc O^{m|n}$ consisting of objects that
have finite ${\bf b}$-Verma flags.

The following lemma is standard in a highest weight category
\cite{Don} (for a proof see e.g.~\cite[Proposition 3.7]{CW}).

\begin{lem}\label{lem:tilt:aux1}
Let ${\bf b}$ be a ${0^m1^n}$-sequence.
\begin{itemize}
\item[(i)]
If $N\in\Omn$ has a ${\bf b}$-Verma flag, then
$\text{Ext}^1_{\Omn}(N,M_{\bf b}(\mu)^\tau)=0$, for all $\mu\in \wtl$.

\item[(ii)]
$N\in\Omn$ has a dual ${\bf b}$-Verma flag if and only if
$\text{Ext}^1_{\Omn}(M_{\bf b}(\mu),N)=0$, for all $\mu\in \wtl$.
\end{itemize}
\end{lem}

We have the following useful characterization of tilting modules,
which is well known in the algebraic group or Kac-Moody setting (cf.
\cite{Don, So}). The same proof can be adapted in our setting, using
Lemma~\ref{lem:tilt:aux1} and \cite[Proposition 5.6]{So2}.

\begin{lem}\label{lem:char:tilt}
A $\gl(m|n)$-module $T\in\Omn$ is a ${\bf b}$-tilting module if and
only if $T$ is an indecomposable self-dual module that has a ${\bf
b}$-Verma flag.
\end{lem}
%
%\begin{proof}
%($\Rightarrow$). Take $T=T_{\bf b}(\la)$. Lemma~
%\ref{lem:tilt:aux1}(i) implies that $\text{Ext}^1(T_{\bf
%b}(\la),M_{\bf b}(\mu)^\tau)=0$, and hence, applying the functor
%$\tau$, we get $\text{Ext}^1(M_{\bf b}(\mu),T_{\bf b}(\la)^\tau)=0$.
%On the other hand, Soergel's construction of tilting modules in
%\cite[Proposition 5.6]{So}, when applied to ${\bf b}$-tilting
%modules, implies that there are no weights in $T_{\bf b}(\la)^\tau$
%greater than $\la$. Since $T_{\bf b}(\la)^\tau$ also has a ${\bf
%b}$-Verma flag, it follows by the uniqueness of tilting modules that
%$T_{\bf b}(\la)=T_{\bf b}(\la)^\tau$.
%
%($\Leftarrow$). Conversely, assume $T\in\Omn$ is indecomposable,
%self-dual, and has a ${\bf b}$-Verma flag. By Lemma
%\ref{lem:tilt:aux1}(ii), $\text{Ext}^1_{\Omn}(M_{\bf b}(\mu),T)=0$.
%Since $T$ has a ${\bf b}$-Verma flag, it must be a ${\bf b}$-tilting
%module corresponding to some $\la\in P$.
%\end{proof}

When dealing with the standard $0^m1^n$-sequence ${\bf
b}_{\text{st}}$ and the standard Borel $\mf{b}_{\text{st}}$, we
shall continue the convention of suppressing ${\bf b}_{\rm st}$, and
hence denote the ${\bf b}_{\text{st}}$-tilting modules by $T(\la)$.

\begin{prop}\label{prop:tilt:diff:borel}
Let $T(\la)$ be the tilting module corresponding to $\la\in \wtl$ in
$\Omn$. Then $T(\la)$ is a ${\bf b}$-tilting module, for an
arbitrary $0^m1^n$-sequence ${\bf b}$.
\end{prop}

\begin{proof}
Fix an arbitrary $0^m1^n$-sequence ${\bf b}$.

Let $\mu\in \wtl$ be anti-dominant and typical. Then it is well known that
(cf. e.g.~\cite{Kac2, Se, Br}) the Verma module $M(\mu)$ with
respect to the standard Borel $\mf{b}_{\text{st}}$ is irreducible,
and hence $M(\mu)^\tau\cong M(\mu)$. This implies that
$L(\mu)=M(\mu)=T(\mu)$. Hence, $M(\mu)$ is equal to a $\bf b$-Verma
module $M_{\bf b}(\mu^{\bf b})$, for some $\mu^{\bf b}\in \wtl$.

Now fix $\la\in \wtl$. Then it is easy to find an anti-dominant
typical weight $\mu \in \wtl$ and a weight $\gamma\in \wtl$ such
that $\text{dim}_\C L(\gamma)<\infty$ and $\la=\mu+\gamma$. The
$\gl(m|n)$-module $M(\mu)\otimes L(\gamma)$ is self-dual (as a
tensor product of two simples) and has a ${\bf b}_{\text{st}}$-Verma
flag, in which $M(\la)$ appears as a subquotient exactly once. By
some standard argument which goes back to Soergel, any direct
summand of $M(\mu)\otimes L(\gamma)$ is self-dual and also has a
${\bf b}_{\text{st}}$-Verma flag. The unique summand $T$ containing
$M(\la)$ must have $M(\la)$ at the bottom, since $\la$ is the
highest weight in $M(\mu)\otimes L(\gamma)$. Hence we have $T\cong
T(\la)$.

Since $M(\mu)=M_{\bf
b}(\mu^{\bf b})$, % clearly has a ${\bf b}$-Verma flag,
the tensor product
$M(\mu)\otimes L(\gamma)$ also has a ${\bf b}$-Verma flag.  Now the
indecomposable summand $T$ has a ${\bf b}$-Verma flag and is also
self-dual. Hence, it must be a ${\bf b}$-tilting module by Lemma~
\ref{lem:char:tilt}.
\end{proof}

Let ${\bf b}$ and ${\bf b}'$ be two adjacent ${0^m1^n}$-sequences
such that the corresponding Borel subalgebra ${\mf b'}$ is obtained
from ${\mf b}$ via the odd reflection with respect to the simple
root $\alpha$ of $\mf{b}$. We introduce the following notation:
\begin{align} \label{eq:laU}
\la^{\mathbb U} =
\begin{cases}
 \la-2\alpha,         &\text{ if }(\la,\alpha)=0,\\
\la-\alpha,   &\text{ if }(\la,\alpha)\neq 0,
\end{cases}
\qquad \text{for } \la \in \wtl.
\end{align}
The following may be regarded as a ``dual version'' of Lemma~
\ref{lem:irred:borels}.

\begin{thm}\label{thm:tilt:relation}
%Let ${\bf b}=({\bf b}^1,{0},{1},{\bf b}^2)$ and ${\bf b'}=({\bf
%b}^1,{1},{0},{\bf b}^2)$ with $\kappa$ in \eqref{kappa}. Let $\alpha
%=\ep_\ka^{\bf b}-\ep_{\ka+1}^{\bf b}$.
%
Let ${\bf b}$ and ${\bf b}'$ be two adjacent ${0^m1^n}$-sequences
such that the Borel subalgebra ${\mf b'}$ is obtained from ${\mf b}$ via the
odd reflection with respect to the simple root $\alpha$ of $\mf{b}$.
Then
$$
T_{\bf b}(\la)= T_{\bf b'}(\la^{\mathbb U}), \qquad \text{for } \la \in \wtl.
$$
\end{thm}

\begin{proof}
By Proposition \ref{prop:tilt:diff:borel}, the ${\bf b}$-tilting
module $T_{\bf b}(\la)$ is also a ${\bf b'}$-tilting module. Since
the ${\bf b'}$-tilting modules form a basis of the Grothendieck
group of $\mc O^{m|n,\Delta}_{\bf b'}$, in order to prove the
theorem, it suffices to prove the following character identities:
\begin{align*}
\text{ch}T_{\bf b}(\la)=\begin{cases}
\text{ch}T_{\bf b'}(\la-2\alpha),&\text{ if }(\la,\alpha)=0,\\
\text{ch}T_{\bf b'}(\la-\alpha),&\text{ if }(\la,\alpha)\not=0.
\end{cases}
\end{align*}

By Soergel's character formula for tilting modules \cite[Theorem
6.7]{So2} and its super generalization \cite[Theorem 6.4]{Br2}, we
have, for an arbitrary ${\bf b}$,
\begin{align}\label{tilt:verma:irred}
\left(T_{\bf b}(\la):M_{\bf b}(\mu)\right)=[M_{\bf
b}(-\mu-2\rho_{\bf b}):L_{\bf b}(-\la-2\rho_{\bf b})].
\end{align}
Using \eqref{tilt:verma:irred} we compute
\begin{align}
\text{ch}T_{\bf b}(\la)
&=\sum_{\mu}\left(T_{\bf b}(\la):M_{\bf b}(\mu)\right)\text{ch}M_{\bf b}(\mu)
 \notag \\
&=\sum_{\mu}[M_{\bf b}(-\mu-2\rho_{\bf b}):L_{\bf b}(-\la-2\rho_{\bf
b})]\text{ch}M_{\bf b}(\mu).
 \label{eq:chTb}
\end{align}
We now apply Lemmas \ref{lem:verma:borels} and
\ref{lem:irred:borels}, and the identity $\rho_{\bf b'}=\rho_{\bf
b}+\alpha$ in two separate cases. We shall also need
\eqref{tilt:verma:irred} for varying $\bf b$, $\la, \mu$.

Case (i). Assume $(\la,\alpha)=0$. Continuing \eqref{eq:chTb}, we
have {\allowdisplaybreaks
\begin{align*}
\text{ch}T_{\bf b}(\la)
 &= \sum_{\mu} \big[M_{\bf b'}(-\mu-2\rho_{\bf
b}-\alpha): L_{\bf b'}(-\la-2\rho_{\bf b})\big] \ \text{ch}M_{\bf
b'}(\mu-\alpha)
 \\
&=\sum_{\mu} \big[M_{\bf b'}(-\mu-2\rho_{\bf b'}+\alpha): L_{\bf
b'}(-\la-2\rho_{\bf b'}+2\alpha)\big] \ \text{ch}M_{\bf
b'}(\mu-\alpha)
 \\
&=\sum_{\mu}\big(T_{\bf b'}(\la-2\alpha):M_{\bf
b'}(\mu-\alpha)\big)\ \text{ch}M_{\bf b'}(\mu-\alpha)
 \\
&=\text{ch}T_{\bf b'}(\la-2\alpha).
\end{align*}

Case (ii). Assume $(\la,\alpha)\not=0$. Continuing \eqref{eq:chTb}
again, we have
\begin{align*}
\text{ch}T_{\bf b}(\la) &=\sum_{\mu}\big[M_{\bf b'}(-\mu-2\rho_{\bf
b}-\alpha):
L_{\bf b'}(-\la-2\rho_{\bf b}-\alpha)\big]\ \text{ch}M_{\bf b'}(\mu-\alpha)\\
&= \sum_{\mu}\big[M_{\bf b'}(-\mu-2\rho_{\bf b'}+\alpha):
L_{\bf b'}(-\la-2\rho_{\bf b}+\alpha)\big]\ \text{ch}M_{\bf b'}(\mu-\alpha)\\
&= \sum_{\mu}\big(T_{\bf b'}(\la-\alpha):
M_{\bf b'}(\mu-\alpha)\big)\ \text{ch}M_{\bf b'}(\mu-\alpha)\\
&=
\text{ch}T_{\bf b'}(\la-\alpha).
\end{align*}}
This completes the proof.
\end{proof}

\subsection{Auxiliary modules}

Let ${\bf b}=(b_1,\ldots,b_{m+n})$ and ${\bf b'}$ be two
${0^m1^n}$-sequences adjacent by the simple odd root
$\alpha=\ep^{\bf b}_\ka-\ep^{\bf b}_{\ka+1}$ of $\mf{b}$ as before,
and let $\mf b$ and $\mf b'$ be the corresponding Borel subalgebras again. For
definiteness let us assume that $(\ep^{\bf b}_\ka,\ep^{\bf
b}_\ka)=1=-(\ep^{\bf b}_{\ka+1},\ep^{\bf b}_{\ka+1})$, i.e.,
$b_\ka=0, b_{\ka+1}=1$.

Let $v_\la$ be a ${\bf b}$-highest weight vector of the ${\bf
b}$-Verma module $M_{\bf b}(\la)$. We denote by $e_{\pm\alpha}$ the
root vectors corresponding to the roots $\pm\alpha$.

Suppose that $(\la, \alpha) = 0$. The Lie superalgebra
\begin{align*}
\mf{a}_\alpha:=\h_{m|n}+\C e_{\alpha}+\C e_{-\alpha}
\end{align*}
is isomorphic to a direct sum of $\gl(1|1)$ and a subalgebra of
$\h_{m|n}$. Thus, the Verma module of $\mf{a}_\alpha$ of highest
weight $\la$, denoted by $M_{(b_\ka,b_{\ka+1})}(\la)$, is
two-dimensional. The irreducible modules of $\mf{a}_\alpha$ of
highest weight $\la$ and $\la-\alpha$, denoted by $\C_\la$ and
$\C_{\la-\alpha}$, respectively, are one-dimensional and we have the
following exact sequence of $\mf{a}_\alpha$-modules
\begin{align}\label{aux:106}
0\longrightarrow \C_{\la-\alpha}\longrightarrow
M_{(b_\ka,b_{\ka+1})}(\la)\longrightarrow \C_\la\longrightarrow 0.
\end{align}
We denote by $\mf n$ the radical corresponding to the Borel
subalgebra $\mf b$, and by $\mf n_{\neq\alpha}$ the subalgebra of
$\mf n$ spanned by the root spaces $\gl(m|n)_{\beta}$, for $\beta
\neq \alpha$. We observe that $\mf{a}_\alpha+\mf{n}_{\neq\alpha} =
\mf{b}+\C e_{-\alpha}$ and it contains $\mf{n}_{\neq\alpha}$ as an
ideal. Thus, \eqref{aux:106} extends trivially to an exact sequence
of $(\mf{b}+\C e_{-\alpha})$-modules.

Noting that $(\la,-\alpha)=0$, we can switch the role of $\alpha$
with $-\alpha$ (and $\bf b$ with $\bf b'$ at the same time) above.
Regarding $\C_\la$ as the one-dimensional $(\mf b+\C
e_{-\alpha})$-module or similarly regarding $\C_\la$ as the
one-dimensional $(\mf b'+\C e_\alpha)$-module, we may form the
parabolic Verma modules
$$
N_{\bf b}(\la) := \text{Ind}_{\mf b +\C e_{-\alpha}}^{\gl(m|n)}
\C_\la,
 \qquad
N_{\bf b'}(\la) := \text{Ind}_{\mf b'+\C e_\alpha}^{\gl(m|n)} \C_\la
$$
Observing $\mf b'+\C e_\alpha =\mf b+\C e_{-\alpha}$, we have
$N_{\bf b}(\la) =N_{\bf b'}(\la). $

We continue to assume that $(\la, \alpha) = 0$. The tilting
$\mf{a}_\alpha$-module of highest weight $\la$ will be denoted by
$T_{(b_\ka,b_{\ka+1})}(\la)$. We have the following exact sequence
of $\mf{a}_\alpha$-modules (see \cite[Theorem~4.37 for $m=n=1$]{Br}
and compare with Lemma~\ref{lem:formula:can:VW}):
\begin{align}\label{aux:105}
0\longrightarrow M_{(b_\ka,b_{\ka+1})}(\la)\longrightarrow
T_{(b_\ka,b_{\ka+1})}(\la)\longrightarrow
M_{(b_\ka,b_{\ka+1})}(\la-\alpha)\longrightarrow 0,
\end{align}
As before, \eqref{aux:105} may be regarded as an exact sequence of
$(\mf{b}+\C e_{-\alpha})$-modules with trivial action by
$\mf{n}_{\neq\alpha}$. We form the $\gl(m|n)$-module
\begin{align*}
U_{\bf b}(\la):= \text{Ind}_{\mf b +\C
e_{-\alpha}}^{\gl(m|n)} T_{(b_\ka,b_{\ka+1})}(\la).
\end{align*}
We can similarly form the module $U_{\bf b'}(\la) := \text{Ind}_{\mf
b'+\C e_\alpha}^{\gl(m|n)}  T_{(b_{\ka+1},b_\ka)}(\la)$. We
note that
$$
U_{\bf b}(\la) = U_{\bf b'}(\la-2\alpha)
$$
since
$T_{(b_\ka,b_{\ka+1})}(\la)=T_{(b_{\ka+1},b_\ka)}(\la-2\alpha)$ by
Theorem \ref{thm:tilt:relation}.

In the case that $(\la, \alpha) \neq 0$, we define $U_{\bf b}(\la), N_{\bf
b}(\la), U_{{\bf b}'}(\la)$ and $N_{{\bf b}'}(\la)$ to be
\begin{equation}  \label{UNM}
U_{\bf b}(\la) =N_{\bf b}(\la) =M_{\bf b}(\la), \qquad U_{{\bf
b}'}(\la) =N_{{\bf b}'}(\la) =M_{{\bf b}'}(\la).
\end{equation}

Recall the notation $\la^{\mathbb L}$ from \eqref{eq:laL} and
$\la^{\mathbb U}$ from \eqref{eq:laU}. Summarizing the above two
cases and using Lemma~\ref{lem:verma:borels}, we have
\begin{align}  \label{NUadj}
{\rm ch}\, N_{\bf b}(\la)={\rm ch}\, N_{\bf b'}(\la^{\mathbb L}),
 \quad
{\rm ch}\, U_{\bf b}(\la)={\rm ch}\, U_{\bf b'}(\la^{\mathbb U}),
\qquad\text{ for } \la \in \wtl.
\end{align}

\begin{lem}\label{expand:repn:M:N}
Let ${\bf b}$ be a ${0^m1^n}$-sequence and $\alpha=\ep^{\bf
b}_\ka-\ep^{\bf b}_{\ka+1}$ be an isotropic simple root as above.
Let $\la\in \wtl$ be such that $(\la,\alpha)=0$. Then we have the
following short exact sequences of $\gl(m|n)$-modules:
\begin{align*}
0 \longrightarrow N_{\bf b}(\la-\alpha) \longrightarrow M_{\bf b}(\la)
\longrightarrow N_{\bf b}(\la) \longrightarrow 0,
 \\
0 \longrightarrow M_{\bf b}(\la) \longrightarrow U_{\bf b}(\la)
\longrightarrow M_{\bf b}(\la-\alpha) \longrightarrow 0.
\end{align*}
\end{lem}

\begin{proof}
The first exact sequence is obtained by regarding the exact sequence
\eqref{aux:106} as an exact sequence of $(\mf{b}+\C
e_{-\alpha})$-modules, and then inducing it to an exact sequence of
$\gl(m|n)$-modules. The second exact sequence is obtained similarly,
now using \eqref{aux:105} in place of \eqref{aux:106}.
\end{proof}

By \eqref{UNM} and Lemma~\ref{expand:repn:M:N}, we have
\begin{align} \label{eq:MNla}
\text{ch}\,M_{\bf b}(\la) =
 \begin{cases}
\text{ch}\,N_{\bf b}(\la) +\text{ch}\,N_{\bf b}(\la-\alpha),
  &\text{ if }(\la,\alpha)=0,\\
\text{ch}\,N_{\bf b}(\la),&\text{ if }(\la,\alpha)\not=0;
\end{cases}
\end{align}
\begin{align} \label{eq:UMla}
\text{ch}\,U_{\bf b}(\la) =
 \begin{cases}
\text{ch}\,M_{\bf b}(\la) +\text{ch}\,M_{\bf b}(\la-\alpha),
  &\text{ if }(\la,\alpha)=0,\\
\text{ch}\,M_{\bf b}(\la),&\text{ if }(\la,\alpha)\not=0.
\end{cases}
\end{align}

\begin{rem}
Similarly, we have a short exact sequence of $\gl(m|n)$-modules:
\begin{align*}
0 \longrightarrow N_{\bf b'}(\la) \longrightarrow M_{\bf b'}(\la-\alpha)
\longrightarrow N_{\bf b'}(\la-\alpha) \longrightarrow 0.
\end{align*}
Since $N_{\bf b}(\la)=N_{\bf b'}(\la)$ and $N_{\bf
b}(\la-\alpha)=N_{\bf b'}(\la-\alpha)$, we see that $M_{\bf b}(\la)$
and $M_{\bf b'}(\la-\alpha)$ are opposite extension of two modules.
\end{rem}

\begin{rem}
All the results in Section~ \ref{sec:repn:prep} remain valid for an
arbitrary basic Lie superalgebra, such as $\mf{osp}(m|2n)$ (cf.
\cite{CWbook}). This, in particular, applies to Propositions
\ref{prop:sameO} and \ref{prop:tilt:diff:borel}, Theorem
\ref{thm:tilt:relation}, and Lemma \ref{expand:repn:M:N}.
\end{rem}

\section{Super duality for general linear Lie superalgebras}
\label{sec:SD}

In this section, we establish a super duality, which is a certain
equivalence of categories and identification of Kazhdan-Lusztig
theories. In contrast to earlier formulations by the authors, we
allow the head (Dynkin) diagrams to correspond to Lie superalgebras.
The equivalence established here will be needed for an inductive
argument in the proof of Brundan's conjecture next section.

\subsection{Infinite-rank Lie superalgebras}
\label{Sec:GGG}

Define the sets
\begin{align*}
\wt{\mathbb I}
&:=\Big\{1,2,\ldots,m+n;\ul{\hf},\ul{1},\ul{\frac{3}{2}},\ldots\Big\},
 \\
{\mathbb I} &:=\{1,2,\ldots,m+n;\ul{1},\ul{2},\ul{3},\ldots\},
 \\
\breve{\mathbb I} &:=\Big\{1,2,\ldots,m+n;\ul{\hf},
\ul{\frac{3}{2}},\ul{\frac{5}{2}},\ldots\Big\}.
\end{align*}
Let ${\bf b}=(b_1,b_2,\ldots,b_{m+n})$ be a $0^m1^n$-sequence.

Let $\wt{V}$ denote the complex vector superspace with homogeneous
ordered basis $\{e^{\bf b}_{i}\vert 1\le i\le
m+n\}\cup\{e_{\ul{r}}\vert r\in\hf\N\}$. Recall that the
$\Z_2$-gradation of $e^{\bf b}_i$ is given by $|e^{\bf b}_i|=b_i$.
The $\Z_2$-gradation for the $e_{\ul{r}}$'s is defined by
$|e_{\ul{r}}|=\ov{2r}\in\Z_2$. We denote by ${\DG}$ the Lie
superalgebra of endomorphisms of $\wt{V}$ vanishing on all but
finitely many $e_r$'s, $r\in\wt{\mathbb I}$. For $r,
s,p\in\wt{\mathbb I}$, let $E_{rs}$ denote the endomorphism defined
by $E_{rs}(e_p):=\delta_{sp}e_r$. Then ${\DG}$ has a basis given by
$\{E_{rs} |r,s \in \wt{\mathbb I}\}$. The subalgebra spanned by
$\{E_{ij} |1\le i,j\le m+n\}$ is isomorphic to $\gl(m|n)$.

Let $\wt\h$ stand for the Cartan subalgebra spanned by $\{E_{rr} |r
\in \wt{\mathbb I}\}$, and let $\wt{\h}^*$ denote its restricted
dual. We may regard the elements $\epsilon^{\bf b}_i$ $(1\le i\le
m+n)$ as elements in $\wt{\h}^*$ in a natural way. For $r\in\hf\N$,
define $\delta_{r}\in\wt{\h}^*$ to be the element determined by
\begin{align*}
\delta_{r}(E_{ss})=\delta_{\ul{r}s},\quad s\in\wt{\mathbb I},
\end{align*}
so that $\{\ep^{\bf b}_i,\delta_r|1\le i\le m+n,r\in\hf\N\}$ is a
basis for $\wt{\mf h}^*$. Denote the set of roots of ${\DG}$ by
$\wt{\Phi}$.  The ordered basis $\{e^{\bf b}_1,\ldots,e^{\bf
b}_{m+n},e_{\ul{\hf}},e_{\ul{1}},\ldots\}$ of $\wt{V}$ determines a
Borel subalgebra $\wt{\mc{B}}_{\bf b}$ with the simple system
\begin{align*}
\Pi(\wt{\mc{B}}_{\bf b})=\{\epsilon^{\bf b}_{1}-\epsilon^{\bf
b}_2,\ldots,\epsilon^{\bf b}_{m+n-1}-\epsilon^{\bf
b}_{m+n}\}\cup\{\ep^{\bf b}_{m+n}-\delta_{\hf}\}\cup
\big\{\delta_r-\delta_{r+\hf}\vert r\in\hf\N\big\}.
\end{align*}
Denote the Dynkin diagram of the Lie superalgebra $\gl(m|n)$ with
respect to the Borel subalgebra $\mf{b}$ by
\makebox(36,0){$\oval(36,14)$}\makebox(-30,8){$\mf{T}^{\bf b}$}.
Then the corresponding Dynkin diagram of $\wt{\mc{B}}_{\bf b}$
together with $\Pi(\wt{\mc{B}}_{\bf b})$ is given by
\begin{center}
\hskip 2cm \setlength{\unitlength}{0.16in}
\begin{picture}(24,2.5)
\put(0,1){\makebox(0,0)[c]{{\ovalBox(3.0,1.4){$\mf{T}^{\bf b}$}}}}
\put(1.5,1){\line(1,0){2}}
\put(3.9,1){\makebox(0,0)[c]{$\bigotimes$}}
\put(4.3,1){\line(1,0){2}}
\put(6.7,1){\makebox(0,0)[c]{$\bigotimes$}}
\put(7.1,1){\line(1,0){2}}
\put(9.5,1){\makebox(0,0)[c]{$\bigotimes$}}
\put(9.9,1){\line(1,0){2}} \put(12.8,1){\makebox(0,0)[c]{$\cdots$}}
\put(3.5,0){\makebox(0,0)[c]{\tiny$\ep^{\bf b}_{m+n}-\delta_{\hf}$}}
\put(6.6,0){\makebox(0,0)[c]{\tiny$\delta_{\hf}-\delta_{1}$}}
\put(9.5,0){\makebox(0,0)[c]{\tiny$\delta_{1}-\delta_{\frac{3}{2}}$}}
\put(17.5,0.5){\text{ if }  $b_{m+n}={0}$;}
\end{picture}

\end{center}
%and in the case when $b_{m+n}={1}$ is
\begin{center}
\hskip 2cm \setlength{\unitlength}{0.16in}
\begin{picture}(24,2.5)
\put(0,1){\makebox(0,0)[c]{{\ovalBox(3.0,1.4){$\mf{T}^{\bf b}$}}}}
\put(1.5,1){\line(1,0){2}} \put(3.9,1){\makebox(0,0)[c]{$\bigcirc$}}
\put(4.3,1){\line(1,0){2}}
\put(6.7,1){\makebox(0,0)[c]{$\bigotimes$}}
\put(7.1,1){\line(1,0){2}}
\put(9.5,1){\makebox(0,0)[c]{$\bigotimes$}}
\put(9.9,1){\line(1,0){2}} \put(12.8,1){\makebox(0,0)[c]{$\cdots$}}
\put(3.5,0){\makebox(0,0)[c]{\tiny$\ep^{\bf b}_{m+n}-\delta_{\hf}$}}
\put(6.6,0){\makebox(0,0)[c]{\tiny$\delta_{\hf}-\delta_{1}$}}
\put(9.5,0){\makebox(0,0)[c]{\tiny$\delta_{1}-\delta_{\frac{3}{2}}$}}
\put(17.5,0.5){\text{ if }  $b_{m+n}={1}$.}
\end{picture}
\vspace{.3cm}
\end{center}

Let $\G$ and ${\SG}$ be the Lie subalgebras of ${\DG}$ spanned by
$\{E_{rs}\vert r,s\in\mathbb I\}$ and $\{E_{rs}\vert
r,s\in\breve{\mathbb I}\}$, respectively. The Cartan subalgebras of
$\G$ and ${\SG}$ are $\h=\wt{\h}\cap \G$ and $\breve\h=\wt{\h}\cap
\breve\G$, and their restricted duals are denoted by ${\h}^*$ and
$\breve{\h}^*$, respectively. The subalgebras $\mc{B}_{\bf
b}=\wt{\mc{B}}_{\bf b}\cap\G$ and $\breve{\mc{B}}_{\bf
b}=\wt{\mc{B}}_{\bf b}\cap{\SG}$ are Borel subalgebras of $\G$ and
${\SG}$, respectively. The simple systems of $\G$ and $\breve{\G}$ with respect to $\mc{B}_{\bf b}$ and $\breve{\mc{B}}_{\bf b}$
are denoted by $\Pi(\breve{\mc{B}}_{\bf b})$ and $\Pi(\mc{B}_{\bf
b})$, respectively. The Dynkin diagrams with
$\Pi(\breve{\mc{B}}_{\bf b})$ and $\Pi(\mc{B}_{\bf b})$ specified
are as follows.
\begin{center}
\hskip 2cm \setlength{\unitlength}{0.16in}
\begin{picture}(24,2.5)
\put(-4,.5){$\G:$}
\put(0,1){\makebox(0,0)[c]{{\ovalBox(3.0,1.4){$\mf{T}^{\bf b}$}}}}
\put(1.5,1){\line(1,0){2}} \put(3.9,1){\makebox(0,0)[c]{$\bigcirc$}}
\put(4.3,1){\line(1,0){2}} \put(6.7,1){\makebox(0,0)[c]{$\bigcirc$}}
\put(7.1,1){\line(1,0){2}} \put(9.5,1){\makebox(0,0)[c]{$\bigcirc$}}
\put(9.9,1){\line(1,0){2}} \put(12.8,1){\makebox(0,0)[c]{$\cdots$}}
\put(3.5,0){\makebox(0,0)[c]{\tiny$\ep^{\bf b}_{m+n}-\delta_{1}$}}
\put(6.7,0){\makebox(0,0)[c]{\tiny$\delta_{1}-\delta_{2}$}}
\put(9.5,0){\makebox(0,0)[c]{\tiny$\delta_{2}-\delta_{3}$}}
\put(17.5,0.5){\text{ if }  $b_{m+n}=0$;}
\end{picture}
\end{center}
\begin{center}
\hskip 2cm \setlength{\unitlength}{0.16in}
\begin{picture}(24,2.5)
\put(-4,.5){$\SG:$}
\put(0,1){\makebox(0,0)[c]{{\ovalBox(3.0,1.4){$\mf{T}^{\bf b}$}}}}
\put(1.5,1){\line(1,0){2}}
\put(3.9,1){\makebox(0,0)[c]{$\bigotimes$}}
\put(4.3,1){\line(1,0){2}} \put(6.7,1){\makebox(0,0)[c]{$\bigcirc$}}
\put(7.1,1){\line(1,0){2}} \put(9.5,1){\makebox(0,0)[c]{$\bigcirc$}}
\put(9.9,1){\line(1,0){2}} \put(12.8,1){\makebox(0,0)[c]{$\cdots$}}
\put(3.5,0){\makebox(0,0)[c]{\tiny$\ep^{\bf b}_{m+n}-\delta_{\hf}$}}
\put(6.6,0){\makebox(0,0)[c]{\tiny$\delta_{\hf}-\delta_{\frac{3}{2}}$}}
\put(9.5,0){\makebox(0,0)[c]{\tiny$\delta_{\frac{3}{2}}-\delta_{\frac{5}{2}}$}}
\put(17.5,0.5){\text{ if }  $b_{m+n}=0$;}
\end{picture}
\vspace{.3cm}
\end{center}
%In the case when $b_{m+n}={1}$ their Dynkin diagrams are given as follows.
\begin{center}
\hskip 2cm \setlength{\unitlength}{0.16in}
\begin{picture}(24,2.5)
\put(-4,.5){$\G:$}
\put(0,1){\makebox(0,0)[c]{{\ovalBox(3.0,1.4){$\mf{T}^{\bf b}$}}}}
\put(1.5,1){\line(1,0){2}}
\put(3.9,1){\makebox(0,0)[c]{$\bigotimes$}}
\put(4.3,1){\line(1,0){2}} \put(6.7,1){\makebox(0,0)[c]{$\bigcirc$}}
\put(7.1,1){\line(1,0){2}} \put(9.5,1){\makebox(0,0)[c]{$\bigcirc$}}
\put(9.9,1){\line(1,0){2}} \put(12.8,1){\makebox(0,0)[c]{$\cdots$}}
\put(3.5,0){\makebox(0,0)[c]{\tiny$\ep^{\bf b}_{m+n}-\delta_{1}$}}
\put(6.6,0){\makebox(0,0)[c]{\tiny$\delta_{1}-\delta_{{2}}$}}
\put(9.5,0){\makebox(0,0)[c]{\tiny$\delta_{{2}}-\delta_{{3}}$}}
\put(17.5,0.5){\text{ if }  $b_{m+n}={1}$;}
\end{picture}
\end{center}
\begin{center}
\hskip 2cm \setlength{\unitlength}{0.16in}
\begin{picture}(24,2.5)
\put(-4,.5){$\SG:$}
\put(0,1){\makebox(0,0)[c]{{\ovalBox(3.0,1.4){$\mf{T}^{\bf b}$}}}}
\put(1.5,1){\line(1,0){2}} \put(3.9,1){\makebox(0,0)[c]{$\bigcirc$}}
\put(4.3,1){\line(1,0){2}} \put(6.7,1){\makebox(0,0)[c]{$\bigcirc$}}
\put(7.1,1){\line(1,0){2}} \put(9.5,1){\makebox(0,0)[c]{$\bigcirc$}}
\put(9.9,1){\line(1,0){2}} \put(12.8,1){\makebox(0,0)[c]{$\cdots$}}
\put(3.5,0){\makebox(0,0)[c]{\tiny$\ep^{\bf b}_{m+n}-\delta_{\hf}$}}
\put(6.6,0){\makebox(0,0)[c]{\tiny$\delta_{\hf}-\delta_{\frac{3}{2}}$}}
\put(9.6,0){\makebox(0,0)[c]{\tiny$\delta_{\frac{3}{2}}-\delta_{\frac{5}{2}}$}}
\put(17.5,0.5){\text{ if }  $b_{m+n}={1}$.}
\end{picture}
\vspace{.3cm}
\end{center}

\subsection{Parabolic categories}\label{SD:para:cat}

We define
{\allowdisplaybreaks
\begin{align}
&\wt{X} :=\Big\{\sum_{i=1}^{m+n}\la_i\ep_i
+\sum_{r\in\hf\N}{}{^+\la}_r\delta_r\mid \la_i\in\Z \text{ and
}{^+\la}_r\in\Z\Big\}\subseteq\wt{\h}^*,
 \notag \\
&{X}:=\Big\{\sum_{i=1}^{m+n}\la_i\ep_i
+\sum_{j\in\N}{^+\la}_j\delta_j\mid \la_i\in\Z \text{ and
}{^+\la}_j\in\Z\Big\}\subseteq{\h}^*,
 \label{eq:X} \\
&\breve{X} :=\Big\{\sum_{i=1}^{m+n}\la_i\ep_i
+\sum_{s\in\hf+\Z_+}{}{^+\la}_s\delta_s\mid \la_i\in\Z \text{ and
}{}{^+\la}_s\in\Z\Big\}\subseteq\breve{\h}^*.
 \notag
\end{align} }
We shall identify
$\la=\sum_{i=1}^{m+n}\la_i\ep_i+\sum_{j\in\N}{^+\la}_j\delta_j\in X$
with the tuple $(\la_1,\ldots,\la_{m+n};{^+\la})$, where we write
${^+\la}=({^+\la}_1,{^+\la}_2,\ldots)$. Recall $\mc P$ denotes the
set of all partitions. We let
\begin{align} \label{eq:X+}
{X}^+:=\Big\{\sum_{i=1}^{m+n}\la_i\ep_i+\sum_{j\in\N}{^+\la}_j\delta_j\mid
\la_i\in\Z,({^+\la}_1,{^+\la}_2,\ldots)\in\mc{P}\Big\}\subseteq{X}.
\end{align}
For a partition $\mu=(\mu_1,\mu_2,\cdots)$, let
$\mu'=(\mu'_1,\mu'_2,\ldots)$ denote the conjugate partition of
$\mu$. We also define $\theta(\mu)$ to be the sequence of integers
(which is a variant of the Frobenius notation of $\mu$)
\begin{equation*}
\theta(\mu):=(\theta(\mu)_{1/2},\theta(\mu)_1,\theta(\mu)_{3/2},\theta(\mu)_2,\cdots),
\end{equation*}
where
$$
\theta(\mu)_{i-1/2}:=\max\{\mu'_i-(i-1),0\}, \quad
\theta(\mu)_i:=\max\{\mu_i-i,0\}, \qquad \forall i\in\N.
$$
 We identify
elements in $X^+$ with tuples in $\Z^{m+n}\times\mc{P}$. For
$\la=(\la_1,\ldots,\la_{m+n},{^+\la})\in X^+$, define
\begin{align}  \label{eq:latheta}
\begin{split}
&\la^\theta :=\sum_{i=1}^{m+n}\la_i\ep_i
 +\sum_{r\in\hf\N}\theta({^+\la})_r\delta_r\in\wt{\h}^*,
  \\
&\la^\natural:=\sum_{i=1}^{m+n}\la_i\ep_i+\sum_{s\in\hf+\Z_+}
({^+\la})'_{s+\hf}\delta_s\in\breve{\h}^*.
\end{split}
\end{align}
Furthermore we set
\begin{equation} \label{eq:X+2}
\wt{X}^+:=\{\la^\theta\vert\la\in X^+\}, \qquad
\breve{X}^+:=\{\la^\natural\vert\la\in X^+\}
\end{equation}
so that we have natural bijections
\begin{align*}
\breve{X}^+\stackrel{\natural}{\longleftrightarrow} X^+
\stackrel{\theta}{\longleftrightarrow} \wt{X}^+, \qquad \la^\natural
\leftrightarrow \la \leftrightarrow \la^\theta.
\end{align*}

For a root $\alpha$ of ${\DG}$ we denote by ${\DG}_\alpha$ the root
space corresponding to $\alpha$. Similar notations apply to
$\G_\alpha$ and ${\SG}_\alpha$. Define
\begin{align*}
\wt{\mf{k}}
:=\wt{\h}+\sum_{\alpha\in \sum_{r\in\hf\N}\Z(\delta_{r}-\delta_{r+\hf})}{\DG}_\alpha,
 \qquad
\wt{\mf p}_{\bf b}:=\wt{\mc{B}}_{\bf b}+\wt{\mf{k}},\\
\breve{\mf{k}} :=\breve{\h}
+\sum_{\alpha\in \sum_{r\in-\hf+\N}\Z(\delta_{r}-\delta_{r+1})}{\SG}_\alpha,
 \qquad
\breve{\mf p}_{\bf b}:=\breve{\mc{B}}_{\bf b}+\breve{\mf{k}},\\
{\mf{k}}
:={\h}+\sum_{\alpha\in \sum_{r\in\N}\Z(\delta_r-\delta_{r+1})}{\G}_\alpha,
 \qquad
{\mf p}_{\bf b}:={\mc{B}}_{\bf b}+{\mf{k}}.
\end{align*}
Then $\wt{\mf{p}}_{\bf b}$, $\breve{\mf{p}}_{\bf b}$ and
${\mf{p}}_{\bf b}$ are parabolic subalgebras of $\DG, \SG$ and $\G$
with Levi subalgebras $\wt{\mf{k}}$, $\breve{\mf{k}}$, and
${\mf{k}}$, respectively. Let us denote the respective nilradicals
and opposite nilradicals by $\wt{\mf{u}}_{\bf b}$,
$\breve{\mf{u}}_{\bf b}$, and $\mf{u}_{\bf b}$, and
$\wt{\mf{u}}^-_{\bf b}$, $\breve{\mf{u}}^-_{\bf b}$, and
$\mf{u}^-_{\bf b}$.

For $\la\in X^+$, let $\wt{L}^0(\la^\theta)$,
$\breve{L}^0(\la^\natural)$, and $L^0(\la)$  denote the irreducible
$\wt{\mf{k}}$-, $\breve{\mf{k}}$, and ${\mf{k}}$-module of highest
weight $\la^\theta$, $\la^\natural$, and $\la$, respectively. They
can be extended in a trivial way to $\wt{\mf{p}}_{\bf b}$-,
$\breve{\mf{p}}_{\bf b}$-, and ${\mf{p}}_{\bf b}$-modules,
respectively. We form the respective parabolic Verma modules
\begin{align*}
\wt{\mc{M}}_{\bf b}(\la^\theta) =\text{Ind}_{\wt{\mf{p}}_{\bf
b}}^{\wt{\G}}\wt{L}^0(\la^\theta),
 \qquad
\breve{\mc{M}}_{\bf b}(\la^\natural)
=\text{Ind}_{\breve{\mf{p}}_{\bf
b}}^{\breve{\G}}\breve{L}^0(\la^\natural),
 \qquad
{\mc{M}}_{\bf b}(\la)=\text{Ind}_{{\mf{p}}_{\bf b}}^{\G}L^0(\la),
\end{align*}
whose unique irreducible quotients are denoted by $\wt{\mc{L}}_{\bf
b}(\la^\theta)$, $\breve{\mc{L}}_{\bf b}(\la^\natural)$, and
${\mc{L}}_{\bf b}(\la)$, respectively.

\begin{definition}\label{def:O:bfa:inf}
Let $\wt{\OO}_{\bf b}$ be the category of finitely generated
$\wt{\G}$-modules $\wt{\mc M}$ such that $\wt{\mc M}$ is a
semisimple $\wt{\h}$-module with finite-dimensional weight subspaces
$\wt{\mc M}_\gamma$, $\gamma\in \wt{X}$, satisfying the following
conditions.
\begin{itemize}
\item[(i)]
$\wt{\mc M}$ decomposes over $\wt{\mf{k}}$ into a direct sum of
$\wt{L}^0(\mu^\theta)$ for $\mu\in X^+$.

\item[(ii)]
There exist finitely many weights ${^1\la},{^2\la},\ldots,{^t\la}\in
X^+$ (depending on $\wt{\mc M}$) such that if $\gamma$ is a weight
in $\wt{\mc M}$, then
$\gamma\in{^i{\la}}^\theta-\sum_{\alpha\in{\Pi(\wt{\mc{B}}_{\bf
b})}}\Z_+\alpha$, for some $i$.
\end{itemize}
The morphisms in $\wt{\OO}_{\bf b}$ are all (not necessarily even)
homomorphisms of $\wt{\G}$-modules.
\end{definition}

Let $\mc M\in\wt{\OO}_{\bf b}$ so that $\mc M=\bigoplus_{\gamma\in
\wt{X}}\mc M_\gamma$. For $\vartheta\in\Z_2$  let
$$
\wt{X}_{\vartheta}=\Big\{\gamma\in\wt{X}\vert \sum_{i=m+1}^n\langle
\gamma,e_{ii} \rangle
+\sum_{r\in\hf+\Z_+}\langle\gamma,E_{\ul{r}\,\ul{r}}\rangle\equiv\vartheta\Big\}.
$$
We define as in \S\ref{repn:prep:prelim}
\begin{align}\label{Z2:gradation1}
\mc M'=\mc M'_{\bar 0} \oplus \mc M'_{\bar 1}, \quad \text{ where }
\mc M'_{\vartheta}:=\bigoplus_{\gamma\in \wt{X}_\vartheta}\mc
M_\gamma\;\; (\vartheta \in \Z_2).
\end{align}
Then $\mc M'\in\wt{\OO}_{\bf b}$, and $\mc M\cong \mc
M'$ in $\wt{\OO}_{\bf b}$ as in Section \ref{sec:BGGmn}. As argued in \S\ref{repn:prep:prelim}, the full subcategory
$\wt{\OO}_{{\bf b},\bar{0}}$ of $\wt{\OO}_{{\bf b}}$ consisting of
objects with $\Z_2$-gradation given by \eqref{Z2:gradation1} is an
abelian category. Since the categories $\wt{\OO}_{{\bf b},\bar{0}}$
and $\wt{\OO}_{\bf b}$ have isomorphic skeleton categories, we conclude that
$\wt{\OO}_{\bf b}$ is an abelian category.

The abelian categories $\breve{\OO}_{\bf b}$  of
$\breve{\G}$-modules and $\OO_{\bf b}$ of $\G$-modules are defined
in a similar fashion.

The modules $\wt{\mc{M}}_{\bf b}(\la^\theta)$ and $\wt{\mc{L}}_{\bf
b}(\la^\theta)$, $\breve{\mc{M}}_{\bf b}(\la^\natural)$ and
$\breve{\mc{L}}_{\bf b}(\la^\natural)$, ${\mc{M}}_{\bf b}(\la)$ and
${\mc{L}}_{\bf b}(\la)$ lie in the categories $\wt{\OO}_{\bf b}$,
$\breve{\OO}_{\bf b}$, $\OO_{\bf b}$, respectively, for $\la\in
X^+$. As in Definition \ref{def:tilt} we can also define, for each
$\la\in X^+$, tilting modules $\wt{\mc{T}}_{\bf b}(\la^\theta)$,
$\breve{\mc{T}}_{\bf b}(\la^\natural)$, and $\mc{T}_{\bf b}(\la)$ in
the categories $\wt{\OO}_{\bf b}$, $\breve{\OO}_{\bf b}$, and
$\OO_{\bf b}$, respectively. We can now adapt the arguments in \cite[Section 5]{CLW2} to show that tilting modules
exist and are unique in these respective categories. In contrast to
the more standard setups in \cite{So2, Br2}, the Lie superalgebras
under considerations here are infinite-rank.

For $\wt{\mc M}\in\wt{\OO}_{\bf b}$ we denote the $n$th
$\wt{\mf{u}}_{\bf b}^-$-homology group with coefficients in $\wt{\mc M}$ by
$H_n\left(\wt{\mf{u}}_{\bf b}^-;\wt{\mc M}\right)$. For $\breve{\mc
M}\in\breve{\OO}_{\bf b}$ and $\mc{M}\in\OO_{\bf b}$ the notations
$H_n\left({\breve{\mf{u}}}_{\bf b}^-;\breve{\mc M}\right)$ and
$H_n\left({\mf{u}}_{\bf b}^-;{\mc M}\right)$ stand for similar homology
groups. We introduce the following, for $\la,\mu \in X^+$:
\begin{align}  \label{eq:KLV}
\wt{\mf l}^{\bf b}_{\mu^\theta\la^\theta}(q)
 &:=\sum_{i=0}^\infty
\dim\text{Hom}_{\wt{\mf{k}}}
\left(\wt{L}^0(\mu^\theta),H_i\left(\wt{\mf{u}}_{\bf b}^-; \wt{\mc
L}_{\bf b}(\la^\theta)\right)\right) (-q^{-1})^i,
 \notag \\
\breve{\mf l}^{\bf b}_{\mu^\natural\la^\natural}(q)
 &:=\sum_{i=0}^\infty \dim\text{Hom}_{\breve{\mf{k}}}
\left(\breve{L}^0(\mu^\natural),H_i\left(\breve{\mf{u}}_{\bf b}^-;
\breve{\mc L}_{\bf b}(\la^\natural)\right)\right) (-q^{-1})^i,
 \\
{\mf l}^{\bf b}_{\mu\la}(q)
 &:=\sum_{i=0}^\infty
\dim\text{Hom}_{{\mf{k}}} \left({L}^0(\mu),H_i\left({\mf{u}}_{\bf
b}^-;{\mc L}_{\bf b}(\la)\right)\right) (-q^{-1})^i.
 \notag
\end{align}
They turn out to be polynomials, and will be called {\em
Kazhdan-Lusztig-Vogan (KLV) polynomials}.

\subsection{Equivalence of categories}\label{Tfunctors}

We may regard $X\subseteq\wt{X}$ and $\breve{X}\subseteq\wt{X}$, by
definitions given in \eqref{eq:X}. Given a semisimple
$\wt{\h}$-module $\wt{\mc M}$ such that $\wt{\mc
M}=\bigoplus_{\gamma\in\wt{X}}\wt{\mc M}_\gamma$, we define
\begin{align*}
T(\wt{\mc M}):= \bigoplus_{\gamma\in X}\wt{\mc M}_\gamma,\qquad
\hbox{and}\qquad \breve{T}(\wt{\mc M}):=
\bigoplus_{\gamma\in\breve{X}}\wt{\mc M}_\gamma.
\end{align*}
Note that $T(\wt{\mc M})$ is an $\h$-submodule of $\wt{\mc M}$
(regarded as an $\h$-module), and $\breve{T}(\wt{\mc M})$ is an
$\breve{\h}$-submodule of $\wt{\mc M}$ (regarded as an
$\breve{\h}$-module). Also if $\wt{\mc M}$ is a
$\wt{\mf{k}}$-module, then $T(\wt{\mc M})$ is a ${\mf{k}}$-submodule
of $\wt{\mc M}$ (regarded as a ${\mf{k}}$-module), and
$\breve{T}(\wt{\mc M})$ is a $\breve{\mf{k}}$-submodule of $\wt{\mc
M}$ (regarded as a $\breve{\mf{k}}$-module). Furthermore if $\wt{\mc
M}\in\wt{\OO}_{\bf b}$, then $T(\wt{\mc M})$ is a $\G$-submodule of
$\wt{\mc M}$ (regarded as a $\G$-module), and $\breve{T}(\wt{\mc
M})$ is a $\breve{\G}$-submodule of $\wt{\mc M}$ (regarded as a
$\breve{\G}$-module).

If $\wt{f}:\wt{\mc M}\longrightarrow \wt{\mc N}$ is an
$\wt{\h}$-homomorphism, we let
\begin{eqnarray*}
\CD T[\wt{f}] : T(\wt{\mc M}) @>>>T(\wt{\mc N}) \qquad
\hbox{and}\qquad \breve{T}[\wt{f}] :  \breve{T}(\wt{\mc M})
@>>>\breve{T}(\wt{\mc N})
 \endCD
\end{eqnarray*}
be the corresponding restriction maps. Then $T[\wt{f}]$
(respectively, $\breve{T}[\wt{f}]$) is an $\h$- (respectively,
$\breve{\h}$-) homomorphism. If $\wt{f}$ is also a homomorphism of
$\wt{\mf{k}}$-modules, then $T[\wt{f}]$ (respectively,
$\breve{T}[\wt{f}]$) is a ${\mf{k}}$- (respectively,
$\breve{\mf{k}}$-) homomorphism. Finally, if $\wt{f}$ is also a
homomorphism of $\wt{\G}$-modules, then $T[\wt{f}]$ (respectively,
$\breve{T}[\wt{f}]$) is a $\G$- (respectively, $\breve{\G}$-)
homomorphism.

Recall the notations $\la^\theta$, $\la^\natural$ from
\eqref{eq:latheta}. Following the line of arguments of \cite{CL} we can show that $T$ and $\breve{T}$ define exact functors
from $\wt{\OO}_{\bf b}$ to $\OO_{\bf b}$ and from $\wt{\OO}_{\bf b}$
to $\breve{\OO}_{\bf b}$, respectively; moreover, we establish the
following theorem similarly.

\begin{thm} [Super Duality]
\label{thm:SD}
\begin{enumerate}
\item
$T: \wt{\OO}_{\bf b}\to\OO_{\bf b}$ and $\breve{T}:\wt{\OO}_{\bf
b}\to\breve{\OO}_{\bf b}$ are equivalences of categories.
Consequently, the categories $\OO_{\bf b}$ and $\breve{\OO}_{\bf b}$
are equivalent.

\item
For $\mc Y=\mc M,\mc L,\mc T$ and $\la\in X^+$, we have
 $
T(\wt{\mc Y}_{\bf b}(\la^\theta))=\mc{Y}_{\bf b}(\la),
\breve{T}(\wt{\mc Y}_{\bf b}(\la^\theta))=\breve{\mc{Y}}_{\bf
b}(\la^\natural).
 $
Consequently,  for $\mc X=\mc L,\mc T$, we have
\begin{align*}
\rm{ch}\mc{X}_{\bf b}(\la)
 =\sum_{\mu\in
X^+} & a_{\mu\la}\rm{ch}\mc{M}_{\bf b}(\mu)\quad
 \Longleftrightarrow
\quad \rm{ch}\wt{\mc{X}}_{\bf b}(\la^\theta) =\sum_{\mu\in
X^+}a_{\mu\la}\rm{ch}\wt{\mc{M}}_{\bf b}(\mu^\theta)
 \\
&\Longleftrightarrow\quad \rm{ch}\breve{\mc{X}}_{\bf
b}(\la^\natural)=\sum_{\mu\in X^+}
a_{\mu\la}\rm{ch}\breve{\mc{M}}_{\bf b}(\mu^\natural),
 \qquad \text{ for } a_{\mu\la}\in\Z.
\end{align*}

\item
For $\la,\mu\in X^+$ the functors $T$ and $\breve{T}$ induces
natural isomorphisms
\begin{align*}
\text{Hom}_{\wt{\mf{k}}}\left(\wt{L}^0(\mu^\theta),H_n\left(\wt{\mf{u}}_{\bf
b}^-;\wt{\mc L}_{\bf b}(\la^\theta)\right)\right) &\cong
\text{Hom}_{\breve{\mf{k}}}\left(\breve{L}^0(\mu^\natural),
H_n\left(\breve{\mf{u}}_{\bf b}^-;\breve{\mc L}_{\bf b}(\la^\natural)\right)\right)
 \\
& \cong \text{Hom}_{{\mf{k}}}\left({L}^0(\mu),H_n\left(\mf{u}_{\bf
b}^-;{\mc L}_{\bf b}(\la)\right)\right).
\end{align*}
Consequently, the corresponding Kazhdan-Lusztig-Vogan polynomials
are identical, that is, $ \wt{\mf l}^{\bf
b}_{\mu^\theta\la^\theta}(q)
  =
\breve{\mf l}^{\bf b}_{\mu^\natural\la^\natural}(q)
 =
{\mf l}^{\bf b}_{\mu\la}(q).
 $
\end{enumerate}
\end{thm}

\begin{rem}
Theorem \ref{thm:SD} affords further parabolic variants which allow
more general even Levi subalgebras on the
\makebox(36,0){$\oval(32,14)$}\makebox(-30,8){$\mf{T}^{\bf b}$} side
(see the Dynkin diagrams in \S\ref{Sec:GGG}). A novel viewpoint of
Theorem~ \ref{thm:SD}, in contrast to \cite{CL,CLW1}, is that super
duality holds also when the {\em head} Dynkin diagram is that of a
Lie {\em super}algebra.
\end{rem}

\begin{rem}
Using the same argument as \cite[Proposition 3.11]{CL}, we can show
that $\breve{\mc M}_{\bf b}(\la^\natural)$, for $\la\in X^+$, has a
finite composition series with composition factors with highest
weights lying in $\breve{X}^+$. It follows therefore that
$\breve{\OO}_{\bf b}$ is the category of $\breve{\G}$-modules that
have finite composition series and that, as
$\breve{\mf{k}}$-modules, decompose into direct sums of irreducible
$\breve{\mf{k}}$-modules with highest weights lying in
$\breve{X}^+$. Thus, the categories $\breve{\OO}_{\bf b}$ are
independent of the choices of the $0^m1^n$-sequences ${\bf b}$.
Similarly the categories $\wt{\OO}_{\bf b}$ (and $\OO_{\bf b}$,
respectively) are all independent of the choices of ${\bf b}$.
\end{rem}

\subsection{BGG categories of finite rank}
\label{sec:finiterank}

For $k\in\N$, we let $\G^k$ and ${\SG}^k$ be the respective
finite-dimensional general linear Lie superalgebras with simple
roots as follows:
\begin{align*}
&\{\ep^{\bf b}_1-\ep^{\bf b}_2,\ldots,\ep^{\bf b}_{m+n-1}-\ep^{\bf b}_{m+n},
\ep^{\bf b}_{m+1}-\delta_1, \delta_1-\delta_2,\ldots,\delta_{k-1}-\delta_k\},
 \\
&\{\ep^{\bf b}_1-\ep^{\bf b}_2,\ldots,\ep^{\bf b}_{m+n-1} -\ep^{\bf
b}_{m+n},\ep^{\bf b}_{m+1}-\delta_{1/2},
\delta_{1/2}-\delta_{3/2},\ldots,\delta_{k-3/2}-\delta_{k-1/2}\}.
\end{align*}
Then $\G^k$ and ${\SG}^k$ may be regarded as subalgebras of $\G$ and
${\SG}$. We denote the standard Borel subalgebras corresponding to
these simple systems by $\mc{B}_{\bf b}^k$ and $\breve{\mc{B}}_{\bf
b}^k$, respectively, and furthermore set $\h^k=\h\cap\G^k$ and
$\breve{\h}^k=\breve{\h}\cap{\SG}^k$. Moreover, we have natural
inclusions $\G^k \subseteq \G^{k+1}$ and ${\SG}^k \subseteq
\breve{\G}^{k+1}$, with $\G =\bigcup_k {\G}^k$ and $\breve{\G}
=\bigcup_k \breve{\G}^k$.

Set
\begin{align}  \label{eq:Xk}
X^k=X\cap(\h^{k})^*, \qquad \breve{X}^k=\breve{X}\cap(\breve{\h}^{k})^*.
\end{align}
Also define
\begin{align}  \label{eq:Xk+}
\begin{split}
X^{k,+} &=\Big\{\la =\sum_{j=1}^{m+n} \la_j \ep^{\bf b}_j +\sum_{i=1}^k \mu_i \delta_i \in X^k \mid \mu_1\ge \ldots \ge\mu_k\Big\},
  \\
\breve X^{k,+} &=\Big\{\sum_{j=1}^{m+n} \la_j \ep^{\bf b}_j +\sum_{i=1}^k \nu_i \delta_{i-\hf} \in \breve X^k \mid \nu_1\ge \ldots \ge\nu_k\Big\}.
\end{split}
\end{align}
We shall identify a weight $\la \in X^{k,+}$ as
the tuple $\la=(\la_1,\ldots,\la_{m+n}; \mu_1, \ldots, \mu_k)$.
Given $\la=(\la_1,\ldots,\la_{m+n};{^+\la})\in X^+$ with
${^+\la}_j=0$ for $j>k$, we may regard $\la$ as a weight in
$X^{k,+}$ in a natural way.
Similarly, for $\la\in X^+$  with ${^+\la}'_j=0$ for
$j>k$, we regard $\la^\natural$ as a weight in $\breve X^{k,+}$.

The Levi subalgebra, parabolic subalgebra, and the nilradical of the
finite-rank Lie superalgebra $\G^k$ are
$$
\mf{k}^k =\mf{k} \cap \G^k,
 \qquad
\mf{p}_{\bf b}^k =\mf{p}_{\bf b}\cap\G^k,\qquad \mf{u}_{\bf b}^k
=\mf{u}_{\bf b}\cap\G^k,
$$
respectively. We denote the parabolic Verma, irreducible, and
tilting $\G^k$-modules by $\mc{M}_{\bf b}^k(\la)$, $\mc{L}_{\bf
b}^k(\la)$, and $\mc{T}_{\bf b}^k(\la)$, for $\la\in X^{k,+}$. The
corresponding parabolic BGG category of $\G^k$-modules is denoted by
$\OO^{\ul k}_{\bf b}$, which is defined similarly as in
Definition~\ref{def:O:bfa:inf}, now with $\h$, $\mf{k}$, et cetera
therein replaced by $\h^k$, $\mf{k}^k$, et cetera.

The statements in the previous paragraph have obvious counterparts
for the Lie superalgebra ${\SG}^k$ as well. We introduce the
self-explanatory notations $\breve{\mc{M}}_{\bf b}^k(\xi)$,
$\breve{\mc L}_{\bf b}^k(\xi)$, $\breve{\mc T}_{\bf
b}^k(\xi)$, $\breve{\OO}_{\bf b}^{\ul k}$,
$\breve{\mf{k}}^k$, $\breve{\mf p}^k_{\bf b}$, where
$\xi\in \breve{X}^{k,+}$.

In an entirely analogous manner as in the definitions of ${\mf l}^{\bf b}_{\mu\la}(q)$ and $\breve{\mf l}^{\bf b}_{\mu^\natural\la^\natural}(q)$
in \eqref{eq:KLV}, we can define the {\em Kazhdan-Lusztig-Vogan
(KLV) polynomials} ${\mf l}^{\bf b,0}_{\mu\la}(q),$ for $\la,\mu
\in X^{k,+}$, and $\breve{\mf l}^{\bf
b,1}_{\xi\eta}(q)$,  for $\xi,\eta
\in \breve X^{k,+}$, in the categories $\OO_{\bf b}^{\ul k}$ and
$\breve{\OO}_{\bf b}^{\ul k}$, respectively.

\subsection{Truncation functors}
\label{subsec:trunc:mod}

Let $\la\in X^+$, and write
$\la=(\la_1,\la_2,\ldots,\la_{m+n};{^+\la})$, where
${^+\la}\in\mc{P}$. Recall the parabolic Verma $\G$-modules
$\mc{M}_{\bf b}(\la)$, $\breve{\mc{M}}_{\bf b}(\la)$, the
irreducible modules $\mc{L}_{\bf b}(\la)$, $\breve{\mc{L}}_{\bf
b}(\la)$, and the tilting modules  $\mc{T}_{\bf b}(\la)$,
$\breve{\mc{T}}_{\bf b}(\la)$ in the categories $\OO_{{\bf b}}$ and
$\breve{\OO}_{{\bf b}}$, respectively. Let $\mc{M}\in\OO_{{\bf b}}$
and $\breve{\mc M}\in\breve{\OO}_{{\bf b}}$. Then we have the weight
space decompositions
\begin{align*}
\mc{M}=\bigoplus_{\mu\in X} \mc{M}_{\mu},\qquad
\breve{\mc{M}}=\bigoplus_{\mu\in\breve{X}} \breve{\mc{M}}_{\mu}.
\end{align*}

We define an exact functor $\mf{tr}:\OO_{{\bf b}}\rightarrow \OO^{\ul
k}_{{\bf b}}$ by
\begin{align*}
\mf{tr}(\mc M):=\bigoplus_\mu \{\mc{M}_\mu \mid
(\mu,\delta_j)=0,\forall j\ge k+1\text{ and }j\in\N\}.
\end{align*}
Similarly, we define an exact functor
$\breve{\mf{tr}}:\breve{\OO}_{{\bf b}}\rightarrow \breve{\OO}^{\ul
k}_{{\bf b}}$ by
\begin{align*}
\breve{\mf{tr}} (\breve{\mc M}):=\bigoplus_\mu
\big\{\breve{\mc{M}}_\mu \mid (\mu,\delta_r)=0,\forall r> k\text{
and }r\in\hf+\Z_+ \big\}.
\end{align*}

We have the following.

\begin{prop} \label{prop:trunc:ML}
The functors $\mf{tr}:\OO_{{\bf b}}\rightarrow \OO^{\ul k}_{{\bf
b}}$ and $\breve{\mf{tr}}:\breve{\OO}_{{\bf b}}\rightarrow
\breve{\OO}^{\ul k}_{{\bf b}}$ satisfy the following: for
$\mc{Y}=\mc{M},\mc{L},\mc{T}$ and
$\la=(\la_1,\ldots,\la_{m+n};{^+\la})\in X^+$,
\begin{align*}
&\mf{tr}\left(\mc Y_{\bf b}(\la)\right)=
\begin{cases}
\mc Y^k_{{\bf b}}(\la),&\text{ if } \ell({^+\la})\le k,\\
0,&\text{ otherwise.}
\end{cases}\\
&\breve{\mf{tr}}\left(\breve{\mc Y}_{\bf b}(\la^\natural)\right)=
\begin{cases}
\breve{\mc Y}^k_{{\bf b}}(\la^\natural),
 &\text{ if } \ell({^+\la'})\le k,\\
0,&\text{ otherwise.}
\end{cases}
\end{align*}
Moreover, we have ${\mf l}^{\bf b}_{\mu\la}(q) ={\mf l}^{\bf
b,0}_{\mu\la}(q)$ for $\ell({^+\la})\le k$ and $\ell({^+\mu})\le k$, and $ \breve{\mf l}^{\bf
b}_{\mu^\natural\la^\natural}(q) =\breve{\mf l}^{\bf
b,1}_{\mu^\natural\la^\natural}(q)$ for $\ell({^+\la'})\le k$ and $\ell({^+\mu'})\le k.$
\end{prop}

\begin{proof}
For $\mc{Y}=\mc{M}$, this is easy. For $\mc{Y}=\mc{L}$ the argument
in \cite[Lemma 3.5, Corollary 3.6]{CWZ}, or, with greater details in
\cite[Proposition 6.7]{CWbook}, can be adapted easily to our
settings here. For $\mc{Y}=\mc{T}$, it follows by the same type of
arguments as in \cite[Proposition 3.12]{CW}.

The coincidence of KLV polynomials under the truncation functors is
an immediate consequence of the property that the truncation
functors commute with the differentials of the complexes for the
respective homology groups (cf.~\cite[Theorem 6.31]{CWbook}).
\end{proof}

\begin{rem}  \label{rem:notsurj}
Due to the conditions in the definitions \eqref{eq:X+} of $X^+$ and
\eqref{eq:Xk+} of $X^{k,+}$, the $\G^k$-modules $\mc Y_{\bf
b}^k(\la)$, for $\mc{Y}=\mc{M},\mc{L},\mc{T}$, appear as images of
$\mf{tr}$ if and only if $\la=(\la_1,\ldots,\la_{m+n}; \mu_1,
\ldots, \mu_k) \in X^{k,+}$ satisfies the additional condition that
$\mu_k \ge 0$. Similar remarks apply to the images of
$\breve{\mf{tr}}$.
\end{rem}

\section{Proof of Brundan-Kazhdan-Lusztig conjecture}
\label{sec:BKL}

In this section, we prove the Brundan-Kazhdan-Lusztig (BKL) conjecture
for the BGG category $\OO$ of $\glmn$-modules which is formulated in
terms of canonical and dual canonical bases on a Fock space $\mathbb
T^{\bf b}$, associated with a $0^m1^n$-sequence $\bf b$. Our proof is built
on a Fock space reformulation of classical Kazhdan-Lusztig theory
for type $A$ Lie algebras, the super duality, and a comparison of BKL conjecture
for adjacent Borel subalgebras.

\subsection{BKL conjecture}

Let ${\bf b}$ be a $0^m1^n$-sequence. Recall from
Section~\ref{sec:repn:prep} that the BGG category $\Omn=\Omn_{\bf
b}$ of $\gl(m|n)$-modules contains the $\bf b$-Verma module $M_{\bf
b}(\la)$, the $\bf b$-highest weight irreducible modules $L_{\bf
b}(\la)$, and the $\bf b$-tilting modules $T_{\bf b}(\la)$ for
$\la\in \wtl$. Recall also that $\mc{O}^{m|n,\Delta}_{\bf b}$ denotes the
full subcategory of $\Omn_{\bf b}$ consisting of objects that have
finite ${\bf b}$-Verma flags, and let $\big[\mc{O}^{m|n,\Delta}_{\bf
b}\big]$ denote its Grothendieck group.

Recall furthermore the Fock space $\mathbb T^{\bf b}$ and its $B$-completion
$\widehat{\mathbb T}^{\bf b}$ with respect to the Bruhat ordering
$\preceq_{\bf b}$ from Definition~\ref{def:completion:T}. Starting
with a $\Z[q,q^{-1}]$-lattice spanned by the standard monomial basis
for the $\Q(q)$-vector space $\mathbb T^{\bf b}$, we define by a
base change to $\Z$ the specialization at $q=1$ of $\mathbb T^{\bf
b}$, denoted by $\mathbb T^{\bf b}_\Z$. The $B$-completion
$\widehat{\mathbb T}^{\bf b}_\Z$ is defined as usual. For a
standard, canonical, or dual canonical basis element $u$ in $\mathbb
T^{\bf b}\wotimes\wedge^{\infty}\mathbb V$, we shall denote by
$u(1)$ the corresponding element in the specialization at $q=1$.
Similar remarks on specialization at $q=1$ and similar notations
apply below to other variants of Fock spaces.

Recall the bijection $\wtl \rightarrow \Z^{m+n}$ given by $\la \mapsto f^{\bf
b}_\la$ from \eqref{eq:bijXZ}. We have a natural $\Z$-linear
isomorphism $\psi_{\bf b}: \big[\mc{O}^{m|n,\Delta}_{\bf b}\big]
\longrightarrow \mathbb T^{\bf b}_\Z$ given by
$%\begin{align*}
[M_{\bf b}(\la)]  \mapsto M^{\bf b}_{f^{\bf b}_\la}(1).
$ %\end{align*}
We define a completion $\big{[}\big{[}{\mc{O}}^{m|n,\Delta}_{\bf
b}\big{]}\big{]}$ \ so that $\psi_{\bf b}$ extends to a $\Z$-linear
isomorphism between the two completions
$$
\psi_{\bf b}:
\big[\big[\mc{O}^{m|n,\Delta}_{\bf b}\big]\big] \longrightarrow
\widehat{\mathbb T}^{\bf b}_\Z, \qquad [M_{\bf b}(\la)]  \mapsto M^{\bf b}_{f^{\bf b}_\la}(1).
$$
We note that $[L_{\bf b}(\la)]\in
\big{[}\big{[}{\mc{O}}^{m|n,\Delta}_{\bf b}\big{]}\big{]}$, though
$L_{\bf b}(\la) \not \in {\mc{O}}^{m|n,\Delta}_{\bf b}$ in general.

We now formulate Brundan's Kazhdan-Lusztig type conjecture for
$\Omn_{\bf b}$, for an arbitrary $0^m1^n$-sequence $\bf b$. Recall
the BKL polynomials $\ell_{f^{\bf b}_\mu f^{\bf b}_\la}(q)$ and
$t_{f^{\bf b}_\mu f^{\bf b}_\la}(q)$ from
Proposition~\ref{prop:existence:can1}.

\begin{conjecture} [BKL conjecture]
 \label{BKL:conj:tilt}
Let ${\bf b}$ be an arbitrary $0^m1^n$-sequence.
\begin{enumerate}
\item
We have $\psi_{\bf b}([L_{\bf b}(\la)]) =L^{\bf b}_{f^{\bf
b}_\la}(1)$, for all $\la\in \wtl$. Equivalently, we have
\begin{align*}
[L_{\bf b}(\la)] =\sum_{\mu}\ell_{f^{\bf b}_\mu f^{\bf
b}_\la}(1)[M_{\bf b}(\mu)].
\end{align*}

\item

We have $\psi_{\bf b}([T_{\bf b}(\la)]) =T^{\bf b}_{f^{\bf
b}_\la}(1)$, for all $\la\in \wtl$. Equivalently, we have
\begin{align*}
[T_{\bf b}(\la)]=\sum_{\mu}t_{f^{\bf b}_\mu f^{\bf b}_\la}(1)[M_{\bf
b}(\mu)].
\end{align*}
\end{enumerate}
\end{conjecture}

\begin{rem}
Conjecture ~\ref{BKL:conj:tilt} for the standard ${0^m1^n}$-sequence
${\bf b}_{\text{st}}=({0}^m,{1}^n)$ is precisely \cite[Conjecture
4.32]{Br}, and the variants of Brundan's conjecture for general $\bf
b$ have been expected (cf. Kujawa's thesis \cite{Ku}), though the
completions of various Fock spaces and their canonical bases were
not formulated in {\em loc.~cit.}. Kujawa provided supporting
evidence for the BKL conjecture by showing the irreducible modules
in $\Omn$ form a crystal basis compatible with the one coming from
$\mathbb T^{\bf b}$. When $n=0$ or $m=0$, the BKL conjecture reduces
to a reformulation of classical Kazhdan-Lusztig conjecture (see
Theorem~\ref{thm:classical:KL} for $k=0$ below).
\end{rem}

The remainder of this paper is devoted to a proof of this
conjecture. We will follow closely the strategy of proof outlined in \S\ref{stra:of:proof}.

\subsection{Bijections}

Let $k\in\N\cup\{\infty\}$. Recall $f^{\bf b}_\la\in\Z^{m+n}$ from
\eqref{la:to:fla}, respectively. Also recall $\Z^k_+$ and $\Z^k_-$ from
\eqref{eq:Zk+} and \eqref{eq:Zk-}. The following maps
\begin{align} \label{eq:2bij}
 \begin{split}
X^{k,+} \longrightarrow \Z^{m+n}\times\Z^k_+, \qquad \la\mapsto
f^{{\bf b}0}_\la,
  \\
\breve{X}^{k,+} \longrightarrow \Z^{m+n}\times\Z^k_-, \qquad
\la\mapsto f^{{\bf b}1}_\la,
 \end{split}
\end{align}
are bijections, where $X^{\infty,+}$ and $\breve{X}^{\infty,+}$ are
understood to be $X^+$ and $\breve{X}^{+}$ in \eqref{eq:X+} and
\eqref{eq:X+2}, respectively. Here $f_{\la}^{{\bf b}0} \in
\Z^{m+n}\times\Z^k_+$ and $f_{\la}^{{\bf b}1} \in
\Z^{m+n}\times\Z^k_-$, for $\la=(\la_1,\ldots,\la_{m+n};{^+\la})$
with ${^+\la}=({^+\la}_1,\ldots,{^+\la}_k)$, are defined by setting
\begin{eqnarray*}
f_{\la}^{{\bf b}0}(i)
 =&f_{\la}^{\bf b}(i), & \text{ if }i\in [m+n],
 \\
f_{\la}^{{\bf b}0}(\ul{i})
 =&{^+\la}_i+1-i,& \text{ if }1\le i\le k,
 \\
f_{\la}^{{\bf b}1}(i)
 =& f_{\la}^{\bf b}(i), & \text{ if }i\in [m+n],
 \\
f_{\la}^{{\bf b}1}(\ul{i})
 =&i -{^+\la}_i, &\text{ if }1\le i\le k.
\end{eqnarray*}
The normalization $\rho_{\bf b}$ in \eqref{eq:2rho} used in the definition
of $f^{\bf b}_\la\in\Z^{m+n}$ in \eqref{la:to:fla} is compatible
with the above definitions in the sense that $f_{\la}^{{\bf b}0}$
and $f_{\la}^{{\bf b}1}$ correspond indeed to $\la+\rho$ for
suitably normalized Weyl vector $\rho$ for $\gl(m+k|n)$ and
$\gl(m|n+k)$, respectively.

\subsection{Classical KL theory}

In this subsection we consider the case when $n=0$ so that ${\bf
b}=(0^m)$.

In this case, $\G$ defined in \S\ref{Sec:GGG} and $\G^k=\gl(m+k)$ in
\S\ref{sec:finiterank} are Lie algebras. For $k\in\N
\cup\{\infty\}$, recall the parabolic BGG category $\OO^{\ul
k}_{{\bf b}}$ of $\G^k$-modules defined in \S\ref{SD:para:cat} and \S\ref{sec:finiterank}, and
let $\OO^{\ul k,\Delta}_{{\bf b}}$ denote the full subcategory of
$\OO^{\ul k}_{{\bf b}}$ consisting of objects that have finite
parabolic ${\bf b}$-Verma flags  (here and below it is understood
that $\G^\infty=\G, \OO^{\ul \infty}_{{\bf b}}=\OO_{{\bf b}}$,
$\OO^{\ul \infty,\Delta}_{{\bf b}}=\OO^\Delta_{{\bf b}}$ and so on).
Let $\big[\OO_{{\bf b}}^{\ul k,\Delta}\big]$ denote its Grothendieck
group.

Note that ${\mathbb T}^{\bf b} =\VV^{\otimes m}$ for ${\bf b}=(0^m)$.
Thanks to the bijection \eqref{eq:2bij} given by $\la\mapsto f^{{\bf
b}0}_\la$, the $\Z$-linear map
\begin{align*}
\Psi: \big[\OO^{\ul k,\Delta}_{{\bf b}}\big]\longrightarrow
{\mathbb T}^{\bf b}_\Z\otimes\wedge^{k}\mathbb V_\Z,
 \qquad
[\mc{M}^k_{\bf b}(\la)]\mapsto M^{{\bf b},0}_{f^{{\bf
b}0}_{\la}}(1),
\end{align*}
is an isomorphism, where ${\mathbb T}^{\bf
b}_\Z\otimes\wedge^{k}\mathbb V_\Z$ denotes the $q=1$ specialization.
We define the completion $\big{[}\big{[}{\OO}^{\ul k,\Delta}_{{\bf
b}}\big{]}\big{]}$ of $\big[\OO^{\ul k,\Delta}_{{\bf b}}\big]$ in the obvious way so that $\Psi$ extends to a $\Z$-linear isomorphism
 $\Psi:
\big{[}\big{[}\OO^{\ul k,\Delta}_{{\bf b}}\big{]}\big{]}
\longrightarrow {\mathbb T}^{\bf b}_\Z\wotimes\wedge^k\mathbb
V_\Z$.

Then, by Vogan's homological interpretation of the Kazhdan-Lusztig
polynomials \cite[Conjecture 3.4]{V} and \cite[Theorem 3.11.4]{BGS},
Theorem~\ref{thm:classical:KL} below (for $k$ finite) is a
well-known Fock space reformulation of the Kazhdan-Lusztig
conjectures for type $A$ Lie algebras \cite{KL, KL2} (proved in
\cite{BB, BK}, and the equivalent tilting module version in \cite{So2}).  Such a reformulation can be found in \cite{Br, Br4}
and \cite[Theorem~4.14]{CW} (also see the proof of
\cite[Theorem~5.4]{CWZ}). The case $k=\infty$ follows from the cases
for finite $k$ by Proposition~\ref{prop:trunc:fock}, once the
existence of tilting modules is established as in \cite[Theorem
4.16]{CW}. Recall the KLV polynomials $\mf l^{{\bf
b},0}_{\mu\la}(q)$ from \S\ref{sec:finiterank} and the BKL
polynomials $\ell^{{\bf b},0}_{f^{{\bf b}0}_{\mu} f^{{\bf
b}0}_{\la}}(q)$ from Proposition~\ref{prop:CBdcb}. Here we recall
our assumption that ${\bf b}=(0^m)$ in the theorem below, though
eventually it turns out to be valid for a general $0^m1^n$-sequence
${\bf b}$, and our formulation in this subsection makes sense for a
general $\bf b$.

\begin{thm}\label{thm:classical:KL}
Let $k\in\Z_+ \cup\{\infty\}$. Then the isomorphism $\Psi:
\big{[}\big{[}\OO^{\ul k,\Delta}_{{\bf b}}\big{]}\big{]}
\longrightarrow {\mathbb T}^{\bf b}_\Z\wotimes\wedge^k\mathbb V_\Z$ satisfies
$$
\Psi([\mc{L}_{\bf b}^k(\la)])=L^{{\bf b},0}_{f^{{\bf b}0}_{\la}}(1),
 \quad
\Psi([\mc{T}_{\bf b}^k(\la)]) =T^{{\bf b},0}_{f^{{\bf
b}0}_{\la}}(1),
 \qquad
\text{ for }\la\in X^{k,+}.
$$
%Moreover, we have $\mf l^{{\bf b},0}_{\mu\la}(q) = \ell^{{\bf
%b},0}_{f^{{\bf b}0}_{\mu} f^{{\bf b}0}_{\la}}(q),$
% for $\la,\mu\in X^{k,+}.$
\end{thm}

\subsection{Super duality and BKL}

Let $\bf b$ be an arbitrary $0^m1^n$-sequence. Just as in the previous section we define
$\big{[}\big{[}{\breve{\OO}}^{\Delta}_{{\bf b}}\big{]}\big{]}$ a similar completion of the Grothendieck group of the full subcategory
$\breve{\OO}^{\Delta}_{{\bf b}}$ of $\breve{\OO}_{{\bf b}}$
consisting of objects with parabolic $\bf b$-Verma flags.
Thanks to the bijection \eqref{eq:2bij}, we now have a $\Z$-linear
isomorphism on the completions:
\begin{align*}
\breve{\Psi}: \big[\big[\breve\OO^{\Delta}_{{\bf
b}}\big]\big]\longrightarrow {\mathbb T}^{\bf
b}_\Z\widehat{\otimes}\wedge^{\infty}\WW_\Z,
 \qquad
[\breve{\mc{M}}_{\bf b}(\la)]\mapsto M^{{\bf b},1}_{f^{{\bf
b}1}_{\la}}(1),
\end{align*}
which is induced by the corresponding isomorphism
$\big[\breve\OO^{\Delta}_{{\bf b}}\big]\cong {\mathbb T}^{\bf b}_\Z
\otimes\wedge^{\infty}\WW_\Z$. The following is a consequence of
super duality (see Theorem~\ref{TwedgeV:isom:TwedgeW} and
Theorem~\ref{thm:SD}).

\begin{thm}  \label{th:KL:SD}
Let $\bf b$ be a $0^m1^n$-sequence. Assume the statement in
Theorem~\ref{thm:classical:KL} is valid for $\bf b$. Then, the
isomorphism $\breve{\Psi}: \big{[}\big{[}\breve{\OO}^{\Delta}_{{\bf
b}}\big{]}\big{]} \longrightarrow {\mathbb T}^{\bf
b}_\Z\wotimes\wedge^\infty\WW_\Z$ satisfies
$$
\breve{\Psi}([\breve{\mc{L}}_{\bf b}^k(\la)])=L^{{\bf b},1}_{f^{{\bf
b}1}_{\la}}(1),
 \quad
\breve{\Psi}([\breve{\mc{T}}_{\bf b}^k(\la)]) =T^{{\bf b},1}_{f^{{\bf
b}1}_{\la}}(1),
 \qquad
\text{ for }\la\in \breve X^{+}.
$$
%Consequently, we have $\breve{\mf l}^{\bf b}_{\mu\la}(q) =
%\ell^{{\bf b},1}_{f^{{\bf b}1}_{\mu} f^{{\bf b}1}_{\la}}(q),$ for
%$\la,\mu\in \breve X^{+}.$
\end{thm}

\begin{proof}
The bijections given in \eqref{eq:2bij} are compatible with the
bijection $\natural$ in \eqref{naturalbi}, so $\natural
(f_{\la}^{{\bf b}0}) =f_{\la^\natural}^{{\bf b}1}$. Combining the
isomorphism $\natural_{\bf b}:{\mathbb T}^{\bf
b}\wotimes\wedge^\infty\mathbb V\rightarrow {\mathbb T}^{\bf
b}\wotimes\wedge^\infty\mathbb W$ from
Theorem~\ref{TwedgeV:isom:TwedgeW}, super duality (SD) from
Theorem~\ref{thm:SD}, and the isomorphism $\Psi$ from
Theorem~\ref{thm:classical:KL}, we have the following commutative
diagram of isomorphisms on the left (the maps are defined in terms
of the standard objects on the right diagram):

\begin{eqnarray}  \label{CD:KL+SD}
\begin{CD}
\big{[}\big{[}\OO^{\Delta}_{{\bf b}}\big{]}\big{]}
 @>\Psi>>
{\mathbb T}^{\bf b}_\Z\wotimes\wedge^\infty \mathbb V_\Z  \\
 @V\text{SD}VV @V \natural_{\bf b} VV \\
\big[\big[\breve\OO^{\Delta}_{{\bf b}}\big]\big]
 @>\breve{\Psi}>>
{\mathbb T}^{\bf
b}_\Z\widehat{\otimes}\wedge^{\infty}\WW_\Z
  \end{CD}
\qquad \qquad
\begin{CD}
[\mc{M}_{\bf b}(\la)]
 @>\Psi>>
 M^{{\bf b},0}_{f^{{\bf b}0}_{\la}}(1)   \\
 @V\text{SD}VV @V \natural_{\bf b} VV \\
[\breve{\mc{M}}_{\bf b}(\la^\natural)]
  @>\breve{\Psi}>>
M^{{\bf b},1}_{f^{{\bf b}1}_{\la^\natural}}(1)  \end{CD}
\end{eqnarray}
Note that $\natural_{\bf b}$ preserves the (dual) canonical bases by
Theorem~\ref{TwedgeV:isom:TwedgeW}(3), super duality SD
preserves the simple and tilting modules by Theorem~\ref{thm:SD}(2), and
$\Psi$ preserves the $L$'s and $T$'s by assumption that the
statement in Theorem~\ref{thm:classical:KL} in valid. Therefore, we
have established the three sides (except the arrow on $\breve{\Psi}$) in
the following diagrams:
\begin{eqnarray*}
\begin{CD}
[\mc{L}_{\bf b}(\la)]
 @>\Psi>>
 L^{{\bf b},0}_{f^{{\bf b}0}_{\la}}(1)   \\
 @V\text{SD}VV @V \natural_{\bf b} VV \\
[\breve{\mc{L}}_{\bf b}(\la^\natural)]
  @>\breve{\Psi}>?>
L^{{\bf b},1}_{f^{{\bf b}1}_{\la^\natural}}(1)  \end{CD}
\qquad \qquad
\begin{CD}
[\mc{T}_{\bf b}(\la)]
 @>\Psi>>
 T^{{\bf b},0}_{f^{{\bf b}0}_{\la}}(1)   \\
 @V\text{SD}VV @V \natural_{\bf b} VV \\
[\breve{\mc{T}}_{\bf b}(\la^\natural)]
  @>\breve{\Psi}>?>
T^{{\bf b},1}_{f^{{\bf b}1}_{\la^\natural}}(1)  \end{CD}
\end{eqnarray*}
The arrows for the map $\breve{\Psi}$ in  $?$ in the above diagrams now follow from the
commutativity of \eqref{CD:KL+SD}.
%
%The identity $\breve{\mf l}^{\bf b}_{\mu\la}(q) = \ell^{{\bf
%b},1}_{f^{{\bf b}1}_{\mu} f^{{\bf b}1}_{\la}}(q)$ follows by the
%corresponding identities in Theorem~\ref{TwedgeV:isom:TwedgeW}(4),
%Theorem~\ref{thm:SD}(3), and Theorem~\ref{thm:classical:KL}.
\end{proof}

\subsection{Comparison of characters}
\label{char:comp:para:fin}

Let ${\bf b}$ be a fixed $0^m1^n$-sequence. For $k\in\N$ let $({\bf
b},{0^k})$ and $({\bf b},{1^k})$ denote the ${0^{m+k}1^n}$- and the
${0^m1^{n+k}}$-sequences obtained by adding $k$ ${0}$'s and $k$
${1}$'s to the end of the sequence ${\bf b}$, respectively. Recall
from \S\ref{sec:BGGmn} that $\OO^{m+k|n}_{({\bf b},{0^k})}$ and
$\OO^{m|n+k}_{({\bf b},{1^k})}$ are the full BGG categories of
$\gl(m+k|n)$- and $\gl(m|n+k)$-modules, respectively. In this
subsection, we compare the simple modules as well as tilting modules
in the parabolic category $\OO_{\bf b}^{\ul k}$ (and respectively,
$\breve\OO_{\bf b}^{\ul k}$) introduced in \S\ref{sec:finiterank}
with their counterparts in the full BGG category $\OO^{m+k|n}_{({\bf
b},{0^k})}$ (and respectively, $\OO^{m|n+k}_{({\bf b},{1^k})}$).
Also, note by \eqref{eq:Xk} and \eqref{eq:Xk+} that $X^{k,+}
\subseteq X^k$, $\breve X^{k,+} \subseteq \breve X^k$.

For $\la\in X^{k,+}$, we can
express $[{L}_{({\bf b},{0^k})}(\la)]$ in terms of Verma
modules:
\begin{align}\label{irred:char:verma}
[{L}_{({\bf b},{0^k})}(\la)]=\sum_{\mu\in X^k}a_{\mu\la}[M_{({\bf
b},{0^k})}(\mu)],\quad \text{ for } a_{\mu\la}\in\Z.
\end{align}
Since the simple objects in the parabolic category $\OO_{\bf b}^{\ul
k}$ defined in \S\ref{sec:finiterank} are the
$\G^k\equiv\gl(m+k|n)$-modules $\mc{L}_{\bf b}^k(\la)\equiv L_{({\bf
b},{0^k})}(\la)$, for $\la\in X^{k,+}$,
we can also express $[{L}_{({\bf b},{0^k})}(\la)]$
in terms of {\em parabolic} Verma modules in $\OO_{\bf b}^{\ul k}$:
\begin{align}\label{irred:char:para:verma}
[{L}_{({\bf b},{0^k})}(\la)]=\sum_{\nu\in
X^{k,+}}b_{\nu\la}[\mc{M}_{\bf b}^k(\nu)],\quad \text{ for }
b_{\nu\la}\in\Z.
\end{align}
Here we recall that $\mc{M}_{\bf b}^k(\nu)=\text{Ind}_{\mf p^k_{\bf
b}}^{\G^k}L^0(\nu)$, where $L^0(\nu)$ is the irreducible
$\left(\h^k+\gl(k)\right)$-module of highest weight $\nu \in X^{k,+}$. Applying the Weyl
character formula to $L^0(\nu)$ gives us
$ [\mc{M}_{\bf b}^k(\nu)] =\sum_{\sigma\in \mf{S}_k}
(-1)^{\ell(\sigma)}[M_{({\bf b},{0^k})}(\sigma\cdot \nu)],
 $
where as usual we have denoted the dot action of a Weyl group element $\sigma$ on $\nu$ by
$\sigma\cdot \nu=\sigma(\nu+\rho_{({\bf b},{0^k})})-\rho_{({\bf
b},{0^k})}$ and the Weyl vector $\rho_{({\bf b},{0^k})}$ is defined
in \S\ref{sec:super:Burhat}. Hence, \eqref{irred:char:para:verma}
can be rewritten as
\begin{align*}
[{L}_{({\bf b},{0^k})}(\la)]=\sum_{\nu\in X^{k,+}} \sum_{\sigma\in
\mf{S}_k} (-1)^{\ell(\sigma)}b_{\nu\la} [M_{({\bf
b},{0^k})}(\sigma\cdot \nu)], \quad \text{ for } \la\in X^{k,+}.
\end{align*}
Comparing this with \eqref{irred:char:verma} together with the
linear independence of the Verma characters show  that
$a_{\nu\la}=b_{\nu\la}$, for $\la,\nu\in X^{k,+}$. We summarize this
in the following.

\begin{prop}\label{prop:irred:char:para}
Let $\la\in X^{k,+}$. Retain the notation as in
\eqref{irred:char:verma}. Then
\begin{align*}
[{L}_{({\bf b},{0^k})}(\la)]=\sum_{\nu\in
X^{k,+}}a_{\nu\la}[\mc{M}_{\bf b}^k(\nu)].
\end{align*}
\end{prop}

Similarly, the simple $\gl(m|n+k)$-module $\breve{\mc{L}}_{\bf
b}^k(\xi)$, for $\xi \in \breve X^{k,+}$, in the parabolic category
$\breve\OO_{\bf b}^{\ul k}$ (cf. \S\ref{sec:finiterank}) can be
identified with ${L}_{({\bf b},{1^k})}(\xi)$ in the full BGG
category $\OO^{m|n+k}_{({\bf b},{1^k})}$. We write that
\begin{align}  \label{irred:verma2}
[{L}_{({\bf b},{1^k})}(\xi)] =\sum_{\mu\in
\breve{X}^k}\breve{a}_{\eta\xi}
 [M_{({\bf b},{1^k})}(\eta)], \quad \text{ for }
\breve{a}_{\eta\xi}\in\Z, \; \xi \in \breve X^{k,+}.
\end{align}
By a parallel argument as above, we obtain the following.

\begin{prop}
Let $\xi \in \breve{X}^{k,+}$. Retain the notation as in
\eqref{irred:verma2}. Then
\begin{align*}
[{L}_{({\bf b},{1^k})}(\xi)] =\sum_{\eta\in
\breve{X}^{k,+}}\breve{a}_{\eta\xi}
[\breve{\mc{M}}_{\bf b}^k(\eta)].
\end{align*}
\end{prop}

We now proceed to compare the characters of the tilting modules in a parabolic BGG
category with those in a full BGG category. In the full BGG
categories $\OO^{m+k|n}_{({\bf b},{0^k})}$ and $\OO^{m|n+k}_{({\bf
b},{1^k})}$, we write the following.
\begin{align}
[{T}_{({\bf b},{0^k})}(\la)]
 &=\sum_{\mu\in X^k}c_{\mu\la}[M_{({\bf b},{0^k})}(\mu)],\quad \text{ for }
c_{\mu\la}\in\Z, \; \la \in X^{k,+};
  \label{eq:TM1} \\
[{T}_{({\bf b},{1^k})}(\xi)]
 &=\sum_{\eta\in \breve{X}^k}\breve{c}_{\eta\xi} [M_{({\bf
b},{1^k})}(\eta)],\quad \text{ for }
\breve{c}_{\eta\xi}\in\Z, \; \xi \in \breve X^{k,+}.
 \label{eq:TM2}
\end{align}
Recall the tilting modules ${\mc T}_{\bf b}^k(\la)$, for $\la \in X^{k,+}$,
in  the parabolic category $\OO_{\bf
b}^{\ul k}$, and the tilting modules $\breve{\mc T}_{\bf b}^k(\xi)$ in  $\breve\OO_{\bf
b}^{\ul k}$, for $\xi \in \breve X^{k,+}$.

\begin{prop}\label{prop:tilting:char:para}
\begin{enumerate}
\item
Let $\la\in X^{k,+}$. Retain the notation in \eqref{eq:TM1}, and
write
\begin{align*}
[{\mc T}_{\bf b}^k(\la)] =\sum_{\nu\in
X^{k,+}}d_{\nu\la}[\mc{M}_{\bf b}^k(\nu)].
\end{align*}
Then
$%\begin{equation}  \label{aux:104}
d_{\nu\la}=\sum_{\tau\in\mf{S}_k}(-1)^{\ell(\tau
w_0)}c_{\tau\cdot\nu,w_0\cdot\la}.
 $%\end{equation}

\item Let $\xi\in \breve{X}^{k,+}$. Retain the notation in
\eqref{eq:TM2}, and write
\begin{align*}
[\breve{\mc T}_{\bf b}^k(\xi)] =\sum_{\eta\in
\breve{X}^{k,+}}\breve{d}_{\eta\xi}
[\breve{\mc{M}}_{\bf b}^k(\eta)].
\end{align*}
Then
$\breve{d}_{\eta\xi}=\sum_{\tau\in\mf{S}_k}(-1)^{\ell(\tau
w_0)}\breve{c}_{\tau\cdot\eta,w_0\cdot\xi}$.
\end{enumerate}
\end{prop}

\begin{proof}
We shall only prove (1), as (2) is analogous.

Set $\rho=\rho_{({\bf b},0^k)}$,
$\rho_{\mf{u}}=\hf\sum_{\alpha\in {\Phi}(\mf{u}^k_{\bf b})}(-1)^{|\alpha|}\alpha$, and $\rho_{\mf{k}}=\rho-\rho_{\mf u}$,
where $|\alpha|$ denotes the parity of the root $\alpha$  and ${\Phi}(\mf{u}^k_{\bf b})$ denotes the roots of the radical subalgebra $\mf{u}^k_{\bf b}$.
Furthermore, let $w_0=w_0^{(k)}$ be the longest element in $\mf
S_k$. Applying \cite[Theorem 6.7]{So2} and its super generalization
\cite[Theorem 6.4]{Br2} to the category $\OO^{m+k|n}_{({\bf
b},{0^k})}$, we compute, for $\la,\mu\in X^{k,+}$,
\begin{align*}
\left[\mc M^k_{{\bf b}}(\la):L_{({\bf b},0^k)}(\mu)\right]
&=\sum_{\tau\in\mf{S}_k} (-1)^{\ell(\tau)}
\left[M_{({\bf b},0^k)}(\tau\cdot\la):L_{({\bf b},0^k)}(\mu)\right]
 \\
&= \sum_{\tau\in\mf{S}_k} (-1)^{\ell(\tau)} \left(T_{({\bf
b},0^k)}(-2\rho-\mu):M_{({\bf b},0^k)}(-2\rho-\tau\cdot\la)\right)
 \\
&=\sum_{\tau\in\mf{S}_k} (-1)^{\ell(\tau)}c_{-2\rho-\tau\cdot\la,-2\rho-\mu}.
\end{align*}
On the other hand,  by applying \cite[Theorem 6.4]{Br2} to the
category $\OO_{\bf b}^{\ul k}$, we also have
\begin{align*}
\left[\mc M^k_{{\bf b}}(\la):L_{({\bf b},0^k)}(\mu)\right] &=
\left(\mc T^k_{\bf b}(-2\rho_{\mf u}-w_0\mu):
 \mc M^k_{{\bf b}}(-2\rho_{\mf u}-w_0\la)\right)
 \\
&=d_{-2\rho_{\mf u}-w_0\la,-2\rho_{\mf u}-w_0\mu}.
\end{align*}
A comparison of the above two identities and replacing $\tau$ by
$\tau w_0$ give us
\begin{align}  \label{d=c}
d_{-2\rho_{\mf u}-w_0\la,-2\rho_{\mf
u}-w_0\mu}=\sum_{\tau\in\mf{S}_k} (-1)^{\ell(\tau
w_0)}c_{-2\rho-\tau w_0\cdot\la,-2\rho-\mu}.
\end{align}
Set $\nu=-2\rho_{\mf u}-w_0\la$ and $\eta=-2\rho_{\mf
u}-w_0\mu$. We shall use repeatedly the following simple identities:
$$
\rho =\rho_{\mf{k}}+\rho_{\mf u},
 \quad \tau \rho_{\mf u} =w_0 \rho_{\mf u} =\rho_{\mf u},
 \quad w_0 \rho_{\mf{k}}=-\rho_{\mf{k}}.
$$
Now we compute
\begin{align*}
-2\rho-\tau w_0\cdot\la
 &=-\rho-\tau w_0(\la+\rho)
 =-\rho_{\mf u}-\tau w_0\la-\tau w_0\rho_{\mf k}-\rho
 \\
&=\tau(-2\rho_{\mf u}-w_0\la +\rho)-\rho = \tau\cdot\nu.
\end{align*}
Also, $w_0\cdot\eta=w_0(-2\rho_{\mf u}-w_0\mu+\rho)-\rho =
-\rho_{\mf u}-\mu-\rho_{\mf k}-\rho= -2\rho-\mu$. From these
computations, we see that \eqref{d=c} rewritten in terms
of $\nu$ and $\eta$ (then followed by a change of notation
$\eta$ to $\la$) is exactly what we want to prove in (1).
\end{proof}

\subsection{BKL for adjacent Borel subalgebras}

The following theorem is a key step in our proof of the BKL
Conjecture~\ref{BKL:conj:tilt}.

\begin{thm}\label{thm:BKL:equiv:borel}
Let ${\bf b}$ and ${\bf b}'$ be two adjacent $0^m1^n$-sequences. The
BKL Conjecture~\ref{BKL:conj:tilt} holds for ${\bf b}$ if and only
it holds for ${{\bf b}'}$.
\end{thm}

\begin{proof}
It suffices to prove that the validity of BKL Conjecture~
\ref{BKL:conj:tilt} for ${\bf b}$ implies its validity for ${{\bf
b}'}$. We shall follow the notations in \S\ref{sec:CB2} to denote
${\bf b}=({\bf b}^1,{0},{1},{\bf b}^2)$ and ${\bf b}':=({\bf
b}^1,{1},{0},{\bf b}^2)$. (The proof below goes through similarly
when switching $\bf b$ and $\bf b'$.)

(1) We first prove this for Part (1) of BKL
Conjecture~\ref{BKL:conj:tilt}. The idea of the proof is to switch
to the bases in notation $N$'s instead of the $M$'s for more
effective comparisons with the $L$'s, based on the results of
\S\ref{sec:CB2}, \S\ref{sec:adjCB} and Section~\ref{sec:repn:prep}.

Recall $\psi_{\bf b} ([M_{\bf b}(\la)])  = M_{f^{\bf b}_\la}(1)$,
for all $\la$. By comparing \eqref {eq:MNf} and \eqref{eq:MNla}, we
have
\begin{equation}  \label{eq:psiN}
\psi_{\bf b} ([N_{\bf b}(\mu)])  = N_{f^{\bf b}_\mu}(1), \qquad
\text{ for } \mu \in \wtl.
\end{equation}
By the assumption on the validity of the BKL
Conjecture~\ref{BKL:conj:tilt}(1) for $\bf b$, we have
\begin{equation}  \label{eq:psiL}
\psi_{\bf b} ([L_{\bf b}(\la)])  = L^{\bf b}_{f^{\bf b}_\la}(1),
\qquad \text{ for } \la \in \wtl.
\end{equation}
By \eqref{dual:canonical:in:N} and Proposition~\ref{LTsame}, we have
\begin{align}  \label{eq:LNbCB}
L^{\bf b}_f=\sum_{g}\check{\ell}_{gf}(q) N_g.
\end{align}
Since $\psi_{\bf b}$ is an isomorphism, it follows by
\eqref{eq:psiN}, \eqref{eq:psiL} and \eqref{eq:LNbCB} that
\begin{align}  \label{eq:LNb}
[L_{\bf b}(\la)]=\sum_{\mu}\check{\ell}_{f^{\bf b}_\mu f^{\bf
b}_\la}(1)\ [N_{\bf b}(\mu)].
\end{align}

Lemma~\ref{lem:irred:borels} states that $L_{\bf b}(\la) =L_{\bf
b'}(\la^{\mathbb L})$, \eqref{NUadj} states that $\text{ch}N_{\bf
b}(\mu)=\text{ch}N_{\bf b'}(\mu^{\mathbb L})$, while
Corollary~\ref{cor:llsame} states that $\check{\ell}_{gf}(q)
=\check{\ell}_{g^{\mathbb L}f^{\mathbb L}}'(q).$ These three
identities together with \eqref{eq:LNb} imply that
\begin{align}  \label{eq:LNb'}
[L_{\bf b'}(\la^{\mathbb L})]=\sum_{\mu}\check{\ell}_{f^{\bf
b'}_{\mu^{\mathbb L}} f^{\bf b'}_{\la^{\mathbb L}}}'(1)\ [N_{\bf
b'}(\mu^{\mathbb L})],
\end{align}
where we have identified $f^{\bf b'}_{\mu^{\mathbb L}} = (f^{\bf
b}_{\mu})^{\mathbb L}$ for all $\mu$ by the definitions of
\eqref{eq:fL} and \eqref{eq:laL}.

By Proposition~\ref{L'T'same} and \eqref{DCB2'}, we have
\begin{align}\label{DCB3'}
L_f^{\bf b'}=\sum_{g}%\preceq_{\bf b, \bf b'}^*f}
\check{\ell}_{gf}'(q) N_g'.
\end{align}
By definition, $\psi_{\bf b'}([M_{\bf b'}(\la)])  = M_{f^{\bf
b'}_\la}(1)$. By straightforward $\bf b'$-counterparts of \eqref
{eq:MNf} and \eqref{eq:MNla}, we have the following $\bf
b'$-counterpart of \eqref{eq:psiN}:
\begin{equation}  \label{eq:psiN'}
\psi_{\bf b'} ([N_{\bf b'}(\mu^{\mathbb L})])  = N_{f^{\bf
b'}_{\mu^{\mathbb L}}}'(1), \qquad \text{ for } \mu \in \wtl.
\end{equation}
Now applying $\psi_{\bf b'}$ to both sides of \eqref{eq:LNb'} and
using \eqref{eq:psiN'}, we obtain by a comparison with \eqref{DCB3'}
that $\psi_{\bf b'} ([L_{\bf b'}(\la^{\mathbb L})]) =L^{\bf
b'}_{\la^{\mathbb L}}(1)$. This proves the BKL
Conjecture~\ref{BKL:conj:tilt}(1) for $\bf b'$.

(2) We employ a similar strategy to prove that the validity of BKL
Conjecture~\ref{BKL:conj:tilt}(2) for ${\bf b}$ implies its validity
for ${{\bf b}'}$. The idea of the proof is to switch to the bases in
notation $U$'s instead of the $M$'s for more effective comparisons
with the $T$'s, based on the results of \S\ref{sec:CB2},
\S\ref{sec:adjCB} and Section~\ref{sec:repn:prep}.

By comparing \eqref{eq:UMf} and \eqref{eq:UMla}, we have
\begin{equation}  \label{eq:psiU}
\psi_{\bf b} ([U_{\bf b}(\mu)])  = U_{f^{\bf b}_\mu}(1), \qquad
\text{ for } \mu \in \wtl.
\end{equation}
By the assumption on the validity of the BKL
Conjecture~\ref{BKL:conj:tilt}(2) for $\bf b$, we have
\begin{equation}  \label{eq:psiT}
\psi_{\bf b} ([T_{\bf b}(\la)])  = T^{\bf b}_{f^{\bf b}_\la}(1),
\qquad \text{ for } \la \in \wtl.
\end{equation}
By \eqref{dual:canonical:in:U} and Proposition~\ref{LTsame}, we have
\begin{align}  \label{eq:TUbCB}
T^{\bf b}_f=\sum_{g} \check{t}_{gf}(q) U_g.
\end{align}
Since $\psi_{\bf b}$ is an isomorphism, it follows by
\eqref{eq:psiU}, \eqref{eq:psiT} and \eqref{eq:TUbCB} that
\begin{align}  \label{eq:TUb2}
[T_{\bf b}(\la)]=\sum_{\mu}\check{t}_{f^{\bf b}_\mu f^{\bf
b}_\la}(1)\ [U_{\bf b}(\mu)].
\end{align}

Theorem~\ref{thm:tilt:relation} states that $T_{\bf b}(\la)= T_{\bf
b'}(\la^{\mathbb U}),$ \eqref{NUadj} states that $\text{ch}U_{\bf
b}(\mu)=\text{ch}U_{\bf b'}(\mu^{\mathbb U})$, while
Corollary~\ref{cor:samett} states that $\check{t}_{gf}(q)
=\check{t}_{g^{\mathbb U}f^{\mathbb U}}'(q).$ These three identities
together with \eqref{eq:TUb2} imply that
\begin{align}  \label{eq:TUb2'}
[T_{\bf b'}(\la^{\mathbb U})]=\sum_{\mu}\check{t}_{f^{\bf
b'}_{\mu^{\mathbb U}} f^{\bf b'}_{\la^{\mathbb U}}}'(1)\ [U_{\bf
b'}(\mu^{\mathbb U})],
\end{align}
where we have identified $f^{\bf b'}_{\mu^{\mathbb U}} = (f^{\bf
b}_{\mu})^{\mathbb U}$ for all $\mu$ by the definitions of
\eqref{eq:fU} and \eqref{eq:laU}.

By Proposition~\ref{L'T'same} and \eqref{CB2'}, we have
\begin{align}\label{TUCB3'}
T_f^{\bf b'}=\sum_{g}  \check{t}_{gf}'(q) U_g'.
\end{align}
By definition, $\psi_{\bf b'}([M_{\bf b'}(\la)])  = M_{f^{\bf
b'}_\la}(1)$. We easily have the following $\bf b'$-counterpart of
\eqref{eq:psiU}:
\begin{equation}  \label{eq:psiU'}
\psi_{\bf b'} ([U_{\bf b'}(\mu^{\mathbb U})])  = U_{f^{\bf
b'}_{\mu^{\mathbb U}}}'(1), \qquad \text{ for } \mu \in \wtl.
\end{equation}
Now applying $\psi_{\bf b'}$ to both sides of \eqref{eq:TUb2'} and
using \eqref{eq:psiU'}, we obtain by a comparison with
\eqref{TUCB3'} that $\psi_{\bf b'} ([T_{\bf b'}(\la^{\mathbb U})])
=T^{\bf b'}_{\la^{\mathbb U}}(1)$. This proves BKL
Conjecture~\ref{BKL:conj:tilt}(2) for $\bf b'$.

The proof of the theorem is completed.
\end{proof}

\begin{rem}  \label{rem:allb}
Any two $0^m1^n$-sequences are connected via a sequence of
$0^m1^n$-sequences such that any two neighboring sequences are adjacent.
Therefore, Theorem~\ref{thm:BKL:equiv:borel} is equivalent to saying
that the validity of the BKL Conjecture~\ref{BKL:conj:tilt} for {\em
one} particular $0^m1^n$-sequence implies its validity for {\em all}
$0^m1^n$-sequences.
\end{rem}

\subsection{The proof}

The following theorem will provide the inductive step for proving
the BKL conjecture. We follow the outline of steps \eqref{ind:para}--\eqref{ind:trunc} in the Introduction.

\begin{thm}  \label{th:indstep}
Let $n\ge 0$ be fixed. The validity of the BKL conjecture for all
$0^m1^n$-sequences for every $m\ge 0$ implies the validity of the
BKL conjecture for some $0^m1^{n+1}$-sequence for every $m$.
\end{thm}

\begin{proof}
In this proof, we will regard $m$ and $n$ as fixed, and prove the
following reformulation: the validity of the BKL conjecture for all
$0^{m+k}1^n$-sequences for every $k\ge 0$ implies the validity of
the BKL conjecture for one particular $0^m1^{n+1}$-sequence.

Take an arbitrary $0^m1^n$-sequence ${\bf b}$, and form the
$0^{m+k}1^n$-sequence $({\bf b}, 0^k)$. The assumption above that
the BKL conjecture for all $0^{m+k}1^n$-sequences for every $k$
holds can be stated more precisely as follows.

(A) The isomorphism
\begin{align*}
\big{[}\big{[}\OO^{m+k|n,\Delta}_{({\bf b},0^k)}\big{]}\big{]}
\longrightarrow \widehat{\mathbb T}^{({\bf b},0^k)}_\Z,
 \qquad
[M_{({\bf b},0^k)}(\la)]\mapsto M^{({\bf b},0^k)}_{f^{({\bf
b},0^k)}_{\la}}(1),
\end{align*}
sends $[L_{({\bf b},0^k)}(\la)]$ to $L^{({\bf b},0^k)}_{f^{({\bf
b},0^k)}_{\la}}(1)$, and $[T_{({\bf b},0^k)}(\la)]$ to $T^{({\bf
b},0^k)}_{f^{({\bf b},0^k)}_{\la}}(1)$, for all $\la\in X(m+k|n).$

We proceed in 4 steps (i)-(iv) below, starting from (A). Note that
${\mathbb T}^{({\bf b},0^k)} ={\mathbb T}^{\bf b} \otimes
\VV^{\otimes k}$.

(i) \underline{Pass to a parabolic version.}

The isomorphism
\begin{align*}
\big{[}\big{[}\OO^{\ul k,\Delta}_{{\bf b}}\big{]}\big{]}
\longrightarrow {\mathbb T}^{\bf b}_\Z\wotimes\wedge^k\mathbb V_\Z,
 \qquad
[\mc{M}^k_{\bf b}(\la)]\mapsto M^{{\bf b},0}_{f^{{\bf
b}0}_{\la}}(1),
\end{align*}
sends $[\mc{L}_{\bf b}^k(\la)]$ to $L^{{\bf b},0}_{f^{{\bf
b}0}_{\la}}(1)$, and $[\mc{T}_{\bf b}^k(\la)])$ to $T^{{\bf
b},0}_{f^{{\bf b}0}_{\la}}(1)$ for $\la\in X^{k,+}.$

Indeed (i) follows from (A), by Propositions \ref{thm:aux:11} and
\ref{thm:aux:12} (which relate the BKL polynomials from the setting
of (A) to the current $q$-wedge setting), as well as
Propositions~\ref{prop:irred:char:para} and
\ref{prop:tilting:char:para} (which relate the simple and tilting
modules from the setting of (A) to the current parabolic setting).

(ii) \underline{Pass from finite $k$ to $\infty$.}

The isomorphism
\begin{align*}
\big{[}\big{[}\OO^{\Delta}_{{\bf b}}\big{]}\big{]} \longrightarrow
{\mathbb T}^{\bf b}_\Z\wotimes\wedge^\infty\mathbb V_\Z,
 \qquad
[\mc{M}_{\bf b}(\la)]\mapsto M^{{\bf b},0}_{f^{{\bf b}0}_{\la}}(1),
\end{align*}
sends $[\mc{L}_{\bf b}(\la)]$ to $L^{{\bf b},0}_{f^{{\bf
b}0}_{\la}}(1)$, and $[\mc{T}_{\bf b}(\la)])$ to $T^{{\bf
b},0}_{f^{{\bf b}0}_{\la}}(1)$ for $\la\in X^{+}.$

Indeed (ii) follows from (i) by Proposition~ \ref{prop:trunc:fock}
and Proposition~\ref{prop:trunc:ML}.

(iii) \underline{Super duality.}

Note that (i) and (ii) are exactly the statements formulated in
Theorem~\ref{thm:classical:KL}, now valid for a general $\bf b$.
Hence, the assumption in Theorem~\ref{th:KL:SD} is now valid. By
Theorem~\ref{th:KL:SD}, the isomorphism
$$
\breve\Psi: \big{[}\big{[}\breve{\OO}^{\Delta}_{{\bf
b}}\big{]}\big{]} \longrightarrow {\mathbb T}^{\bf
b}_\Z\wotimes\wedge^\infty\WW_\Z,
 \qquad
[\breve{\mc{M}}_{\bf b}(\la)]\mapsto M^{{\bf b},1}_{f^{{\bf
b}1}_{\la}}(1),
$$
sends $[\breve{\mc{L}}_{\bf b}(\la)]$ to $L^{{\bf b},1}_{f^{{\bf
b}1}_{\la}}(1)$, and $[\breve{\mc{T}}_{\bf b}(\la)]$ to $T^{{\bf
b},1}_{f^{{\bf b}1}_{\la}}(1)$, for $\la\in \breve X^{+}.$

(iv) \underline{Truncation.}

Now let $k\in\N$. Recall the category $\breve\OO^{\ul
k,\Delta}_{{\bf b}}$ of $\gl(m|n+k)$-modules from
\S\ref{sec:finiterank}, and recall the bijection $X^{k,+}
\rightarrow \Z^{m+n}\times\Z^k_+$ by sending $\la$ to $f_\la^{{\bf
b}1}$ from \eqref{eq:2bij}. Consider the isomorphism
\begin{align*}
\breve{\Psi}^k: \big[\big[\breve\OO^{\ul k,\Delta}_{{\bf
b}}\big]\big] \longrightarrow {\mathbb T}^{\bf
b}_\Z\widehat{\otimes}\wedge^{k}\WW_\Z,
 \qquad
[\breve{\mc{M}}^k_{\bf b}(\la)]\mapsto M^{{\bf b},1}_{f^{{\bf
b}1}_{\la}}(1),
\end{align*}
where $\big[\big[\breve\OO^{\ul k,\Delta}_{{\bf b}}\big]\big]$ is a
suitable completion of $\big[\breve\OO^{\ul k,\Delta}_{{\bf
b}}\big]$ as before. We have the following.

{\bf Claim.} For $\la\in\breve X^{k,+}$, we have
\begin{equation} \label{psikLT}
\breve{\Psi}^k([\breve{\mc{L}}_{\bf b}^k(\la)])=L^{{\bf b},1}_{f^{{\bf
b}1}_{\la}}(1),
 \quad
\breve{\Psi}^k([\breve{\mc{T}}_{\bf b}^k(\la)]) =T^{{\bf b},1}_{f^{{\bf
b}1}_{\la}}(1). %\qquad \text{ for }\la\in \breve X^{k,+}_{\ge},
\end{equation}

Let us specialize $k=1$. In this case, the parabolic category
$\breve\OO^{\ul 1}_{{\bf b}}$ is exactly the full BGG category
$\OO^{m|n+1}$, and $\breve X^{1,+}=X(m|n+1)$. Hence assuming the
claim,
 we have verified the BKL conjecture for the special
$0^m1^{n+1}$-sequence ${({\bf b},1)}$.

It remains to prove \eqref{psikLT} for $\la\in\breve
X^{k,+}$, using the truncation maps and truncation functors. We have
the following commutative diagram by a direct computation using the
basis $\{[\mc{M}_{\bf b}(\la)]\}$ for
$\big{[}\big{[}\breve\OO^{\Delta}_{{\bf b}}\big{]}\big{]}$:
\begin{eqnarray}  \label{CD:trunc}
\begin{CD}
\big{[}\big{[}\breve\OO^{\Delta}_{{\bf b}}\big{]}\big{]}
 @>\breve\Psi>>
{\mathbb T}^{\bf b}_\Z\wotimes\wedge^\infty \WW_\Z  \\
 @V\breve{\mf{tr}}VV @V \texttt{Tr} VV \\
\big[\big[\breve\OO^{\ul k,\Delta}_{{\bf b}}\big]\big]
 @>\breve{\Psi}^k>>
{\mathbb T}^{\bf b}_\Z\widehat{\otimes}\wedge^{k}\WW_\Z
  \end{CD}
\end{eqnarray}
It follows from (iii), \eqref{CD:trunc},
Propositions~\ref{prop:can:trunc} and \ref{prop:trunc:ML} that
\eqref{psikLT}  holds for those $\la\in\breve X^{k,+}$ satisfying
the condition $\big\langle\la,e^{({\bf
b},{1^k})}_{m+n+k,m+n+k}\big\rangle\ge 0$. This condition arises in
the parametrization set for the standard basis of the image of
$\breve{\mf{tr}}$ (which is not surjective); see
Remark~\ref{rem:notsurj}.

We have the following commutative diagram:
\begin{eqnarray}  \label{CD:shift}
\begin{CD}
\big[\big[\breve\OO^{\ul k,\Delta}_{{\bf b}}\big]\big]
 @>\breve{\Psi}^k>>
{\mathbb T}^{\bf b}_\Z\widehat{\otimes}\wedge^{k}\WW_\Z
  \\
 @V\otimes{\rm Str}VV @V \texttt{sh}VV \\
\big[\big[\breve\OO^{\ul k,\Delta}_{{\bf b}}\big]\big]
 @>\breve{\Psi}^k>>
{\mathbb T}^{\bf b}_\Z\widehat{\otimes}\wedge^{k}\WW_\Z
  \end{CD}
\end{eqnarray}
Here $\otimes{\rm Str}$ denotes the map induced from tensoring with
the $1$-dimensional supertrace representation (see \eqref{strace}
with $n$ therein replaced by $n+k$), and $\texttt{sh}$ denotes the
$\Z$-linear shift map which sends $M^{{\bf b},1}_{f}$ to $M^{{\bf
b},1}_{f+\texttt{1}_{m|(n+k)}}$, for each $f$; see \eqref{eq:1mn}
for notation $\texttt{1}_{m|(n+k)}$.

Just as in the proof of Proposition~\ref{prop:shift:can:p} where a
similar shift map has been used, $\texttt{sh}$ also commutes with
the bar map, and then
\begin{equation}  \label{shTL}
\texttt{sh}(T^{{\bf b},1}_{f})=T^{{\bf
b},1}_{f+\texttt{1}_{m|(n+k)}}, \qquad \texttt{sh}(L^{{\bf
b},1}_{f})=L^{{\bf b},1}_{f+\texttt{1}_{m|(n+k)}}, \quad \forall f.
\end{equation}
On the other hand, it is clear that
\begin{equation}  \label{tensorStr}
\breve{\mc{L}}_{\bf b}^k(\la) \otimes{\rm Str} =\breve{\mc{L}}_{\bf b}^k(\la+{\rm
Str}),
 \qquad
\breve{\mc{T}}_{\bf b}^k(\la) \otimes{\rm Str} =\breve{\mc{T}}_{\bf b}^k(\la+{\rm
Str}),
 \quad
\forall \la\in \breve X^{k,+}.
\end{equation}

It follows by \eqref{shTL}, \eqref{tensorStr} and the commutative
diagram \eqref{CD:shift} that \eqref{psikLT} holds for $\la\in\breve
X^{k,+}$ satisfying $\langle\la,e^{({\bf
b},{1^k})}_{m+n+k,m+n+k}\rangle\ge -1$. Repeatedly using
\eqref{CD:shift}, \eqref{shTL} and \eqref{tensorStr}, we conclude
that \eqref{psikLT} holds for all $\la\in\breve X^{k,+}$.

This proves the claim, and hence completes the proof of the theorem.
\end{proof}

Now we are ready to prove the main result of this paper.

\begin{thm} \label{th:BKL}
The BKL Conjecture ~\ref{BKL:conj:tilt} holds for an arbitrary
$0^m1^n$-sequence.
\end{thm}

\begin{proof}
We shall proceed by induction on $n$. The base case when $n=0$ is
Theorem~\ref{thm:classical:KL} (with $k=0$), which is a Fock space
reformulation of the classical Kazhdan-Lusztig conjecture. By
induction hypothesis, for a given $n$, the BKL conjecture holds for
all ${0^m1^{n}}$-sequences and for every $m$. By
Theorem~\ref{th:indstep}, the BKL conjecture holds for one
particular ${0^m1^{n+1}}$-sequence. Now by
Theorem~\ref{thm:BKL:equiv:borel} and Remark~\ref{rem:allb}, the BKL
conjecture holds for all ${0^m1^{n+1}}$-sequences. The induction is
completed.
\end{proof}

\begin{rem}
It follows from Theorem~\ref{th:BKL} that all the parabolic versions
with even standard Levi subalgebras of the BKL conjecture hold, via
similar comparisons as formulated in \S\ref{sec:comp:para:tensor}
and \S\ref{char:comp:para:fin}. Note however that our proof of
Theorem~\ref{th:BKL} uses in an essential way a distinguished
parabolic case which was established earlier in \cite{CL} via the
approach of super duality \cite{CWZ, CW}.
\end{rem}

\end{document}